\documentclass[12pt]{article}
\usepackage{amssymb,amsmath,amsthm}
\usepackage{lineno} 
\usepackage{color}

\def\gg{\mathfrak g}
\def\o{\otimes}
\def\+{\oplus}
\newtheorem{theorem}{Theorem}[section]
\newtheorem{definition}[theorem]{Definition}
\newtheorem{defn}[theorem]{Definition}
\newtheorem{corollary}[theorem]{Corollary}

\newtheorem{example}[theorem]{Example}
\newtheorem{exa}[theorem]{Example}
\newtheorem{lemma}[theorem]{Lemma}

\newtheorem{remark}[theorem]{Remark}
\newtheorem{rmk}[theorem]{Remark}

\newtheorem{main}[theorem]{Main Theorem}
\newtheorem{mainht}[theorem]{Main Homotopy Theorem}
\newtheorem{mainhtp}[theorem]{Main Homotopy Perturbation Theorem}
\newtheorem{crl*}[theorem]{Corollary}

\newcommand{\dual}{\# }
\newcommand{\dgL}{dg Lie algebra\ }
\newcommand{\dgLL}{dg Lie algebra}
\newcommand{\dgLs}{dg Lie algebras\ }

\newcommand{\Loo}{$L_\infty$}

\newcommand{\spprox}{\cong }

\newcommand{\mathbold}{\mathbf}
\newcommand{\Q}{\mathbold Q}
\newcommand{\HH}{\mathcal H}     

\newtheorem{thm}[theorem]{Theorem}
\newtheorem{lm}[theorem]{Lemma}

\newtheorem{crl}[theorem]{Corollary}
\newtheorem{dfn}[theorem]{Definition}
\newtheorem{dfneg}[theorem]{Definition and Example}

\newtheorem{ex}{Example}[section]

\def\del{\partial}
\def\v{\vskip2ex}

\def\tensor{\otimes}
\begin{document}

\title {Deformation Theory and Rational Homotopy Type}
\author
{Mike Schlessinger and Jim Stasheff }
\linenumbers
\maketitle
\begin{abstract}
We regard the classification of rational homotopy types as a
problem in algebraic deformation theory:  any space with given cohomology is a
perturbation, or deformation, of the \lq\lq formal\rq\rq\,  space with that
cohomology.  The classifying space is then a \lq\lq moduli\rq\rq\,  space --- 
a certain quotient of an algebraic variety of perturbations. The
description we give of this moduli space links it with corresponding structures
 in homotopy theory, especially
the classification of fibres spaces $F \to E \overset{p}{ \rightarrow} B$
with fixed fibre $F$ in terms of homotopy classes of maps of the base
$B$ into a classifying space constructed from $Aut(F)$, the monoid of
homotopy equivalences of $F$ to itself.  We adopt  the
philosophy, later promoted by Deligne  in response  to
Goldman and Millson, that any problem in deformation theory is ``controlled'' by
a differential graded Lie algebra, unique up to homology
equivalence (quasi-isomorphism) of dg Lie algebras.  Here we extend this philosophy 
further to control by $L_\infty$ -algebras.

\end{abstract}
\centerline{\emph{In memory of Dan Quillen who established the foundation on which this work rests.}}
\tableofcontents


\vfill\eject

\section{Introduction}
\label{introduction}
In this paper, we regard the classification of rational homotopy types as a
problem in algebraic deformation theory:  any space with given cohomology is a
perturbation, or deformation, of the \lq\lq formal\rq\rq\,  space with that
cohomology.  The classifying space is then a \lq\lq moduli\rq\rq\,  space --- 
a certain quotient of an algebraic variety of perturbations. The
description we give of this moduli space links it with others which occur in
algebra and topology, for example, the moduli spaces of algebras or complex
manifolds. On the other hand, our dual vision emphasizes the analogy
with corresponding structures in homotopy theory, especially
the classification of fibres spaces $F \to E \overset{p}{ \rightarrow} B$
with fixed fibre $F$ in terms of homotopy classes of maps of the base
$B$ into a classifying space constructed from $Aut(F)$, the monoid of
homotopy equivalences of $F$ to itself.
In particular, the moduli space of rational homotopy types with fixed
cohomology algebra can be identified with the space of 
``path components'' of a certain differential graded coalgebra.

Although the majority of this paper is concerned with constructing and verifying
the relevant machinery, the final sections are devoted to a variety of examples,
which should be accessible without much of the machinery and might provide 
motivation for reading the technical details of the earlier
sections.

Portions of our work first appeared in print in \cite{Sjds,SS} and then in `samizdat' versions over the intervening 
decades (!), partly due to some consequences of the mixture of languages.
Some of those versions have worked there way into work of other researchers; we have tried to maintain 
much of the flavor of our early work while taking advantage of progress made by others.

Crucially, throughout this paper,  the ground field is the rational numbers, 
$\Bbb Q$
(characteristic 0 is really the relevant algebraic fact),
although parts of it make sense even over the integers.
\subsection{Background}
Rational homotopy theory regards {\bf rational homotopy equivalence}
of two simply connected spaces  as the equivalence relation generated
by the existence of a map $f : X \to Y$ inducing an
isomorphism $f^* : H^*(Y;\Q) \to H^*(X,\Q)$.  Here we are much closer to a
complete classification than in the ordinary (integral) homotopy category.  An
obvious invariant is the cohomology algebra $H^*(X;\Q)$.  Halperin and Stasheff
\cite{halperin-stasheff} showed that all simply connected spaces $X$ 
with fixed cohomology algebra $\HH$ of finite type over $\mathbold Q$
can be described (up to rational homotopy type) as follows:  
({\it Henceforth `space' shall mean `simply connected space of finite type'
unless otherwise specified.})

Resolve $\HH$ by a {\bf d}(ifferential)
{\bf g}(raded) {\bf c}(ommutative)  {\bf a}(lgebra) $(S Z,d)$
 which is connected and  free as a graded 
commutative algebra with a map $(S Z,d)\to \HH$ of dgcas inducing
$H(S Z,d) \simeq \HH$. Here $S$ denotes graded symmetric algebra.
(See section \ref{T-J} for details, especially in re: the various gradings involved.)
The notation $\Lambda$ instead of $S$ is often used
 within rational homotopy theory, where it   is a historical accident
 derived from de Rham theory.
Let $A^*(X)$ denote
a differential graded commutative algebra of ``differential forms
over the rationals'' for the space $X$,
e.g. Sullivan's version of the deRham complex \cite {sullivan:inf,
bousfield-guggenheim}).
Given an isomorphism $i : \HH \overset{\simeq }{\rightarrow} H^*(X)$, there is
a perturbation $p$ (a derivation of $S Z$ of degree $1$ which lowers
resolution degree by at least $2$ such that $(d+p)^2 = 0$) and a map of dga's
$(S Z,d+p)\to A^*(X)$ inducing an isomorphism of rational cohomology.
If $X$ and $Y$ have the same rational homotopy type, the perturbations $p_X$ and
$p_Y$ must be related in a certain way, spelled out in \S 2.2.  
This is one of several ways (cf. 
\cite {lemaire:tronq, felix:Bull.Soc.Math.France80}) it can be seen that
\begin{main}

 For fixed $\HH$, the set of homotopy types of pairs $(X,i : \HH
\simeq H(X))$ can be represented as the quotient $V/G$ of a (perhaps infinite
dimensional) conical rational algebraic variety $V$ modulo a pro-unipotent (algebraic) group
action $G$.
\end{main}

\begin{crl*}
 The set of rational homotopy types with fixed cohomology
$\HH$ can be represented as a quotient ${Aut}\, \HH\backslash V/G.$
\end{crl*}
\subsection{Control by DGLAs}
The variety and the group can be expressed in the following terms:  let
${Der}\ S Z$ denote the graded Lie algebra of graded
derivations of $S Z$, which is itself a \dgL (see Definition \ref{D:dgl})
with the commutator bracket and the differential
induced by the internal differential on $S Z$. 
Let $L\subset {Der}\, S Z$ be the sub--Lie algebra of 
derivations that decrease the sum of the total degree plus the resolution 
degree. 
The variety $V \subset L^1$ is precisely $V = \lbrace p \in L^1
\vert (d+p)^2 = 0\rbrace $.  In fact, $V$ is the cone on a
projective variety (of possibly infinite dimension) \S\ref{T:cone}
The pro-unipotent group $G$ is exp$ \, L$, which acts via
the adjoint action of $L$ on $d+p$.

We said above that we regard our problem as one of deformation theory
in the (homological) algebra sense.  A commutative
algebra $\HH$ has a Tate resolution which is an {\bf almost free}
 commutative dga $S Z$, that is, free as graded commutative algebra,
ignoring the differential.
A deformation of $\HH$ corresponds to a change in differential $d \to d+p$ on
$S Z$.  Instead of $L \subset {Der}\, S Z$ as above, the 
sub-dg Lie
algebra  $\bar L \subset \, {Der}\, S Z$ of nonpositive resolution
degree is used. 


As far as we know, deformation theory arose with work on families of complex structures.   Early on, these were expressed in terms of a moduli space \cite{riemann, teichmuller}.
This  began with Riemann, who first introduced the term "moduli".  He proved that the number of moduli of a surface of genus $0$ was $0$.  $1$ for genus $1$ and $3g-3$ for $g > 1$.  These are the same as the numbers of quadratic differentials on the surface, which was perhaps the impetus for Teichmueller's work identifying these as "infinitesimal deformations".  Teichmueller was probably the first to introduce the concept of an infinitesimal deformation of a complex manifold, but in his form it was restricted to Riemann surfaces and could not be extended to higher dimensions.
Froelicher and Nijenhuis \cite {froh-nij} gave the appropriate definition of  an infinitesimal deformation of complex structure of arbitrary dimension.
This  was the essential starting point of the work by Kodaira and Spencer \cite{ks:I&II}  and with Nirenberg \cite{kns}.
in terms of deformations.
To vary complex structure, one varied the
differential in the Dolbeault complex of a complex manifold. In this context, analytic questions were important, especially for convergence of power series solutions.(See both Doran's and Mazur's historical annotated bibliographies \cite{doran:hist, mazur:hist} for richer details.) 

Algebraic deformation theory began with Gerstenhaber's seminal paper \cite{gerst:coh}:
\begin{quote}
This paper contains the definitions and certain elementary theorems 
of a deformation theory for rings and algebras....mainly associative rings and algebras with only
brief allusions to the Lie case, but the definitions hold for wider classes of algebras.
\end{quote}
Subsequent work of Nijenhuis and Richardson \cite{nij-rich:Liedef} provided more detail for the Lie case.
Computations were  in terms of formal power series.
If a formal deformation solution was found, its convergence could be studied, sometimes without estimates. 
This is the context in which we work in studying deformations of rational homotopy types. Unfortunately, the word \emph{`formal'} appears here in another context, referring to a dgca which is weakly equivalent to its homology as a dgca.

An essential ingredient of our work is the combination of this algebraic
geometric aspect with a homotopy point of view.  
Indeed we adopted  the
philosophy, later promoted by Deligne \cite{deligne:gm} in response  to
Goldman and Millson \cite{ goldman-millson} (see \cite{gm:ihes} for a history of that development),
that any problem in deformation theory is ``controlled'' by
a differential graded Lie algebra, unique up to homology
equivalence (quasi-isomorphism or quism) (cf. \ref{quism}) of \dgLs. 
 In Section \ref{Linfty}, we extend this philosophy further to control by $L_\infty$ -algebras.

For a problem controlled by a general \dgLL, the deformation equation 
(also known as the Master Equation in the physics and physics 
inspired literature and now most commonly as the Maurer-Cartan equation) is written
$$dp + 1/2 [p,p]=0.$$
It appears in this form (though with the opposite sign and without
the 1/2, both of which are irrelevant)
in the early works on deformation of complex structure 
by Kodaira, Nirenberg and Spencer \cite{kns} and Kuranishi \cite{kur1}.


Differential graded Lie algebras provide a natural setting in which to pursue 
the obstruction method
for trying to integrate ``infinitesimal deformations'', elements of
$H^1({Der} S Z)$, to full deformations.  In that regard,
$H^*({Der} S Z)$ appears not only as a graded Lie algebra (in
the obvious way) but also as a strong homotopy Lie algebra, a
concept  proven to be of significance in physics,
especially in closed string field theory\cite {jds:csft,z:csft,wz}. This allows us to go beyond
the consideration of quadratic varieties so prominent in Goldman and Millson.

Since we are trying to describe the space of homotopy types, it 
is natural to do so in homotopy theoretic terms.
Quillen's approach to rational homotopy theory emphasizes differential graded
Lie algebras in another way.  The rational homotopy groups  $\pi_* (\Omega X)
\otimes  \Q$ form a graded Lie algebra under Samelson product.  Moreover, 
Quillen
\cite{quillen:qht}  produces a non--trivial differential graded Lie algebra $\lambda_X$ 
which not only gives  $\pi_* (\Omega X) \otimes \Q$ as $H(\lambda_X)$ 
but also faithfully
records the rational homotopy type of $X$.  A simplistic way of characterizing
such an $\lambda_X$ for nice $X$ is as follows: 
There is a standard
construction $A$ such that for any \dgL $L$, we have $A(L)$ as a
dgca, and for $\lambda_X$, we have $A(\lambda_X) \to A^*(X)$ as a model for
$X$ \cite{majewski}.  (For ordinary Lie algebras $L$, $A(L)$ is the standard 
(Chevalley-Eilenberg) complex of alternating forms used to define Lie
algebra cohomology \cite{CE}.)

On the other hand, $A$  can also be applied to the \dgL $L \subset \,
{Der}\, S Z$ above;  then it plays the role of a classifying
space.

\begin{mainht}\label{mainht}  For a simply connected graded commutative algebra
of finite type $\HH$, homotopy types of pairs
$(X,i : \HH \simeq  H(X))$ 
are in 1--1 correspondence with homotopy classes of dga
maps  $A(L) \to  \Q$.  
\end{mainht}

Now $A(L)$ will in general $not$ be connected, so the homotopy types correspond
to the \lq\lq path components\rq\rq.  One advantage of this approach is that it
suggests the homotopy invariance of the classifying objects.  Indeed, if
$(S Z,d)$ and $(S Z',d)$ are homotopy equivalent as dga's, then
${Der}\, S Z$ and ${Der}\, S Z'$ will be weakly homotopy
equivalent as dg Lie algebras) (Theorem 3.9).  In particular, if
$(S Z,d)$ is a resolution of  $\HH$, then $H({Der}\, (S Z,d))$
will be an invariant of $\HH$, not just as a graded Lie algebra but, in fact, as 
an $L_\infty$-algebra.  Similarly if $L$ and $L'$ are homotopy equivalent
in the category of \dgLs or $L_\infty$-algebras, then $A(L)$ and $A(L')$ will be homotopy
 equivalent dgca's.


The variety and group of the Main Theorem depend on a particular choice 
of $L$, but
the moduli space $V/G$ depends only on the \lq\lq homology type of $L$ (in
degree 1)\rq\rq (Definition \ref{indegree1}).

 We show in \ref{ } that for a simply connected dgca $A$ with
finite dimensional homology, the entire dg Lie algebra 
$L = \, {Der} \, A$ is
homology equivalent to a dg Lie algebra 
$K$ which is finite dimensional in each degree.  In
particular, the $V$ and $G$ of the classification problem may be taken to be
finite dimensional when $\HH$ has finite dimension.

The varieties $V$ which occur above are the \lq\lq versal deformation\rq\rq\,
spaces of algebraic geometry.  Our combination of the algebraic--geometric and
homotopy points of view leads to a particularly useful (and conceptually
significant) representation of the minimal versal deformation space called
\lq\lq miniversal\rq\rq.  
The coordinate ring of the miniversal $W$ is approximated
by the $0^{th}$ cohomology of $A(L)$, and its Zariski tangent space (the vector
space of infinitesimal deformations of structure) is $H^1(L)$.  
However, the quotient action on
the miniversal $W$ is no longer given by a group action in general, but only by
a `foliation' (in the generalized sense of e.g. Haefliger) induced from an  $L_\infty$-action of $H^0(L)$.  

From this and deformation theory,
we deduce information about the structure of a general $V_L/\text{exp}\, L,$ 
for example:

$V_L/\text{exp}\, L$ is one point provided $H^1(L) = 0$; 

$V_L/\text{exp}\, L$ is a variety provided $H^0(L) = 0$;  and 

$V_L$ is non-singular (a flat affine space) provided
$H^2(L) = 0$.

If $L$ is formal, i.e., if there is a \dgL map $L \to H(L)$ inducing
a homology isomorphism, then $V_L$ is isomorphic to the quadratic
variety $\{\theta\in H^1(L)| [\theta,\theta] =0\}$.  If $L$ is not
formal, then $[\theta,\theta]$ may be only the {\it primary obstruction}.
Let $V_1$ denote $\{\theta\in H^1(L)| [\theta,\theta] =0\}$.
There is a {\it secondary} obstruction defined on $V_1$ which
vanishes on a subspace denoted $V_2$.
The secondary obstruction was first
addressed in specific examples by Douady \cite{douady}.
Proceeding in this way, the succesive obstructions to
constructing a full perturbation can be described in terms of
Lie-Massey brackets \cite{retakh} and lead to
a filtration $H^1(L) \supset V_1 \supset V_2 \dots$
with intersection being $V_L$.

There are conceptual and computational advantages to replacing
$(S Z,d)$ by the   quadratic model $A(\lambda_X)$ where $X$ 
is the formal space (determined by just the cohomology algebra $\HH$) 
or better yet by $A(L(\HH))$ where $L(\HH)$ is the free Lie coalgebra model in
which the differential is generated by the multiplication 
$\HH \otimes \HH \to \HH$
(\S 3).  (In the ungraded case, $L(\HH)$ is the complex for the 
Harrison  \cite{harrison}
cohomology of a commutative algebra and, here in characteristic 0, also
known as the cotangent complex.) While $(S Z,d)$ corresponds to a 
Postnikov system,  
$L(\HH)$ relates to cellular data for the formal space determined by $\HH$.
  A perturbation of $L(\HH)$ which decreases bracket
length can be identified with a Lie symmetric map $\HH^{\otimes k}\to \HH$ of
degree $2-k$; this can sometimes be usefully interpreted via Massey products.
\subsection{Applications}
Although the general theory we develop has some intrinsic interest, it is 
in the application to specific rational homtopy types that some readers will find the appropriate justification.
In the last two sections, we study several families of examples. 
Finding a complete  set of
invariants for rational homotopy types seems to be of about the same order of
difficulty as finding a complete set of invariants for the G--orbits of a
variety  $V$, but, for special cohomology rings, a lot more can be said  
\cite{halperin-stasheff,felix:Bull.Soc.Math.France80,lemaire-sigrist,neis}; 
for example, 
Massey product structures can be very helpful, though
they are in general described in a form that is unsatisfactory.  When the
dimension is a small multiple of the connectivity, results can be given in a
neat form (\S 8). The deformations of the homotopy type of a bouquet of spheres
allows for detailed computations and reveals some pleasant surprises.

We have mentioned that $A(L)$ behaves like a classifying space. We carry this insight to fruition 
in \S 9
 by considering  nilpotent topological fibrations
$$
F \to  E \to  B
$$
of nice spaces.  For these, Sullivan (\cite {sullivan:inf} p. 313)
has proposed an algebraic model for $B\, {Aut}\, F$, 
the classifying space for nilpotent fibrations with fibre $F$.  His model
is essentially $A\, ({Der}\, S Z)$ suitably altered to model a nice
space where $S Z$ models $F$.  Here we view a fibration as a \lq\lq
deformation\rq\rq\, of the trivial fibration $F \times B$.  The set of such
fibrations is a quotient  $V_L/\text{exp}\, L$ of the type 
considered above.
($L$ is the complete tensor product of ${Der}\, S Z$ with an
algebraic model of the base, cf. 9.4.)  It then follows that $B\, {Aut}\,
F$ is modeled by $A(D)$ for a suitable sub \dgL $D$ of ${Der}\, S Z$,
in the sense that equivalence classes of fibrations $F \to E \to B$ correspond
to homotopy classes of maps of $A(D)$ into the model of $B$, at least
if $B$ is simply connected (9.6).  By placing
suitable restrictions on $D$, we obtain certain special kinds of fibrations and
also ``almost'' a space which simultaneously classifies perturbations of $F$ and $F$
fibrations.  Moreover, we can describe, in terms of the weight of the
perturbation, which fibrations correspond to perturbing the homotopy type of $F
\times B$ or the algebra structure of $H(F \times B)$.  The principal problem
left is to decide whether an arbitrary quotient variety can arise as a moduli
space for homotopy types or fibrations.

\subsection{Outline}
We have tried to write this paper so it will be of interest and accessible to
algebraic topologists, algebraic geometers and algebraists.  
We have presented at least a quick sketch
of all the machinery we use.  It would not hurt to be somewhat familiar with
 rational homotopy theory 
(especially \cite {sullivan:inf,halperin:lectures,
halperin-stasheff, halperin-avramov,tanre:modeles}) and Sullivan's models
in particular.  The examples 4.3, 6.4, 6.5, and 6.6 in \cite{halperin-stasheff}
 may help put the abstract constructions in focus. In  \S \ref{miniversal}, we move
 from homotopy theoretic language to that of algebraic geometry for a more
 traditional approach to moduli functors.

 In \S \ref{models} , we recall the notion of a dgca model of a rational homotopy type and
particularly the formal one given by the Tate--Jozefiak resolution, as well as
its perturbations.  We show the perturbations form a cone on a projective
variety.

In \S \ref{dgl}, we recall the standard constructions $C(L)$ and $A(L)$ for a dg Lie
algebra $L$, as well as the adjoint $L(C)$ from  dgcc's to \dgLs.  The homology
or cohomology of these constructions has significance in both homotopy and Lie
algebra theory, but we are more concerned with the constructions themselves.
Since some of our results depend on finite type or boundedness restrictions, we show
how the construction $C(L)$
does not change homotopy type under certain changes in $L$.
We also look at the special model  $A(L(\HH))$ and at $L(\HH)$ 
itself, where $\HH = H(C).$

In \S \ref{main}, we look at the key notion of homotopy of dg algebraic maps (as opposed to
homotopy of chain maps) and the corresponding notion for coalgebras.  The
corresponding relation between perturbations is best expressed via a
differential equation.
In \S \ref{proof main}, we give a direct proof for the invariance of $V_L/\text{exp} L^0 $ 
in terms of $H(L)$.
is induced on $H(L)$
In \S \ref{invariance}, we use the differential equation to complete the proof of the main
results: that for the appropriate dg Lie algebra $L$, the set of homotopy types
$(X;i:\HH \simeq  H(X))$ or equivalently the set $V_L/\text{exp} L^0$ corresponds to
the set $[\Q,\hat C(L)]$ of homotopy classes of maps of $\Q$ into $\hat C(L)$, the
completion of $C(L)$.

In \S \ref{Linfty}, we explain how an $L_\infty$-structure is induced on $H(L)$ so that our classification
can be expressed in terms of $H(Der ??).$

In \S \ref{alggeom}, we switch to the algebraic geometric language for deformation theory. 
We focus on the miniversal variety $W_L$ contained in $H^1(L)$ and the corresponding moduli space.


In \S \ref{examples}, we consider examples computationally, including relations to Massey
products and  examples of  exp$\, L^0$ actions in terms of maps of spheres.

In \S \ref{fibrations}, we establish the corresponding results for fibrations and compare
fibrations to perturbations of the product of the base and fibre 
in terms of weight conditions. We conclude with some open questions \S \ref{questions}.

As the first draft of this paper was being completed (by slow convergence), we
learned of the doctoral thesis of Yves F\'elix \cite{felix:diss}
(published as \cite{felix:Bull.Soc.Math.France80})
 which obtains some of these
results from a point of view less homotopy theoretic as far as the
classification is concerned, but which focuses more on the orbit
structure in $V_L$ and is closer to classical
deformation theory.  F\'elix deforms the algebra structure 
as well as the higher order
structure.  We are happy to report our computations agree where they
overlap.

Later, we learned of the thesis of Daniel Tanr\'e and his lecture notes where he
carried out the classification in terms of Quillen models \cite{tanre:modeles} (Corollaire VII.4. (4)).
with slightly more restrictive hypotheses in terms of connectivity, producing the equivalent classification.

J.-C. Thomas \cite{thomas:thesisI}  analyzes the internal structure of fibrations
 from a compatible but different point of view \cite{thomas:fibrations}.
Subsequently, Berikashvili (cf. \cite{beri:specseq}) and his school have
studied `pre-differentials', which are essentially equivalence classes
of our perturbations.  Applications to rational homtopy types and to
fibrations have been developed by Saneblidze \cite{sane:pre-diff}
 and Kadeishvili \cite{kad:pre-diff} respectively.

These are but some of the many results in rational homtopy theory that
have appeared
since the first draft of this paper, some in fact using our techniques.
For an extensive bibliography, consult the one prepared by F\'elix
 \cite{felix:bib} building on an earlier one by Bartik.

Finally, we owe the first referee a deep debt for his insistence that our results
deserved better than the exposition in our first draft (compare the even more
stream of consciousness preliminary version in \cite {SS}).  He has forced us 
to gain some perspective (even pushed us toward choosing what was, at that time,
 the right category in which to work) which hopefully we have revealed in the presentation.  
 In particular, the
Lie algebra models which have been emphasized in \cite{baues-lemaire,neis} 
are certainly
capable of significant further utilization, whether for classifying spaces or
manifolds \cite {jds:ratsp}.  He has also encouraged us to make extensive 
use of the
fine expositions of Tanr\'e \cite{tanre:modeles,tanre:harr}
 which appeared as this manuscript began
to approach the adiabatic limit.
The original research in this paper was done in the late 70's and
early 80's \cite{Sjds,SS},
as will be obvious from what is {\it not} assumed as `well known'.  
Rather than rewrite the paper in contemporary
fashion, postponing it to the next millenium, we have
elected to implement most of the referee's suggestions while leaving 
manifest the philosophy of those bygone times.
\vfill\eject
\def\sime{\approx}
\section {Models of homotopy types}
\label{models}


We begin with a very brief summary of those features of rational homotopy theory
which are relevant for our purpose.

Quillen's rational homotopy theory  \cite{quillen:qht} focuses on the equivalence
of the
rational homotopy category of simply connected $CW$ spaces and the homotopy
category of simply connected dgcc's (differential graded commutative coalgebras)
over $\Q$.  Sullivan \cite {sullivan:inf}
 uses dgca's (differential graded commutative algebras)
and calls attention to minimal models for dgca's so as to replace
homotopy
equivalence by isomorphism.  Halperin and Stasheff \cite{halperin-stasheff}
 discovered another
class of models which turn out to be appropriate for classification and can be
used without any elaborate machinery.  Indeed, we recall here what little we
need.

For our entire discussion, we let $\Q$ denote an arbitrary fixed ground field of characteristic 0. We adopt a 
strictly cohomological point of view, i.e. all graded vector spaces will be written
with an upper index (unless otherwise noted) and all  differentials  on graded algebras (associative or Lie or...) will be derivations of degree 1 and square 0.
As will be 
explicitly noted, many  graded algebras we encounter will be either  non--negatively (as for cochains on topological spaces) or non-positively
graded (as in algebraic geometry).  For an associative algebra  $A = \underset{n\ge 0}{\oplus} A^n$ and
 $A^0 = \Q$, we call $A$ \emph{connected}.  When needed, we will refer to 
{\bf simply connected} algebras, i.e., $A^0 = \Q$ and $A^1 = 0$.

\begin{dfn}\label{D:qht}
 Two dgca's $A_1$ and $A_2$ have the same {\bf rational
homotopy type} if there is a dgca $A$ and dga maps $\phi_i : A \to A_i$ such
that $\phi^*_i : H(A) \to H(A_i)$ is an isomorphism.  This is an equivalence
relation since such an $A$ can always be taken to be free as a gca; it is then
called a {\bf model} for $A_i$.  
\end{dfn}

One way to construct a model is as follows:

For any connected graded commutative algebra $\HH$ over $\Q$, there are free
algebra resolutions $A \to  \HH$; that is,
 $A$ is itself a  dgca, free as gca,  and the map,
which is a morphism of dga's regarding $\HH$ as having $d = 0$, induces an isomorphism $H(A)
 \overset{\sime}{\to}  \HH$.  Indeed, there is a minimal free dgca resolution which we denote
$(SZ,d)\to \HH$ due to Jozefiak \cite{jozefiak}
 which is a generalization of the Tate resolution \cite{tate} 
in the ungraded case.

\subsection{ The Tate--Jozefiak resolution in characteristic zero}
\label{T-J}
The free graded commutative algebra $SZ$ on a graded vector space $Z$ is
$E(Z^{odd}) \otimes P(Z^{even})$ where $E$ = exterior algebra and
$P$ = polynomial algebra.  In resolving a connected graded commutative
algebra $A$ by a dgca $(S Z, d)$, the 
generating graded vector space $Z$  will be bigraded:
$$
Z^n =\bigoplus^{w + q=n}_ {q \leq 0 }Z^{q,w}.
$$
We will refer to $n = w+q$ as the {\bf total} or {\bf topological} or
{\bf top} degree, to $q$ as the {\bf resolution} degree (for
historical reasons), and to $w$ as the {\bf weight} = topological
degree minus resolution degree. 
When a single
superscript grading appears, it will always be total degree.
The graded commutativity is with respect to the total degree.

The Tate--Jozefiak resolution $(S Z,d)\to \HH$ of a connected 
cga $\HH$ has a differential $d$ which increases total degree by $1$, 
decreases resolution
degree by $1$ and hence \emph{preserves} weight.  It is a graded derivation with
respect to total degree.  
For connected $\HH$, we let $Q\HH = \HH^+/(\HH^+ \cdot \HH^+)$ be the module of
indecomposables.  The vector space $Z^{0,*}$ is $Q\HH$.  The resolution  
$\rho : (S Z,d)\to \HH$ induces an isomorphism
$H(\rho): H(S Z,d) \overset{\sime}{\to} \HH$ which identifies
$\HH$ with $H^{0,*} (S Z,d)$.


 We refer to the Tate--Jozefiak resolution $(S Z,d)$ also as
 the {\bf bigraded} or {\bf minimal model} for  $\HH$.
It is minimal in the sense that the dimension of each $Z^{q,*}$
is as small as possible, but
the minimality is best expressed as $dZ
\subset
S ^+Z \cdot S ^+Z$ where $S ^+Z = \underset{n>0}{\oplus}
(S Z)^n$.

Several comments are in order.  Just as in ordinary homological algebra, one can
easily prove $(SZ,d)$ is uniquely determined up to isomorphism.  If
$(SZ',d) \underset{\rho^\prime}{\longrightarrow} 
\HH$ is any other minimal
free bigraded dgca with $H(\rho^{\prime })$ an isomorphism, then $(SZ^\prime ,d)$ is  isomorphic to $(S Z,d)$. 

\subsection{ The Halperin--Stasheff or filtered model} 
\label{H-S}
Given a dgca $(A,d_A)$ 
 and an isomorphism $i : \HH \sime H(A)$, Halperin and
Stasheff (\cite{halperin-stasheff} p. 249)
 construct a {\bf perturbation} $p$ of the Tate-Jozefiak $(SZ,d)$ of  $\HH$
  and thus a derivation $d+p$ on $SZ$ and a map
of dgca's $\pi : (SZ,d+p) \to (A,d_A)$ such that $(d+p)^2 = 0$ and
\begin{enumerate}
\item $p$ decreases resolution degree by at least 2 (i.e., decreases weight,
thus $(SZ,d+p)$ is filtered graded, not bigraded), and
\item $H(\pi)$  is an isomorphism $H(S Z,d+p) \sime  H(A)$.

\noindent In fact, filtering $S Z$ by resolution degree, we have in the resulting
spectral sequence
\item $(E_1, d_1) = (S Z,d)$ and $E_2 = H(S Z,d)$ which is concentrated
in resolution degree 0 and,
\item  by construction, $H(\pi)$ is the composite 
$$H(S Z,d+p) \overset{\sime }{\rightarrow}
H(S Z,d) \overset{\sime }{\underset{\rho^*}{\longrightarrow}}\HH
\overset{\sime }{\underset{i}{\longrightarrow}}H(A)$$
 \noindent
 with the first
isomorphism being the edge morphism of the spectral sequence.
\end{enumerate}

Now consider two dgca's $A$ and $B$ with isomorphisms
$$
H(A) \overset{\sime }{\underset{i}{\leftarrow}}\HH
\overset{\sime }{\underset{j}{\rightarrow}}H(B).
$$

The corresponding perturbed models $(S Z,d+p)$ and $(S Z,d+q)$ are
{\bf homotopy equivalent  relative to} $i$ {\bf and} $j$ (i.e., by
a map $\phi$
such that $H(\phi)= j i^{-1})$ if and only if there exists

$(*)$ a dga map $\phi : (S Z,d+p) \to (SZ, d+q)$ such that $\phi - Id$ lowers resolution degree. (It follows that $H(\phi)= j i^{-1}$.)

\begin{dfn} \label{D:perturb}
Let $(S Z,d)\to \HH$ be the bigraded model of $\HH$. 
A {\bf perturbation} of $d$ is a weight decreasing
derivation $p$ of total degree $1$ such that $(d+p)^2 = 0$.
\end{dfn}

For any perturbation, $(S Z,d+p)$ has cohomology isomorphic to
the original $\HH$.  Thus the classification of homotopy types can be done in
stages:
\begin{enumerate}
\item Fix a connected cga $\HH$.
\item Let $V=\lbrace \text{ perturbations } p\text{ of the
minimal model }(S Z,d)\text{ for }\HH\rbrace$.
\item Consider $V/ \sim $  where we write $p\sim q$ if there
exists a \hfill\newline $\phi : (S Z,d+p) \to (S Z,d+q)$ as
in $(*)$.
\item Consider $ {Aut}\, \HH\backslash V/\sim$ as a ``moduli space" to be called $M_\HH$.
\end{enumerate}

Here ${Aut}\, \HH$ acts as follows: If $\rho : (S Z,d)
\to \HH$ is a Tate-Jozefiak resolution of $\HH$ and $g \in \,{Aut}\, \HH$,
then $g\rho$ is also a resolution and hence $g$ lifts to an
automorphism $\bar g : (S Z,d) \to (S Z,d)$. Now if
$(S Z,d+p)$ is a perturbation, so is $(S Z,d+\bar
gp\bar g^{-1})$ and $p \to \bar gp\bar g^{-1}$ is the action of
${Aut}\, \HH$ on $V$.  The topology on ${Aut}\, \HH \backslash V/\sim$
will turn out to have an invariant meaning, so that
this quotient can meaningfully be called {\bf the space of homotopy types 
with cohomology algebra $\bold \HH$}.

If $\HH$ is finite dimensional, we will see fairly easily that $V$ is an algebraic variety and the equivalence
is via a group action of a unipotent algebraic group.

Consider the \dgL 
 ${Der}\, S Z = \underset{i}{\oplus}{Der}^iS Z$
where  ${Der}^i\, S Z$ is the sub--\dgL consisting
of all derivations that raise total
degree by the integer  $i$.  Define $L\subset {Der}\, S Z$
 to consist of all derivations which
decrease weight = total degree minus resolution degree.  Any $\phi
\in L$ of a particular total degree can
be regarded as an infinite sum (and conversely)
$$
\phi = \phi_1 + \phi_2 + \dots +\, \phi_k \, + \dots
$$
where $\phi_k$ decreases weight by $k$.  The infinite sum causes
no problem because there are no elements of negative weight and,
for any $z \in S Z$ of fixed
weight $k \ge 0$, $\phi (z)$ will be a finite sum: $\phi_1
(z)+\dots \phi_{k+1}(z)$. In other words, $L$ is complete with
respect to the weight filtration.

The weights on $L$ above allow us to analyze further $V \subset L^1$, i.e., 
$$V =
\text{ the variety } \lbrace p \mid (d+p)^2 = 0 \rbrace.$$

\begin{thm}\label{T:cone} If $\HH$ is finite dimesional,
   $V$ is the cone on a projective variety.
\end{thm}

\begin{proof}  The point is that, if $Z$ is finite dimensional, so will be $L^1.$ 
Since $d$ preserves weight, the equation
$(d+p)^2 = 0$ is weighted homogeneous.  That is, writing $p = \Sigma\ p_i$
 where  $p_i$ decreases weight by $i$, we have that $\bar t p = \Sigma\  t^i p_i$
satisfies $(d+\bar tp)^2 = 0$, as can be seen by expanding and collecting terms
of equal weight.  The locus of this ideal $$I = \lbrace d+p \mid (d+p)^2 =
0\rbrace$$ of weighted homogeneous polynomials is a subvariety of
a (finite dimensional) projective subspace.
\end{proof}

If $\HH$ is only of finite type (i.e. finite dimensional in each degree), then $V$ will be a pro-algebraic group
acted upon by a pro-unipotent group.


As $(S Z,d)$ is a model for (the cochains of) a simply
connected space $X$ of  finite type, we can interpret $H(Z,d)$ as dual 
to $\pi_* (X)$ where $d$ acts on $Z$ as the \emph{indecomposables}, i.e. the
quotient $S^+Z/S^+Z \cdot S^+Z$ (\cite {sullivan:inf} p. 301). 
For the  restriction 
for $Z$ to be finite dimensional, it is sometimes more
appropriate for us to model the space 
$X$ via a differential graded Lie algebra.  We next
consider several aspects of the theory of \dgLs and return to the
classification
in \S 4.
\newpage
\vfill\eject

\section {Differential graded Lie algebras, models and perturbations}
\label{dgl}

Differential graded Lie algebras appear in our theory in two ways, as models for
spaces and as the graded derivations of either a dgca or of another dg Lie
algebra (or the coalgebra analogs).  We will be particularly concerned with 
certain standard constructions
$C$ and $L$ which provide adjoint functors from the homotopy category of \dgLs
to the homotopy category of dgcc's and vice versa \cite{jcm, quillen:qht}.

CONNECTIVITY ASSUMPTIONS

  Quillen's approach to
rational homotopy theory is to construct a functor from simply connected
rational spaces to \dgLs and then apply $C$ to obtain a dgcc model.  The functor
$A(\quad ) = \, {Hom}\, (C(\quad ),\Q)$ from \dgLs to dga's fits more readily into a traditional
 exposition, but the usual subtleties of the Hom functor
necessitate the detour into the more natural differential graded coalgebras.  We recall
definitions and many of the basic results from Quillen  \cite{quillen:qht},
 especially Appendix B.

We will also be concerned with perturbations of filtered dg Lie algebras.  We conclude
this section with a comparison of Der$\, L$ and Der$\, A(L)$.
(Recall all vector spaces are over $\Q$.)

A crucial motivation for Quillen's theory is Serre's result 
\cite{serre:classes} that
$\pi_*(\Omega X)\otimes \Q$ is isomorphic as a graded Lie algebra to 
the primitive subspace $P\, H_*(\Omega X,\Q)$, 
so Quillen uses \dgLs with lower indices and $d$ of degree $-1$.  We do not
follow this tradition, but rather, consistent with our cohomological point of view, 
our \dgLs will 
have upper indices and differentials $d$ of degree $+1$.  \footnote{The two traditions
are identified via the convention: $L^i = L_{-i}$.}
(The history of graded Lie algebras is intimately related to homotopy theory,
see \cite{haring}.)

\begin{dfn}\label{D:dgl}
   \cite{quillen:qht}, p. 209.  A {\bf differential graded Lie algebra}
(dgL) $L$ consists of
\begin{enumerate}
\item a graded vector space $L = \{ L^i\} ,\ i \in \mathbf[Z]$,
\item a graded Lie bracket $[\ ,\ ] : L^i \otimes L^j \to L^{i+j}$ such that
\begin{gather*}
\quad [\theta,\phi] =- (-1)^{ij}[\phi,\theta]\qquad \text{and}\\
\quad (-1)^{ik}[\theta,[\phi,\psi]] + (-1)^{ji}[\phi,[\psi,\theta]] +
(-1)^{kj}[\psi,[\theta,\psi]] = 0,
\end{gather*}
\noindent
(In other words, $ad\, \theta := [\theta,\ ]$ is a derivation of degree $i$ for
$\theta \in L^i$.)
\item a graded Lie derivation $d : L^i \to L^{i+1}$ such that
$d^2 = 0$. \end{enumerate}
\end{dfn}
\begin{dfneg} \label{D:der} For any dga $(A=\oplus A^i,d_A)$, we 
have the dg Lie
algebra $L = \, {Der} \, A = \oplus {Der}^i\, A$ where
${Der}^i\, A = \{\text{derivations }
\theta : A \to A \text{ of degree } i\}$.  The bracket  is the graded
commutator: $[\theta, \phi] = \theta\phi -(-1)^{ij}\phi\theta$ for $\theta \in
{Der}^i\, A$ and $\phi \in {Der}^j\, A$.  We have the
differential $d_A \in {Der}^{1}\, A$ and $d_L(\theta) :=
[d_A,\theta] := d_A\theta - (-1)^i \theta d_A$, i.e., $d_L = ad(d_A)$.
\end{dfneg}
Note that $i$ ranges over all integers, not necessarily just positive or just
negative; later we will have to consider bounds.  Even if $A$ is of finite type,
Der$\, A$ need not be.

Of course, for any \dgL $L$, the homology $H(L)$ is again a graded Lie algebra, but there is more structure than that inherited for $L$  (see \ref{Linfty}).

\begin{exa}\label{D:tjder} For the special case of the Tate-Jozefiak resolution 
$(S Z,d)$ of $\HH$, we are interested in the 
sub--Lie algebra $L \subset \, {Der}\,  S Z$
consisting of all derivations which decrease weight.
\end{exa}
\begin{dfneg}\label{D:derL} For any \dgL $(L=\oplus L^i,d_L)$, we 
have the dg Lie
algebra ${Der}\, L = \oplus {Der}^i\, L$ 
where ${Der}^i\, L = \{\text{derivations }\theta : L^j\to  L^{j+i}\}$ 
with again the graded commutator bracket and the induced differential
$d(\theta) := [d_L,\theta] :=   d_L\theta - (-1)^i\theta d_L$.
\end{dfneg}
\subsection{ Differential graded commutative coalgebras and  \dgLs}

In order to dualize conveniently, and motivated by the case of the 
(admittedly not commutative) chains of a topological space, we 
cast our definition of dgcc in the following form.  

\begin{dfn}\label{D:dgcc} 
A {\bf dgcc}  {\bf d}ifferential {\bf g}raded ({\bf c}ommutative {\bf c}oalgebra) consists of
\begin{enumerate}
\item a graded vector space $C = \{ C^n$, $n \in Z\} $, \item a
differential $d: C^n \to C^{n+1}$ and \item a differential map
$\Delta : C \to C \otimes C$ called a {\bf comultiplication}
which is associative and graded commutative and \item a
counit $\epsilon : C \to \Q$ such that $(\epsilon \otimes 1)\Delta
= (1 \otimes \epsilon)\Delta = id_C$.  
\end{enumerate}

\noindent
We say $C$ is {\bf augmented} if there is given a dgc map $\eta: \Q \to C$
and that $C$ is
{\bf connected} if $C^n = 0$ for $n < 0$ and $C = \Q$.
\end{dfn}

The dual ${Hom}\, (C,\Q)$ of a connected dgcc is a
connected dgca with $A^{-n} = {Hom}\, (C^n,\Q)$.  Conversely, if $A$ is a dgca
of finite type (meaning that
each  $A^{-n}$ is finite dimensional over $\Q$), then ${Hom}\,(A,\Q)$ inherits the structure of a dgcc.


Most of our constructions make use of the tensor algebra and tensor coalgebra; we pause
to review structure, notation and nomenclature.

Let $M$ be a graded $\Q$-vector space; it generates free objects as 
follows:

\subsection{ The tensor algebra and free Lie algebra}.
\label{TA}

The tensor algebra $T(M) = \underset{n\geq 0}{\oplus} M^{\otimes n}$ 
where $M^{\otimes 0} = \Q$ and $M^{\otimes n} := M \otimes \dots \otimes M$
 with $(a_1 \otimes \dots \otimes a_p)(a_{p+1} \otimes \dots \otimes
a_{p+q}) = a_1 \otimes \dots \otimes a_{p+q}$ is the free  associative
algebra generated by $M$. It is the free graded associative
algebra generated by $M$ with respect to the {\bf total} grading, which 
is $\Sigma (|a_i|)$ where $|a_i|$ is the grading of $a_i$ in $M$.

The free graded Lie algebra $L(M) $ can be realized  as a Lie sub-algebra
of $T(M)$ as follows: 
Regard $T(M)$ itself as a graded Lie algebra under the {\it graded commutator}
$$[x,y] = x\otimes y - (-1)^{deg \, x deg \, y}   y \otimes x,$$
then the Lie sub-algebra generated by $M$ is (isomorphic to) the free
Lie algebra $L(M)$.  In characteristic 0, 
this can usefully be further analyzed (cf. Friedrichs' Theorem \cite{perrin,serre:lalg})
by considering
$T(M)$ as a Hopf algebra with respect to the \emph{unshuffle} diagonal 
$$\bigtriangleup (a_1 \otimes \dots \otimes a_n) = \underset{(p,q)-\text{shuffles\,}\sigma}
{\Sigma} (-1)^\sigma
(a_{\sigma (1)} \otimes \dots \otimes (a_{\sigma (p)}) \otimes (a_{\sigma (p+1)}
\otimes \dots \otimes a_{\sigma (n)})$$
 where $\sigma$ being an unshuffle (sometimes called a shuffle!) means
$\sigma (1) < \sigma (2) < \dots < \sigma (p)$ and $\sigma (p+1) < \dots <
\sigma (n)$ and $(-1)^\sigma$ is the sign of the graded permutation.
The free graded Lie algebra $L(M)$ is then isomorphic to 
the algebra of primitives, $P(T(M))$,
i.e., $x \in T(M)$ is {\bf primitive} if and only if $\bigtriangleup x =
x \otimes 1+1 \otimes x$.  The Hopf algebra $T(M)$ is isomorphic
to the universal enveloping algebra $U(L(M))$
 \cite{quillen:qht}.

{\bf Convention}.  The {\bf shift} functor $s$ on differential graded objects shifts
degrees down by one for algebras and up by one for coalgebras: 
\begin{gather*} 
s: A^{q+1} \simeq  (sA)^q and\\ 
s : C^{q-1} \simeq  (sC)^q,
\end{gather*}
while the differential anti-commutes with $s$, i.e. $ds=-sd.$
\begin{dfneg}\label{D:L(C)}
 \cite{quillen:qht}  p. 290.  For any augmented differential graded coalgebra
(dgc) $C$ with commutative diagonal $\Delta$, let $L(C)$ be the
differential graded Lie algebra which is the free graded Lie
algebra on $s {\bar C} = s(C/\Q)$ with differential $d$ which is the
derivation determined by $d_C$ and $\Delta$, i.e.,
 $$ d(sc) = -s(d_Cc) + \frac{1}{2} \Sigma(-1)^{deg \, c_i} [sc_i,sc'_{i}] $$ 
where $\Delta c = \Sigma c_i \otimes c'_{i}$ 
or, in Heyneman-Sweedler notation,
$\Delta c = \Sigma c_{(1)} \otimes c_{(2)}$,  omitting the terms
 $1\otimes c$ and $c\otimes 1$.
\end{dfneg} 

An enlightening alternative description is that $L(C)$ can be identified with $P
\Omega C$, the space of primitive elements in the cobar construction
 on the commutative
coalgebra $C$ \cite {adams:cobar}.  That is, $\Omega C$ is the graded Hopf 
algebra $T(s \bar C)$ described above with differential which is the
derivation determined by 
$$
d(sc) = -s(d_Cc) + \Sigma (-1)^{deg \, c_i} sc_i \otimes sc'_{i}.
$$

In particular, we wish to apply this construction $L$ to the dual $\HH_*$ of a
graded commutative algebra $\HH$. Thinking of $\HH$ as a cohomology algebra and
assuming $\HH$ connected of finite type, the graded linear dual $\HH_*$  
we  then regard as homology,

The coalgebra grading is given by $(\HH_*)^{-n} = {Hom}\, (\HH^n,\Q)$, which
can also be written as $\HH_n$. 

By abuse of notation, we
will write $L(\HH)$ instead of $L(\HH_*)$.  The differential above then becomes
$$
d(s \phi)(su\otimes sv) =\pm \phi (uv)
$$
for $u,v \in \HH$ and $\phi : \HH^n \to \Q$.

{\it Henceforth, we rely heavily on our assumption that all (non--differential) cga's $\HH$ 
are simply connected  of finite type. Thus $L(\HH)$ will be of finite type.}


Models for dgc algebras can  be obtained conveniently via the functor  $C$
adjoint to $L$.  For ordinary Lie algebras $L$, the construction  $C(L)$ reduces
to the complex used by Cartan--Chevalley--Eilenberg \cite{CE} and Koszul
\cite{koszul:liecoh} to define the homology of Lie algebras.

\subsection{ The tensor coalgebra and free Lie coalgebra}
\label{TC}

Again let $M$ be a graded $\Q$-vector space.  Consider the
tensor coalgebra $T^c(M) = \underset{n\geq 0}{\oplus} M^{\otimes n}$ with
the \emph{deconcatenation} or \emph{cup coproduct}
$$\Delta (a_1 \otimes \dots \otimes a_n) = \underset{p+q=n}{\Sigma} (a_1
\otimes \dots \otimes a_p) \otimes (a_{p+1} \otimes \dots \otimes a_n).$$

\begin{rmk} Some authors refer to $T^c(M)$ as the cofree associative unitary coalgebra cogenerated by $M$,
but it is cofree only as a pointed irreducible coalgebra. On the other hand, others refer to the  `tensor coalgebra' meaning the completion $\hat T^c(M)$ which is  the cofree associative unitary coalgebra, see below and \cite{walter:handbook} Remark 3.52.  \end{rmk}

As for the tensor algebra, the total grading
$\Sigma (|a_i| + 1)$ where $|a_i|$ is the grading of $a_i$ in $M$
is usually most important, but the number of tensor factors provides
a useful decreasing filtration: $\mathcal F_pT^c(M) =
\underset {n\geq p}{\oplus}M^{\otimes n}$.
For our purposes, it is important to pass to the completion $\hat T^c(M)$
of $T^c(M)$ with respect to that filtration:
$\hat T^c(M) \cong 
\underset{n\geq 0} {\prod}M^{\otimes n}$. The diagonal $\Delta$ on $T^c(M)$
extends to $\Delta: \hat T^c(M)\to \hat T^c(M)\hat \otimes \hat T^c(M)$,
the completed tensor product being given by the completion with respect to
$F_n ( {\hat \otimes} ) = \underset{p+q\geq n}{\oplus} F_p  \otimes F_q$. 
 With this structure, $\hat T^c(M)$ is 
universal with respect to the category of 
cocomplete connected graded 
coalgebras (cf.\cite{walter:handbook} for a very thorough treatment or  \cite {serre:lalg} for the algebra version
 and \cite  {quillen:qht} Appendix A for connected coalgebras).

\begin{dfn}\label{D:cocomplete}
A {\bf cocomplete} dgc $(C,d,\Delta)$  consists of
\begin{enumerate}
\item a decreasingly filtered dg vector space $C$ which is
cocomplete (i.e., $C = \underset{\leftarrow}{lim}\, C/ F_p C$), and
\item  a filtered (= continuous) chain map $\Delta : C \to C {\hat \otimes} C$ 
which is associative.
\end{enumerate}
\end{dfn}
\noindent
Morphisms of cocomplete dgcc's respect the given filtrations.

Thus $(\hat T^c(M),\Delta)$ is the cofree graded cocomplete associative 
coalgebra cogenerated
by $M$ (cf. \cite{michaelis}).  That is,  $(\hat T^c(M),\Delta)$ has the following 
universal property:
$$
Coalg(C,\hat T^c(M)\cong    Hom(C,M)
$$
for 
cocomplete associative
coalgebras $C$, where $Coalg$ denotes morphisms of such coalgebras while
$Hom$ denotes linear maps of $\Q$-vector spaces.
In other words, the functor  $(\hat T^c(M),\Delta)$ is adjoint to the forgetful
functor from  graded cocomplete associative 
coalgebras to graded vector spaces.

A major attribute of cocomplete coalgebras and
$\hat T^c(M)$ in particular is that their space of group-like elements,
$\mathcal G C:=\{c\vert \Delta c = c\hat\otimes c\}$ need not be spanned
by $1\in \Q$.  There
is the obvious $1 \in \Q = M^{\otimes 0}$ with $\Delta 1 = 1
\otimes 1$ but for any $p \in M$, the element $q = 1 + p + p
\otimes p + p \otimes p \otimes p + \dots$ also has $\Delta q = q
\otimes q$.

The cofree graded Lie coalgebra can be realized   as
a quotient of $T^c(M)^+$, where $+$ denotes the part of strictly
positive $\otimes$--degree, i.e. $F_1$. The tensor coalgebra $T^c(M)$ can be
given a Hopf algebra structure by using the shuffle
multiplication, i.e., 
$$ (a_1 \dots \otimes a_p) * (b_1 \otimes \dots \otimes b_q) = 
\underset{\sigma}{\sum} (-1)^\sigma c_{\sigma(1)} \otimes \dots
\otimes c_{\sigma(p+q)} $$ 
where $\sigma$ is a $(p,q)$-shuffle
permutation as above and $c_1 \otimes \dots \otimes c_{p+q}$ is
just $a_1 \otimes \dots \otimes a_p \otimes b_1 \otimes \dots
\otimes b_q$.  The Lie coalgebra $L^c(M)$ consists of the
indecomposables of this Hopf algebra, i.e.,
$T^c(M)^+/T^c(M)^+*T^c(M)^+$.  The Hopf algebra $T^c(M)$ is,
provided $M$ is 
 the universal enveloping coalgebra of
$L^c(M)$; the associated graded  of $T^c(M)$ is, as algebra,
isomorphic to $SM)$, the free graded commutative
algebra on $M$.

This discussion can be `completed' by using $\hat T^c(M)^+$ and
the desired universal property follows:
$$
Liecoalg(C,\hat L^c(M)\cong    Hom(C,M)
$$
for  cocomplete Lie coalgebras $C$.
The cofree graded cocomplete commutative coalgebra $S^c(M)$ generated by $M$
is the maximal commutative sub--coalgebra of $T^c(M)$, i.e., the
cocomplete sub--coalgebra of all graded symmetric tensors (invariant under
signed permutations).  The
cofreeness of $S^c(M)$ means that a coalgebra map $f$ from a
connected commutative coalgebra $C$ into $S^c(M)$ is determined by the
projection $\pi f : C\to M$ and the same is true with respect to
coderivations $\theta:S^c(M) \to S^c(M)$: $\theta$ is 
determined by the projection $\pi\theta: S^c(M) \to M$.

Recall that a {\bf  coderivation} $\theta$
  of a coalgebra $C$ is a linear map
 $\theta:C\longrightarrow C$ such that
$(\theta\otimes 1 + 1\otimes \theta)\Delta = \Delta \theta$.

\subsection{ The standard construction $C(L)$}
\label{C(L)}
The Cartan--Chevalley--Eilenberg chain complex
\cite{CE,koszul:liecoh} generalizes easily to graded Lie algebras and even
\dgLs.

\begin{dfn} \label{D:C(L)} (\cite{quillen:qht} p. 291).
 Given a \dgL $(L,d_L)$, let $C(L)$ denote the
free cocomplete commutative coalgebra $S^c(sL)$ with total differential
cogenerated  as a graded coderivation
by $d_L$ and $[\ ,\ ]$, meaning that it is the graded coderivation such that:
$$
d(s \theta) = -s(d_L \theta)$$
and
$$
 d(s \theta \wedge s \phi) = s[\theta,\phi] - s(d\theta)\wedge s\phi
 - (-1)^{|s\theta|} s\theta\wedge s(d\phi)
$$
(here $s \theta \wedge s \phi$ is the symmetric tensor 
$s\theta\otimes s\phi+(-1)^{|s\theta|\,|s\phi|}s\theta\otimes s\phi).$
\end{dfn}
If $L$ is ungraded, $C(L)$ is precisely 
the Cartan--Chevalley--Eilenberg chain complex
\cite{CE,koszul:liecoh} and $H(C(L))$ is the Lie algebra
homology $H_*^{Lie} (L)$.  We will exhibit
this in more detail below for the cochain complex and cohomology.

\begin{dfn}\label{D:hdgl}
The {\bf homology} $H^{dg\ell }(L)$ is the graded coalgebra
$H(C(L),d_C)$ with the grading and comultiplication inherited from $C(L)$.  The
decoration $H^{dg\ell }$ indicates this is the homology of $L$ qua differential
Lie algebra, i.e., in the sense of category theory or homological algebra.
\end{dfn}
\begin{thm}\label{T:adjoint}
 (\cite{quillen:qht} Appendix B) $C$ and $L$ are adjoint
functors between the homotopy category of \dgLs with $L_q = 0$ for $q \leq 0$
and the homotopy category of simply connected dgcc's.  The adjunction morphisms
$\alpha:LC(L) \to L$ and $\beta:C\to CL(C)$ induce isomorphisms in
homology.
\end{thm}


Extending the terminology from dgca's (as in \cite{baues-lemaire}), we speak 
of $\alpha:LC(L)\to L$
 as a {\bf model} for $L$ and of $\beta:C\to CL(C)$ as a {\bf model} for $C$.  
If $C=C(X)$ is a commutative
chain coalgebra over $\Q$ of a simply connected topological space $X$, we speak
of $C\to L(C)$ as a \dgL model for $X$.

We will often find it useful to replace $L$ by a simpler \dgL $K$ with the same
homology and then will want to compare $C(L)$ and $C(K)$.


\begin{thm}\label{T:Cquism}  \cite{quillen:qht}.  If $f: L \to  K$ is a map of 
dg Lie algebras which are positively graded 
and $H(f) : H(L) \simeq  H(K)$, then the induced map
$H(C(f)) : H^{dg\ell }(L) \to H^{dg\ell }(K)$ is an isomorphism; i.e., $C(L)
\to  C(K)$ is a homology equivalence/quasi-isomorphism.
\end{thm}

To compare the construction for more general \dgLs will be important for our
homotopy classification.

\subsection{ The Quillen and Milnor/Moore et al spectral sequences
   \cite{ quillen:qht,
milnor:BG, moore:alghomologique}}\label{QMM}

The coalgebra $C(L)$ is equipped with a natural increasing filtration, the
tensor degree, i.e., $s \theta_1 \otimes \dots \otimes s \theta _n \text{ (where}
\ \theta _i \in L)$ has filtration $p\geq n$.  The associated spectral sequence has
\begin{align}
E &= (S^c(sL),d_L) \quad \text{ so that}\\
E_1 &= (C(H(L)),d_1 = [\ ,\ ])
\end{align}
and hence  $E_2$ is the Lie algebra homology of $H(L)$, while the spectral sequence
abuts to $H^{dg\ell}(L)$.

Compare the spectral sequence 
for relating the homology of a loop space $\Omega X$ to that of $X$
as its classifying space. In that situation, 
$E_2$ is the associative algebra homology of $H(\Omega X)$, that is,
$Tor \, H^{(\Omega X)}(\Q,\Q)$, the homology of $T^c(s{\bar
H}(\Omega X))$, with $d_1$ determined  by the loop
multiplication $m_*$.  But over the rationals, Quillen has constructed
a dg Lie algebra $\lambda_X$ such that $H(\lambda_X) \simeq  \pi_*(\Omega X)
\otimes \Q$ and $C(\lambda_X)$ is homotopy equivalent to the coalgebra
of chains on $X$.  Moreover, over the rationals, Serre \cite{serre:classes}
has shown
that $H(\Omega X) \simeq  U(\pi_*(\Omega X) \otimes \Q)$, the
universal enveloping algebra on the Lie algebra $\pi_*(\Omega X)
\otimes \Q$.  Thus comparing Quillen and Milnor--Moore at the
$E_2$ level we have $$ H^{Lie}(\pi_*\Omega X \otimes \Q) \simeq
H^{assoc} (U( \pi_*\Omega X \otimes \Q)) $$ by a well--known
result in homological algebra, while the spectral
sequence abuts to $$ H(C(\lambda_X)) \simeq  H(X).  $$

To fix ideas and for later use, we consider a special case  in which
$H(L) = L(V)$, the
free Lie algebra on a positively graded vector space $V$.  Since $H(L)$ is free,
we can choose representative cycles and hence a \dgL map $H(L) \to L$ which is a
homology isomorphism.  Thus we have isomorphic spectral sequences, but for
$H(L)$ the spectral sequence collapses: $E_\infty \approx E_2 \approx
H(S^csH(L))$.  That is, since $d_1$ is $[\ ,\ ]$ and $H(L)$ is free, only the
$[\ ,\ ]$-indecomposables of $H(L)$ survive, i.e.,
$$
H^{dg\ell }(L) \approx H(S^c sH(L)) \approx \Q \otimes sV
$$
with $sV$ primitive.

Use of this spectral sequence implies (compare Quillen, Appendix B \cite{quillen:qht}: 


\begin{thm}\label{T:Cf} If $f : L \to K$ is a map of \dgLs 
which are connected and $H(f) : H(L) \spprox H(K)$, then
$H(C(f))$ is an isomorphism.
\end{thm} 


\subsection{ The standard construction $A(L)$}
\label{A(L)}
The dual of $C(L)$ is a dgca which we denote by $A(L) = {Hom}\, (C(L),\Q)$.


We can interpret $A(L)$ in terms of  the (set of) alternating forms on  $L$:  for any
L--module $M$, a linear homomorphism $C(L) \to M$ of degree $q$
can be regarded as an alternating multilinear form on $L$.  
(Only if $C(L)$ is of finite dimension in each degree,  e.g. if $L$ 
is non-negatively graded of
finite type, should we think of
$A(L)$ as $S(sL^{\dual})^2.$ 
\footnote{With the exception of (co)algebras that are interpreted as, vice versa,
(co)homology, we will usually denote duals by $^\dual.$}
The coboundary on $A(L)$ can then be written explicitly as 
\begin{align*}
\quad (d_Af) (X_1,\dots,X_n) = \Sigma &\pm f(\dots,dX_i,\dots)\\
 &\pm \Sigma f([\hat X_i,\hat X_j],\dots,X_i,\dots,X_j,\dots)\\ 
&\pm \Sigma X_i \circ f(X_1,\dots,X_i,\dots).
\end{align*} 
For future reference,
we write $d_A = d' + d''$ corresponding to the first term and the
remaining terms above.

If $L$ is an ordinary (ungraded) Lie algebra, $d' = 0$ and $d''$ is (up to sign)
the differential used by Chevalley--Eilenberg and Koszul.

\begin{dfn}\label{D:cohdgl}
For a \dgL $L$, the {\bf cohomology} $H^*_{dg\ell }(L)$ is the
algebra $H(A(L)) = H({Hom}\, (C(L),\Q)$.
\end{dfn}
The adjointness of $L$ and $C$ will show that, given suitable 
finiteness conditions, 
$A(L(\HH))$ is a model for $\HH$, i.e., a (possibly non--minimal) 
$(SZ,d)$ resolution of $\HH$ as in Chapter \ref{models}.  Because of the shift in grading and the way degrees add, $L$ must be of finite
type and suitably bounded for $C(L)$ to be of finite type: $L^n =0$ for $n \leq 0$ or $n \geq 2.$ 
For example, if $\HH$ is simply connected and of finite type,
then $C(L(\HH))$ is of finite type.

\subsection{ Comparison of Der $ L$ and Der $ C(L)$}
\label{comparison}

We are interested in comparing perturbations of $A(L)$  with the corresponding
changes in $L(\HH)$.

\begin{dfn}\label{D:Lperturb} A {\bf perturbation} of $L(\HH)$ is a Lie
derivation $p$ of the same degree as $d$ such that $p$ increases
bracket length by at least 2 and $(d+p)^2 = 0$.
\end{dfn}

We are interested in derivations of $A(\pi)$ where $\pi$ is a dg Lie
algebra.  We will be using $L$ to denote ${Der}\,\pi.$  Although
somewhat unfamiliar, the dg Lie algebra of coderivations of
$C(\pi)$ turns out to be more susceptible of straightforward
analysis.  

The graded space of all graded coderivations of $C$ will be denoted 
${Der}\, C$. Later we will examine $A(\pi)$ directly, under suitable
finiteness conditions.

\begin{dfn}\label{D:semidir} For a dg Lie algebra $\pi$, 
the {\bf semidirect product}
$s\pi \, \sharp \, {Der}\, \pi$ is, as $\Q$-vector space,
 $s\pi \oplus \, {Der} \, \pi$. As a  graded Lie algebra, 
it has $s\pi$ as an
abelian sub--algebra and ${Der}\, \pi$ as a subalgebra which
acts on $s \pi$ by derivations via $[\phi,s \theta] =
(-1)^{\phi} s\phi(\theta)$. The differential $d_\sharp$ is given by
 $d_\sharp (s\theta) = -sd_\pi\theta \oplus
ad \, \theta$ for $\theta \in \pi$, $\phi \in \, {Der}\,
\pi$.
\end{dfn}
\begin{thm}\label{T:semi-iso}
  For any dg Lie algebra $\pi$ with $\pi_i = 0$ for
$i \leq 0$, there is a canonical map
$$
\rho : s \pi \, \sharp \, {Der}\, \pi \to \, {Der} \, C(\pi)
$$
of dg Lie algebras.  If $\pi$ is free as a graded Lie algebra, then  $\rho$ is a
homology isomorphism.  (If $\pi$ is free on more than one generator, then
 $s \pi \, \sharp \, {Der} \, \pi \to \, {Der}\, \pi/ad\, \pi$ 
is also a homology isomorphism.)\end{thm}

\begin{proof} Since $C(\pi)$ is cofree on $s\pi$, a coderivation of $C(\pi)$ is
determined by its projection onto $s\pi$.  Thus ${Der}\, C(\pi)$ is
isomorphic to ${Hom}\, (C(\pi),s\pi)$.  Moreover, this is an isomorphism of
dg $\Q$--modules precisely if ${Hom}\, (C(\pi),s\pi)$ is given the
Chevalley--Eilenberg differential (3.11), where $\pi$ is regarded as a
$\pi$--module by the adjoint action, i.e., $ad\,x : y \to [x,y]$.  We can define
$\rho$ via this identification.  The coderivation $\rho(sx)$ is determined by
projecting $C(\pi)$ onto $\Q$ (by the counit) and then mapping to $sx$, while
for $\theta \in \, {Der}\, \pi$, correspondingly $\rho(\theta)$ 
is determined by projecting
$C(\pi)$ onto $s\pi$ and then composing with $s\theta : s\pi \to s\pi$.  A
careful check shows  $\rho$ is a map of \dgLs.

To calculate $H(\rho)$, notice that $\rho \vert {Der}\, \pi \subset
{Hom}\, (s\pi,s\pi)$ and,  in fact,  lies in the kernel of part of the
differential. That is,  for $sh \in {Hom}\, (s\pi,s\pi)$,
$$
d''sh(s[x_1,x_2]) = s(h[x_1,x_2] - [x_1,h(x_2)]+ (-1)^{ x_1  x_2}[x_2,h(x_1)])
$$
 which is zero if and only if $h$ is a (graded)
derivation of $\pi$.  Thus, with regard to tensor degree, $H(\rho)$ is an
isomorphism in tensor degrees $0$ and $1$.  If $d_\pi = 0$, then $\pi$ being free
implies $H^{Lie} (\pi;\pi) = 0$ above tensor degree $1$ 
\cite{hilton-stammbach}.

For a general \dgL $\pi$, we use a spectral sequence comparison.  Filter
${Der}\, C(\pi)$ by internal degree, i.e., $\theta \in {Der}\, C(\pi)
\spprox {Hom}\, (C(\pi),s\pi)$ is of filtration $\leq q$ if $\text{proj}\,
\circ \, \theta (sx_1 \,\wedge \dots \wedge \, sx_n)$ is of $\text{deg}\, \leq
\Sigma \, \text{deg}\, x_i + q$.  Thus the associated graded coalgebra has $d_\pi$
equivalent to zero and $\rho$ induces an isomorphism of $E_1$ terms for $\pi$
free.  Since $E_1$ is concentrated in complementary (= tensor) degrees 0  and  1,
the homology isomorphism follows.

Finally, if $\pi$ is free on more than one generator, then the center of $\pi$
is $0$, so that the sub--\dgL\quad $s\pi \, \sharp \, ad\, \pi$\quad of\quad 
$s\pi \, \sharp \, Der \, \pi$ \quad has $0$ homology.  The exact sequence
$$
0 \to  s\pi \, \sharp \, \text{ad}\, \pi \to  s\pi \, \sharp \, {Der}\, \pi
\to  \, {Der}\, \pi/ \text{ad} \pi \to 0
$$
now yields $H(s\pi \, \sharp \, {Der} \, \pi) \spprox H({Der}\,
\pi/\text{ad}\, \pi)$.
\end{proof}
\subsection{ $A(L(\HH))$ and filtered models}
\label{filtered model}

We will be interested in $A(L(\HH))$ {\bf only} if $\HH$ is simply connected 
$(\HH =
\underset{i>1}{\oplus \HH^i} )$ and of finite type.  
The dgc $C(L(\HH))$ is then of finite type
and we can then describe $A(L(\HH))$ usefully without dualizing twice.  Earlier,
we described the cofree Lie coalgebra as the indecomposable quotient of the tensor
coalgebra.  The construction $A(L(\HH))$ for {\bf simply connected 
$\HH$ of finite
type} can be described as the free commutative algebra on ($s$ of) the free 
Lie coalgebra on $\HH$.  A typical element of 
$A(L(\HH))$ is then a sum of symmetric
tensors  $a_1\wedge \dots \wedge a_n$ where each $a_i$ is a sum of terms
$$
\Sigma(-1)^\sigma [sh_{\sigma(1)} \vert \dots \vert sh_{\sigma(k)}].
$$
The differential $d$ is a derivation generated by the bracket in $L(\HH)$ and
$$
m : [h_1 \vert \dots \vert h_k] \to \
\Sigma(-1)^i [h_1 \vert \dots \vert h_i h_{i+1}\vert \dots \vert h_k].
$$
The obvious algebra map $A(L(\HH)) \to \HH$ determined by $[h] \to h$ and
$$
[h_1 \vert \dots \vert h_k] \to  0, \qquad \text{for } k > 1
$$
is a resolution with $k-1$ as the {\bf resolution} degree.  The total degree of
$[h_1 \vert \dots \vert h_n]$ is $\Sigma \, deg \, h_i - k+1$; thus the {\bf
weight} is $\Sigma \, deg \, h_i$.

The resolution $A(L(\HH)) \to \HH$ is the Tate--Jozefiak resolution if and only 
if $\HH$ has trivial products.  In general, the Tate--Jozefiak resolution is a
minimal model for $A(L(\HH))$, but more can be said because 
$A(L(\HH)) \to \HH$ is a
filtered model if we use the filtration by weight.

\begin{dfn}\label{D:filtmodel} A {\bf filtered model} $(S Z,d) \to
A$ of a dgca $A$ is a model with a dga filtration such that $E_1
(S Z,d) \spprox H(A)$ is concentrated in filtration $0$.
\end{dfn}
The comparison theorem of Halperin and Stasheff \cite{halperin-stasheff}
 generalizes directly 
to filtered models, given an equivalence $(S Z,d) \to A(L(\HH))$
with the Tate--Jozefiak model which respects filtration. In \S 6, we will
consider perturbations of $A(L(\HH))$ as an alternative
method of classifying homotopy types.  By using $A(L(\HH))$, the
problem can be further reduced to perturbations in ${Der}\,
L(\HH)$.  In particular, a perturbation which decreases weight by
$i$ will be represented in terms of maps of a subspace of
$\HH^{\otimes i+2}$ into $\HH$ of degree $-i$, which is suggestive of
a Massey product (as further explained in \S 8).

As one would hope, the classification does not depend on the model used.  As a
first step toward this independence, we compare Lie algebras of derivations.

\begin{thm}    Let $M \to A$ be the minimal model for a 
simply connected dgca $A$, free as a gca..  There is an induced map ${Der}\, M\,
\to \, {Der} \, A$  of dg Lie algebras which is a homology isomorphism (quasi-isomorphism).
\end{thm}

\begin{proof} According to Sullivan \cite {sullivan:inf} p. ???, 
$A$ splits as 
$M \otimes C$ as dgcas with $C$ contractible.  
The algebras $M,\  C$ and $A$ are free on differential graded vector spaces $X, Y$ and $X\oplus Y$ respectively with $Y$ contractible.
We have a sequence of maps
$Der B  =$ $$ Hom (X , M) \to  Hom(X, M\otimes C =A) \to  Hom(X, A) \oplus Hom(Y, A ) = Hom(X\oplus Y, A) = DerA.$$
Since $X, Y$ and $X\oplus Y$ are graded vector spaces, we can use the identity
$H( Hom (U,V)) \simeq Hom(HU,HV)$ to conclude that each of the above maps is a quasi-isomorphism.
\end{proof}
Thus the \dgL ${Der}\, A$ is a `` homology invariant'' of the free
simply connected dga $A$; its cohomology does not depend on the choice of $A$.
We may always choose a model $L = s \pi \, \sharp \, {Der} \, \pi$ for
${Der}\, A\quad (\pi = {\tilde L}H(A)$ with suitable
differential), which will have finite type if ${dim}\, H(A) <
\infty$.

Notice further that if $A$ is a filtered model, then the map
$(S Z,d) \to A$ in
the theorem preserves the filtration; in particular, this is true for the weight
filtration.

Let $W_A \subset \, {Der} \, A$ denote the sub--\dgL of weight decreasing
derivations.  Now compare $W_{(S Z,d)} \to W_A$ as above.  If  $\theta \in
\, {Der} \, (B \otimes C)$ is weight decreasing, then so is $\phi$, since
$d$ preserves weight.  Thus we have:

\begin{thm}   If $S Z \to A$ is the minimal model for a
filtered free dgca $A$, then the induced map $W_{S Z} \to W_A$ is a
homology isomorphism.
\end{thm}

Our classification of homotopy types will proceed via the classification of
perturbations in such a way that it will depend on only the homotopy type of
$W_A$ as a \dgL and thus can be analyzed in terms of the minimal model or
$A(L(\HH))$.  The advantage of the latter is that we can further reduce the
problem to perturbations of $L(\HH)$ with respect to the bracket length as weight.

Recall $L(\HH)$ is a free \dgL naturally filtered by bracket length.  The
complementary degree is the sum of the degrees in $\HH$, which generates the
weights in $A(L(\HH))$; we refer to this sum as the weight in $L(\HH)$ also.

Thus a perturbation of $L(\HH)$ generates one in $A(L(\HH))$ and indeed we have a
natural map $W_{L(\HH)} \to W_{A(L(\HH))}$ of \dgLs.

\begin{thm}   For simply connected $\HH$ of finite type, the natural
map $W_{L(\HH)} \to  W_{A(L(\HH))}$ is a homology isomorphism.
\end{thm}
\begin{proof} Under the given hypotheses on $\HH$,  the analog of Theorem \ref{T:semi-iso} implies there is
a homology isomorphism
$$
sL(\HH) \, \sharp \, {Der} \, L(\HH) \to  \, {Der} \, A(L(\HH)).
$$
The factor $sL(\HH)$ maps to derivations of $A(L(\HH))$ as follows:  For $x \in
L(\HH)$, the derivation $\theta_{sx}$ of $A(L(\HH))$ is defined as the partial
derivative with respect to $sx$.  Thus $\theta_{sx}$ decreases the bracket
length and hence increases weight.  Thus the weight decreasing elements on the
left hand side are precisely $W_{L(\HH)}$ and the proof of the analog of Theorem \ref{T:semi-iso}
restricts to the
sub--\dgLs $W$ of weight decreasing derivations, since the differentials in
$L(\HH)$ and $A(L(\HH))$ preserve weight.
\end{proof}
Thus we turn to the classification of homotopy types having a variety of models
to use.

\newpage
\vfill\eject

\section {Classifying maps of perturbations and homotopies:
The Main Homotopy Theorem.}
\label{main}

Motivated by the classification of fibrations (see \S 9) \cite{jds:bingh,tanre:fibrations,allaud}, we find we can
classify perturbations (and therefore homotopy types) by the \lq\lq path
components\rq\rq\, of a universal example.  Although we originally tried to use 
a universal dgca, we gradually came to the firm conviction that {\bf
cocomplete} dgccs are the real classifying objects; the
dual algebras work under suitable finiteness restrictions.

\begin{mainht}
\label{mainht1}  Let $\HH$ be a simply connected cga of
finite type and $(S Z,d) \to \HH$ a filtered model.  The set of augmented
homotopy types of dgca's \newline $(A,i : \HH \spprox H(A)$ is in
1--1 correspondence with the path components of $C(L)$ where $L
\subset\, {Der}\, SZ$ consists of the weight
decreasing derivations and $\hat C(L)$ is as in \ref{ }.
\end{mainht}


The rationals $\Q$ as a
coalgebra serve as a \lq\lq point\rq\rq\, and a \lq\lq path\rq\rq\,
is to be a special kind of homotopy of $\Q \to \hat C(L)$, but homotopy
of coalgebra maps is a subtle concept.  For motivation, we first
review homotopy of dgca maps.

\begin{dfn}\label{D:dgahmtpy} For dgca's $A$ and $B$ with $A$ free, two dga maps
$f,f_1 : A \to B$ are {\bf homotopic} if there is a dga map $A \to B[t,dt]$
such that $f_i$ is obtained by setting $t=i$, $dt=0$.
\end{dfn}
(For simply connected $A$, there is a completely equivalent definition 
\cite {sullivan:inf,bousfield-guggenheim}
in terms of dga maps $A^I \to  B$ where $A^I$ models the topological space of
paths $X^I$.)

{\bf Remark.}  \cite{bousfield-guggenheim}, p. 88:  Such a homotopy 
does $not$ imply there is a chain
homotopy $h : A\to  B$ such that $ h(xy) = h(x)\, f_1(y) \pm f(x)\,h(y)$ which
would be the  dga  version of \lq\lq homotopic through multiplicative maps
\rq\rq\,
but does imply that the induced maps of bar constructions

$$
{\Bbb B}f_i :{\Bbb B}A \to {\Bbb B}B
$$
are homotopic through coalgebra maps.

This is the notion of homotopy appropriate to specifying the uniqueness of
perturbations.
\begin{thm}  
\label{hs}  Compare \cite{halperin-stasheff} p. 253-4:  If  $\pi_i : (S Z,d+p_i) \to
(A,d_A)$ for $i=0,1$ are maps which induce $\rho ^*$ in bottom degree $0$ and
the $p_i$ decrease weight, then there is an isomorphism and dga map $\phi :
(S Z,d+p) \to ( S Z,d+p_1)$ such that $\phi - Id$ decreases weight
and $\pi_1\phi$ is homotopic to $\pi$.\end{thm}

\subsection{ Homotopy of coalgebra maps}
\label{homotopycoalg}
The algebra $I = S [t,dt]$ with $t$  of degree $0$  and $ dt$  of degree
$1$ is implicit in the above definition (4.2) of homotopy of dg algebra maps.
We regard $I$ as the dual of the coalgebra
$I^\dual $ with (additive) basis $\{t_i,t_i u\vert i = 0,1,\dots\}$, 
diagonal $\Delta t_n =
\underset {i+j=n}{\Sigma} t_i \otimes t_j$ and coderivation differential 
$d(t_nu) = (n+1)t_{n+1}$.  We denote $t$ by $t_1$ when convenient.


There are difficulties with defining homotopies of coalgebras as maps $C \otimes
I_* \to D$ because the end of the homotopy corresponds to the image of $\Sigma\  t_i$.

\begin{dfn}\label{D:J} The completion of $I_*$ with respect to the obvious
filtration we will denote by $J$ and refer to that completion 
as the {\bf unit interval
coalgebra}.  (We could call $t_i$ the $i$--th copower of $t = t_1$, and $J$ the
coalgebra of formal copower series.)  The symbol $\Sigma\ t_i$ does represent an
element of $J$.
\end{dfn}
\begin{dfn}\label{D:filthmtpy} 
Given two cocomplete dgcc's $C$ and $D$, two filtration
preserving maps $f_i : C \to  D$ are {\bf homotopic} if there is a filtration
preserving dgc map $h : C \hat \otimes J \to D$ such that $f(c) = h(c \otimes
1)$ and $f_1(c) = \Sigma\, h (c \otimes t_i)$.
\end{dfn}

In particular, the two \lq\lq endpoints\rq\rq\, $\Q \to J$ given by $1 \to 1$ 
and $1 \to  \Sigma\ t_i$ are homotopic.

A case of particular importance is that in which $D$ is $\hat C(L)$, the
cocompletion of $C(L),$ which can be regarded a a sub-coalgebra of $\hat T^c sL$ .

\begin{dfn}\label{D:classmap}
 Let $(L,d_L)$ be a \dgL and $p \in L_1$ such that $d_Lp +
\frac{1}{2}[p,p] = 0$.  The {\bf classifying map}
$$\chi (p): \Q \to  C(L)$$
\noindent is the dgc map determined by
$$1 \to sp \in sL,$$
$$i.e.,
\chi (p)(1) = 1 + \Sigma\ sp^{ \otimes n}  = 1 + sp + sp \otimes sp +
\cdots.$$
\end{dfn}
As $\hat C(L)$ is a sub-coalgebra of $\hat T^c sL,$ so  $\chi(p)$ is essentially exp($p$).
That $\chi(p)$ is a differential map is true precisely because $d_Lp + \frac{1}{2}[p,p] =
0$, which is the same as $d_Lp+p^2=0$.

We are concerned with the special case of dg homotopy between two 
classifying maps
\newline $\chi(p_i) : \Q \to \hat C(L)$.  

A {\bf homotopy} or {\bf
path} $\lambda : J \to \hat C(L)$ is determined by a linear
filtered homomorphism $\bar \lambda : J \to L$.  Denote the image of $t_n$
by $y_n$ and that of $t_nu$ by $z_n$.  By the correspondence
${Hom}_{filt}(J,L) \to L \hat \otimes J^*$, we can represent $\bar
\lambda$ by $\Sigma y_i t^i + dt \Sigma z_i t^i$.  That $\lambda$
is a dgc map is expressed by $d_L \bar \lambda + \bar \lambda d +
\frac{1}{2}[\bar \lambda,\bar \lambda] = 0$ which translates to
\begin{gather} 
d_L y_n + \frac{1}{2} \, \underset{i+j=n}{\Sigma} [y_i,y_j] = 0 \\
(n+1)y_{n+1} + d_L z_n + \underset{i+j=n}{\Sigma} [y_i,z_j] = 0.  
\end{gather}
Letting $\eta(t) = \Sigma y_n t^n$ and $\zeta(t) = \Sigma z_n t^n$ in $L[[t]]$, 
 we have the differential equations 
\begin{align} 
d_L\eta + \frac{1}{2}[\eta,\eta] = 0\\ 
\frac{d\eta}{dt} + d_L\zeta + [\eta,\zeta] = 0.
\end{align}
The first equation says that $\eta(t)$ is a perturbation, while the
second gives an action of $L^0[[t]]$ on the set of perturbations.
To check that if $\eta(0)$ satisfies the MC equation, so does $\eta(t)$ for all $t$,
proceed as follows:
Let $ u(t) = d_L\eta + \frac{1}{2}[\eta,\eta] .$
Then $$\frac{du}{dt} = dL\frac{d\eta}{dt} + [\frac{d\eta}{dt} , \eta]  = $$
$$d_L(-d_L\zeta - [\eta,\zeta]) -[d_L\zeta + [\eta,\zeta],\eta]=$$
$$-d_L [\eta,\zeta] -[d_L\zeta,\eta] - [[\eta,\zeta],\eta]=$$
$$-[\zeta,d_L\eta]- [[\eta,\zeta],\eta],$$
using $[\eta,\zeta] = -[\zeta,\eta].$
By the Jacobi relation.
$$ [\eta,\zeta],\eta] = -[[\eta,\eta],\zeta] + [\eta,[\zeta,\eta]],$$
so $ \frac{1}{2}[[\eta,\eta] ,\zeta] =  [\eta,\zeta],\eta].$
Thus
$$-[\zeta,d_L\eta]- [[\eta,\zeta],\eta] =- [\zeta, d_L\eta + \frac{1}{2}[\eta,\eta]].$$
In other words, $$\frac{du}{dt} = -[\zeta,u].$$ In fact, this gives
$$u(t) = u(0) exp (-[\zeta, \ ] t).$$
If $\eta(0)$ satisfies the MC equation, i.e.$ u(0) = d_L\eta(0) + \frac{1}{2}[\eta(0),\eta(0)] = 0,$
then $u(t)=0$ for all t (by uniqueness of solutions of ODE).

Later in \ref{Loo version} we will replace $L$ by its homology with respect to $d_L$ together with the
$L_\infty$-algebra structure transferred from the strict \dgL structure of $L$. The MC equation is correspondingly generalized to
$$dp + \sum \frac{1}{n!} [p,...,p]=0.$$ The second differential equation generalizes in a perhaps
less obvious way:
$$\frac{d\eta}{dt} + d_L\zeta + \sum \frac{1}{n!}  [\eta,..., \eta,\zeta] = 0,$$
where there are $n$ factors of $\eta.$
These differential equations correspond to equations in (formal) power series. (Although the word \emph{formal}
is of necessity used in two different ways in this paper: in the sense of homotopy type
and in the sense of power series;  hopefully the context will make it clear which is intended.)

We will use the second differential equation to prove the main result from a
homotopy point of view, but first we need to worry about the transition from the
formal theory using formal power series to the subtler results involving
convergence.  We are concerned with $L \subset \, {Der} \, L(\HH)$ or
${Der}\, A(L(\HH))$ consisting of the weight decreasing derivations. $L$ is
complete with respect to the weight filtration, so $C(L)$ is with respect to the
induced filtration and hence with respect to the $\otimes$--filtration.

\emph{IS THAT RIGHT? SO WE DON'T NEED TO COCOMPLETE??}

\begin{dfn}\label{D:filtmap}
  For a complete  \dgL $L$ with filtration $F_p$ and a
filtered dgcc $C$, a dgcc map $f : C \to C(L)$ is {\bf filtered} if $\pi f : C
\to  L$ is filtration preserving.  In particular, a homotopy $\lambda : J \to
C(L)$ being {\bf filtered} implies  $\Sigma \pi \lambda (t_nu)$ and $\Sigma
\pi\lambda (t_nu)$ are well--defined elements of $L^{1}$ and $L^0$
respectively.
\end{dfn}
\begin{dfn}\label{D:path} 
A {\bf path component} of $C(L)$ for a complete \dgL $L$ is a
filtered homotopy class of points: $\Q \to C(L)$.
\end{dfn}
{\bf Remark}.  The Main Homotopy Theorem could be rephrased entirely in terms of
\lq\lq classifying twisting cochains\rq\rq \, $\pi \chi (p): \Q
\to L$ and filtered homotopies thereof.  


Motivation from
topology, especially the generalization to the classification of
fibrations \S 9, are better served by staying in the category
of dgcc's.  

In light of Theorem \ref{hs},  we have that the Main Homotopy Theorem is equivalent to:

\begin{mainhtp}
Two perturbations $p$ and $q$ represent the same augmented homotopy
type if and only if the classifying maps $\chi (p),\chi (q) : \Q
\to C(L)$ are homotopic as filtered maps of
coalgebras.
\end{mainhtp} 
Here $L$ is the weight decreasing
subalgebra of ${Der}\, L(\HH)$.  After proving the theorem, we
investigate in the following section 
the possibility of replacing $L$ by an equivalent dgL; some changes will be conceptually significant, others of
computational importance.

\subsection { Proof of the Main Homotopy Theorem} \label{S:modsp}

\label{proof main}

In this section, we use the basic deformation differential equations to provide
the proof of the hard part of Theorem 4.1: homotopy implies equivalence.  For
the easy part, observe \cite{halperin-stasheff} that $(S Z, d+p)$ 
and $(S Z,d+q)$
have the same augmented homotopy type provided there is an automorphism $\Q$ of
$S Z$ of the form: $Id$ plus ``terms which decrease weight''.  The equivalence
can be expressed directly in terms of $p$ and $q$ as elements of the Lie algebra
${Der}\,S Z$.

For any \dgL $(L=\oplus L_i,d)$, consider the adjoint action of $L$ on $L$,
i.e., $ad (x)(y) := [x,y]$.  We can define an action of the
universal enveloping algebra $UL$ on L making $L$ a $UL$
module:  for $u = x_1\cdots x_n \in UL$ with $x_i\in L$,
define $uy = [x_1, [x_2, [\cdots , [x_n, y]\cdots ]]$ for $y\in
L.$  Provided $L$ is complete with respect to the filtration $L
\supset [L,L] \supset [L,[L,L]] \supset \dots ,$  there is a
sensible meaning to $(exp\, x)y= \Sigma \frac{ x^n}{n!} y$ for
$x\in L, y\in L.$  In fact, $exp\, x$ acts as an automorphism
of $L$.  In particular, for the Lie algebra $L \subset \,
{Der} \, S Z$ of weight decreasing derivations, $L$ is
complete with respect to the weight filtration and hence with
respect to the action by $UL$.

For this $L$ of weight decreasing derivations, we can similarly
define $exp\, x$ as an automorphism of the algebra $S Z$ for
$x \in L.$ Given a dga map and automorphism $\phi: (S Z,
d+p) \to (S Z, d+q)$ of the form $Id + \text{``terms which
decrease weight''}$, let $b = log (\phi -Id) \in L$ so that 
$exp\, b = \Sigma \frac{b^n} {n!} = \phi.$  The equation 
$(d+q)\circ\phi = \phi \circ (d+p)$ can then be written 
$$ d+q = (exp b)\circ (d+p) \circ (exp b)^{-1} $$ 
which is the same as $$ d+q = (exp\, b)(d+p) $$ using the action of 
$UL$ on $L$.

If we set $\zeta (t) = b$ and $\eta (t) = (exp\, tb)(d+p) - d,$ the
differential equations (5) and (6) are satisfied and $\eta (1) = q,$ so
$\chi (p)$ and $\chi (q)$ are filtered homotopic.

Given a homotopy, it is much harder to find $\phi$ because of the
non-additivity of $exp$, but use of the differential equation viewpoint
will allow us to succeed.

A homotopy $\lambda: J \to \mathcal {C}(L)$ gives a solution $(\eta (t),
\zeta (t))$ of the differential equation.  We will use that solution
to solve the equivalence, i.e., find $\theta$ so that $d+q = 
(exp\, \theta)(d+p)$ where $q = \eta (1).$  We change point of view slightly
and look for $\eta$ given $\zeta.$  It is helpful first to write
$\mu (t) = d+\eta (t),$ so that equation (6) becomes
$$
d\mu(t)/dt+ [\mu (t), \zeta (t)] = 0.
$$
For the remainder of this section, $\mu (t)$ will have this meaning
with $\mu (0) = d+p$. We will solve this equation formally, i.e.,
by power series, and then remark where appropriate on
convergence. 
\begin{lm} If $\zeta (t) = z \in L$,
then $\mu (t) = (exp\, t\zeta ) \mu (0)$ is a solution of $d\mu(t)/dt+ [\mu (t),\zeta ] = 0$.  \end{lm} 
\begin{proof} $d\mu(t)/dt= \Sigma \frac {t^{n-1}}{(n-1)!}  (ad^n\, z)
\mu (0)$ \text{while} $[z,\mu (t)] = \Sigma
\frac{t^n}{n!}[z,(ad^n z)\mu (0)]$.  (Notice both begin
with the term $[z,\mu (0)].$)  \end{proof}


\noindent Similarly for $\zeta
= t^k \phi$ with $\phi \in L$, we have that $\mu (t) = (exp\,
\frac{t^{k+1}}{k+1} \, \phi)\mu (0)$ is a solution.  
We would
like to handle a general homotopy, i.e., a general $\zeta$, in a
similar manner.  We have to contend with the \lq\lq
additivity\rq\rq of homotopy and the non--additivity (via
Campbell--Hausdorf--Baker) of $exp$.  For this purpose, we write
$\mu(t) = u(t)\mu(0)$.  In the case just studied,
$u(t) = exp\, \frac{t^k}{k} \phi$ acts as an automorphism of
$L$ and  is associated to $\zeta \in L[[t]]$ such
that for any $\theta = d + \psi$ with $\psi \in L_1$, we have
$$\dot u \theta + [u\theta,\zeta ] = 0.$$ We wish to preserve these
attributes as we construct $\mu(t)$ for a more general $\zeta (t)$.

Consider the \lq\lq additivity\rq\rq of homotopy.  If we have pairs $(\mu
_1,\zeta _1)$ and  $(\mu _2,\zeta _2)$ which are solutions of the differential
equation
$$
\dot \mu + [\mu,\zeta] = 0        
$$
with $\mu _1(0) = d+p$ and $\mu _2(0) = \mu _1(1) = d+q$, we wish to find a
solution $(\mu _3,\zeta _3)$ such that $\mu _3 (0) = d+p$ 
and $\mu _3 (1) =  \mu _2 (1)$.  It is sufficient for our purposes to consider
\begin{align}
\mu _1 (t) &= u_1 (t) \mu _1 (0),\\
\mu _2 (t) &= u_2 (t) \mu _2 (0) \end{align}
where each $u_i(t)$ acts as an automorphism of $L[[t]]$ and,
for any $\theta = d + \psi$ with $\psi \in L_1[[t]]$, we have
$$
\dot u_i \theta + [u\theta,\zeta_i ] = 0.$$

\begin{lm}  The pair ($\mu _3,\zeta _2 + u_2\zeta _1)$
is a solution for $$ \mu _3 (t) = u_2 (t) u_1 (t) \mu _1 (0). $$
\end{lm}

\begin{proof} We compute
\begin{align*}
\dot \mu _3 &= \dot u_2 u_1\mu _1(0) + u_2\dot u_1\mu _1 (0)\\
&= -[u_2 u_1 \mu _1 (0),\zeta _2] - u_2 [u_1 \mu (0),\zeta _1]\\
&= -[u_2 u_1 \mu _1 (0),\zeta _2] - [u_2u_1 \mu (0),u_2\zeta _1]
\end{align*}  
as desired. \end{proof}

Thus transitivity of homotopy corresponds to a sort of crossed
additivity of the $\zeta$'s.


Since we have a solution for each $\zeta_k = t^k\phi,$ 
the lemma can be applied
inductively to show:
\begin{crl} For any $\zeta (t) = \Sigma z_kt^k$
and given $d+p$, there is a (unique) formal solution $\eta (t)$
with $d+\eta (t)$ of the form
$$
\cdots exp\, t^n\theta_n \cdots exp\, t\theta_1 (d+p), \ \
\theta_i\in L_1.
$$
Of course, the $\theta_i$ are in general {\it not} the $z_i$.
\end{crl}

Now we address the question of convergence of $\mu (t).$  When
$L$ is complete and $J \to \mathcal {C} (L)$ is filtered, we have
immediately that $\Sigma\ z_n$ is convergent. The deviation from
additivity shows $n\theta_n$ differs from $z_n$ by terms of
still more negative weight; the finite product therefore
converges also for $t \leq 1.$

Similarly, if $\phi_n = exp\, \phi_n,$ then the sequence of
automorphisms $\phi_n\cdots\phi_1$ converges to show that $d+q =
d+ \eta (1)$ is equivalent to $d+p,$ which completes Theorems \ref{mainht}, \ref{mainht1}.
\newpage
\vfill\eject

\def\congo{\equiv}
\section{Homotopy invariance of the space of homotopy types}\label{invariance}

One advantage  of the  homotopy  theoretic  point  of view is that  it suggests the  homotopy invariance 
of the space $M_L$   of augmented  homotopy types with respect to changes in the dg Lie algebras used.
We began with the filtered model $(SZ, d)$ but  could just  as well have used the filtered model $A(L(\HH)).$

Using the  models $(SZ, d)$  and  $A(L(\HH))$, we can  consider  perturbations of $(SZ, d)$  as
before, of $L(\HH)$  with respect  to bracket  length  or of $A(L(\HH))$ with respect  to the grading
induced from bracket length.  Having perturbed $ L(\HH) $ to $\bar L(\HH)$, it follows that  $A(\bar L(\HH))$
 is a perturbation of $A(L(H)). $  Since $(SZ, d) $ is minimal,  by Theorem  \ref{3.5} there  is an induced
map of $Der\ SZ $ into  $Der\ A(L(\HH)).$  In fact,  regarding $ (SZ, d) \to A(L(\HH))$ as a model, it
is not hard  to see this map preserves both  gradations,  so a perturbation of $SZ $ maps to a
perturbation of $A(L(\HH)).$

For dg Lie algebras, the \emph{weight} of a derivation  is the decrease in total  degree plus the
increase in the bracket  length. For  dgaÕs, the  \emph{weight } of a derivation is, as before, the increase in total degree plus the decrease in resolution degree. 

For any weighted dga $ A$ or dg Lie algebra $L$, let $W_A  \subset Der\ A$ (respectively $W_L  \subset Der\ L$) denote the subÐdg Lie algebra of weight decreasing (respectively increasing) derivations.   Let  $p \in W_{(SZ,d)} $  and $q \in W_{L(\HH)}$ have the same image in $W_{A(L(\HH))}$, then  the respective classifying maps give a commutative diagram:


$$\hskip10exC (W_{ (SZ,d)})$$
$$\nearrow\hskip5ex\downarrow	$$
$$\hskip5exQ \longrightarrow C (W_{ A(L(\HH) })$$
$$\searrow\hskip5ex\uparrow$$
$$\hskip10exC (W _{L(\HH)}).$$

At the  end of \ref{dgl}, 
we saw that  the  maps of $WÕ$s induced  homology isomorphisms.   Thus  the classifications  are  equivalent  at the  homological level.   In fact,  we get the  same space of augmented  homotopy types independent of the model used, as we now show in detail.

\begin{definition} \label{incorporated}  For  any  dg Lie algebra $(L, d_L)$,  we define the   {\bf incorporated}  dg Lie algebra
 $(L[d], ad\ d )$ by  adjoining  a single new generator  also called $d$ of degree one with the obvious relations:  $[d,d] = 0, 
  [d, \theta]= ad\ d_L(\theta)$.
  \end{definition}


For any dg Lie algebra $L$, we define the variety $V_L  \subset L[d]$ to be

$$
\{p \in L^{1} |(d + p)^2=0\}.
$$

The variety is in fact in $L$ where the defining equation  might more appropriately  be written
 $$
dp + 1/2[p, p] =0,\label{integrability}
$$
called, at various times, the deformation equation, the integrability equation, the Master equation and now most commonly the Maurer-Cartan equation.

If $L$ is complete with respect to the $L^0$Ðfiltration, the action

$$p \mapsto (exp\  b)(d+ p) - d$$

\noindent
makes sense in $L[d]$ for $p \in L^{1},  b \in L^0.$
We define the quotient  {bf space} $ M_L$   to be $V_L /exp L^0.$

For complete dg Lie algebras, we have the notion of filtered homotopy of classifying maps
and the Main Homotopy Theorem  holds in that generality.

\begin{theorem}
If $L$ is an $L^0$-complete dg Lie algebra , then $M_L$   is in one-to-one correspondence with the set of filtered homotopy classes of maps $Q \to \hat C (L).$
\end{theorem}

As for the homology  invariance  of $M_L$ , we have:

\begin{theorem}
 Suppose $f : K  \to  L$ is a map of decreasingly filtered dg Lie algebras which are complete and bounded above in each degree.  If $f$ induces

a monomorphism   $H ^2(K ) \to  H ^2(L)$

an isomorphism   $H ^1(K ) \to H^1(L)$	and

an epimorphism 	$H^ 0(K ) \to H^ 0(L),$

\noindent
then $f$induces a one-to-one correspondence  between $M_K$    and $M_L.$
\end{theorem}\label{1iso}

\begin{definition}\label{indegree1}  Such an  $f$ will be called an  \emph{ (homology) equivalence in degree 1}.  The de- creasingly filtered dg Lie algebras $K$ and $L$ will be said to be (homology) equivalent in degree 1.
\end{definition} 

By \emph{ bounded above in  each  degree}, we mean  there exists $N(i)$  such that $F^ n L_i  Õ=0$ for $n \geq N (i).$

The proof will, in fact, show that $f$ induces a homeomorphism  between the appropriate
quotient topologies on $M_K$   and $M_L $.

\begin{lemma}
 Let  $f :  K  \to L$  be a  map  of decreasingly  filtered  complexes,  complete and bounded from  above in each degree.  Suppose $i$ is an index such that $H^{i-1}(gr f )$ is injective and $H^i (gr f )$ is surjective,  then the same is true  respectively of $H^{i-1}(f )$ and $H^i (f ).$
 \end{lemma}

\begin{proof}    For the injectivity,  consider the induced map of cohomology sequences for $K$ and
$L$ respectively arising  from $0 \to F^{n+l}  \to  F^n  \to  gr^n F \to 0$.


Consider $x \in H^{i-1}(F^n )$ such that $H (f ) (x) = 0.$  A standard  diagram chase shows $x$ comes
from  $y  \in H^{i-1}(F^{n+l})$ and  $H (f )(y)  = 0.$ 	Starting  with $F^{N (i-1)}= K^{i-1},$ we have classes
$x_k   \in H^{i-1}(F^{ N(i-1)+k} )$ such that  $x_k  - x_{k+l}  = 0$ in $H^{i-1}(F^{ N(i-1)+k} ).$
Now for  $x \in H^{i-1}(K)$
such that  $H (f )(x) = 0,$ we have $$x=\sum_{k\geq 0}x_k  - x_{k+l},$$
which makes sense since the filtration  is complete.  On the other hand,  the right hand side has each term  0, so $x$ is 0 and $H^{i-1}(f )$ is injective.

Similarly, for $x \in H^i (K ) = H^i (F^{ N (i)}), $ 
there  is $y \in H^i (F^{N (i)})$ such that $ x - H (f )(y)$ comes from $H^i (F^{ N (i)+1}).$   By induction  then,
there are $y_k  \in H^i (F^{ N (i)+k })$ such that $x - \sum^{k-1}_{j=1}H(f)(y_j)$ comes from 
$H^i (F^{ N (i+1}), $ hence, in the
limit, $x =H (f )(\sum y_k ).$
 \end{proof}

The following lemma is familiar in the ungraded  case.

\begin{lemma}
Suppose $K$  is a complete dg Lie algebra .  If $\theta = d + p \in d+ K^1$  and $b \in K^0$
is of positive filtration  such that $ (exp\  b)\theta = \theta$, then $[b, \theta] =0.$
\end{lemma}
\begin{proof}
We will show that  $[b, \theta]$ has arbitrarily  high filtration  and hence is zero.  Suppose
$[b, \theta] = c_n \in F^nK^1.$ 
Then
\begin{align}
exp (b) \theta&= \theta + [b, \theta] +1/2[b, [b, \theta]] + \cdots\\
&=\theta + c_n + 1/2[b,c_n] + \cdots
\end{align}
\noindent
If $ (exp\  b)\theta = \theta$, then $c_n + 1/2[b,c_n] + \cdots =0$, but $b$ has positive filtration, so
$1/2[b,c_n]$ and the further terms come from $F^{n+1}$, hence for the sum to be zero, $c_n$ must come from
$F^{n+1}.$ Thus $[b,\theta]$ has arbitrarily high filtration and must be zero.
\end{proof}
\begin{proof} {\bf of Theorem \ref{1iso} }
The map $f:K\to L$ induces a map of the spaces of perturbations  
$f : V_K  \to V_L .$\newline

$f : V_K\to V_L$ is {\bf surjective}

Let $q\in L^1$  be a perturbation, i.e.  $(d + q)^2 = 0.$  Assume we have constructed $ p \in K^1$
and $b \in L^0$  such that $(d + p)^2 = a_{n+l}$  of filtration  $n + 1$ and $(exp\  b)(d+ f p) = d + q + c_n$   with
$c_n$   of filtration  $n \geq 1.$  Squaring the second equation gives

$$((exp\  b)(d + fp))^2 = [d, c_n ] + [q, c_n ] + c_n^ 2$$

\noindent while applying$ (exp\  b)^2f$  to the first shows this is also $(exp\  b)^2f a_{n+1} $.  Since $q$ is of filtration
$\geq 1$ and $n \geq 1$, this shows $[d, c_n ]$ is of filtration  $n + 1$,  i.e.,  $c_n$   is a cycle mod filtration  
$n + 1$
and hence there exists $r \in K^1$ of filtration  $\geq n + 1$ and $c \in L^0$  such that

$$(exp (b + c))(Õd+ f p + f r) \congo  d+ q$$
 
\noindent
modulo filtration  $n + 1$. Since $(d + q)^2 = 0$, we can further  choose $r$ so that $(d + p + r)^2$ is of filtration $n + 2$, completing the induction.  Thus, since $K^1$  and $L^0$  are complete, $f : V_K  \to V_L$ is onto.

$M_K\to M_L$   is {\bf  injective}
 
Suppose we have perturbations  $p, q \in K^1$  which are equal modulo $F^n K$
  and  such that
 $(exp\  b)(d + f p) \congo d + f q$ for $b \in L^0 $  of positive filtration.   Write  $q - p =a_n$   of filtration  $n$,
then in $L/F^{n+1} L$ we have
$$[b, d + f p] + f a_n ,$$
so the isomorphism  $H^1 (K, d+ p) \simeq H^1(L, d+ f p)$ implies there is a $c \in K^0$  such that

$$[c, d + p] \congo a_n  \text{ modulo} F^{n+1} .$$
\noindent
Now $(exp\  c)(Õ + p) \congo d+ q$ modulo $F^{n+1}.$

On the  other  hand, $ (exp\  b)(d + f p)  \congo d+ f q$ implies  $[b, d + f p]  = f a_n  + e_{n+1}$   with
$e_{n+} \in   F^{ n+1} L.$  Since $q$ is a perturbation, we have $[d+ p, a_n ] \congo 0$ modulo $F^{ n+2} ,$
 so $e_{n+1}$  is a $(d + f p)$-cycle modulo $F^{ n+2} $. Thus there is a $(d + p)$-cycle 
 $z_{n+1}  \in F^{ n+l} K /F^{ n+2}$  such that
$e_{n+1}  = f z + [g, d + f p]$ for some $g \in L^0 /F^{ n+2} $. Replace $b$ by $bÕ= b - f c - g$ and $c$  by
$ c'$  such that  $[c', d + p] \congo a_n  - z_{n+l}$  modulo $F^{ n+2} $.  Consider  $(exp\  b')(exp\  f c')(d + f p). $ Modulo $F ^{n+2},$
this is
\begin{align}
&d+ f p + [f c', d + f p] + [b', d + f p]\\
&\congo  d + f p + f a_n  Ñ f z_{n+1}  + f a_n  + e_{n+1}  Ñ f a_n  Ñ e_{n+1}  + f z_{n+1}\\
&\congo  d+ f p + f a_n\\
&\congo  d+ f q.
\end{align}

\noindent The induction  is complete.
\end{proof}
$M_L$   as a {\bf space}

As promised, we point out that if $K$  and $L$ are regarded as topological vector spaces with the topology given by the filtrations,  then $ f : K \to L$ is an open map and hence $V_K  \to V_L$  is an open surjection  or quotient map.  Thus the quotient $M_K   \to M_L$   is not only a bijection but a homeomorphism.
Notice  that  if $H^2(f)$  is mono  and  $H^1(f )$  only onto,  the  first  half of the  proof goes through, i.e., 
$M_K   \to M_L$   will still be onto.  In particular, let $ f : K \to L$ be defined by

$$K^i  = 0, 	i \leq 0,$$
$$K^1 \text{is a complement to\ } dL^0 \subset L^1    (i.e.,  L^1=dL^0 \oplus  K^1 ),$$
$$K^i  = L^i , 	i > 1.$$

\noindent Thus $H^i(f )$ is an isomorphism  for $ i > 0$, so $V_K  = M_K   \to M_L$  is onto.

Observe that  in classifying homotopy  types, we began with perturbations of the filtered model $(SZ, d)$  but  have  also considered  other  models such  as $A(L(\HH)).$ 	This  led us  to consider
$K \subset Der\ L(H)$  and $L \subset Der\ A(L(\HH)).$  The homology isomorphism $sL(H) \sharp Der\ L(\HH)  \to Der\ A(L(\HH))$ of Theorem  
\ref{T:semi-iso}  restricts to a homology isomorphism $K \to L$ since the weight decreasing derivations of $L(\HH)$ induce derivations of $A(L(\HH))$ decreasing weight by the same amount  and  $sL(\HH)$ corresponds to derivations  of  $A(L(\HH))$ which do not  decrease weight. Finally,  the natural  map $Der\ C (L(\HH)) \to Der\ A(L(\HH))$ is an  isomorphism provided $L(\HH)$ is of finite type  and concentrated in degrees $<0,$ 
as it is for $\HH$ simply connected  of  finite type. (Recall $ Der\ C (L(\HH))$ consists of coderivations.)

\vfill\eject
\def\Linf{$L_\infty$}
\def\eps{\epsilon}

\section{Control by $L_\infty$-algebras.}\label{Linfty}

An essential ingredient of our work is the combination of the
deformation theoretic aspect with a homotopy point of view.
 Indeed we adopted  the
philosophy, later promoted by Deligne \cite{deligne:gm} in response  to
Goldman and Millson \cite{ goldman-millson} (see \cite{gm:ihes} for a history of that development),
that any problem in deformation theory is ``controlled''  by
a differential graded Lie algebra, unique up to
quasi-isomorphism of \dgLs. Neither the variety $V$ nor the group $G$ are unique,
 but the quotient $M = V/G$ is (up to appropriate isomorphism).
 

Implicit in the use of quasi-isomorphisms, even for strict dg Lie algebras, is the fact that $L_\infty$-morphisms  respect the deformation and moduli space functors.

\subsection{Quasi-isomorphisms and homotopy inverses}


\begin{dfn}\label{quism} A morphism in a category of dg objects is a \emph{quasi-isomorphism} if it induces an isomorphism of the respective cohomologies \emph{as graded  objects}. For dg objects over a field,  quasi-isomorphisms are also known as
\emph{weak homotopy equivalences.}
\end{dfn}
Indeed, for dg vector spaces, a quasi-isomorphism always admits a homotopy inverse in the category of dg vector spaces, (i.e. a chain homotopy inverse). That is,
a morphism $f: A\to B$ is a homtopy equivalence means there exists a homotopy inverse, a morphism 
$g:B\to $A such that $fg$ is homotopic to $Id_B$ and $gf$ is homotopic to $Id_A .$
If $f$ respects  additional structure, such as that of a \dgL, the inverse $g$ need not; however, it will respect that structure \emph{up to homotopy} in a very strong sense, e.g. as an \Linf-morphism.
Thus, even when a controlling \dgL is at hand, comparison of relevant \dgLs is to be in terms of $L_\infty$-morphisms. Hence,  one should consider control more generally by $L_\infty$-algebras. In terms of a \dgL  $L$ as we have been doing, an attractive candidate is its homology 
$H(L),$ not as the  obvious \dgL with trivial $d$ but rather with the more subtle $L_\infty$-structure as transferred from $L$, e.g. via a Hodge decomposition of $L$ \cite{johannes:Liepert}. This is why we introduced $L_\infty$-algebras in our early drafts, although they had been implicit in Sullivan's models. Since such algebras are now well established in the literature, we recall just a few aspects of the theory. 
The following definition follows our cohomological convention; $d$ is of degree 1. The original definition was homological, $d$ of degree -1 and thus the operations are of degree $k-2$.  Special cases of $L_\infty$-algebras $L$ occur with names such as
\emph{Lie n-algebras}.
An important distinction exists according to bounds for $L$  from  above or below.
\begin{definition}
  An \Linf{\bf -algebra} is a graded vector space $L$ with a sequence
  $[x_1,\dotsc,x_k]$, $k>0$ of graded antisymmetric operations of
  degree $2-k$, 
  such that for each $n>0$, the
  $n$-Jacobi relation holds:
  \begin{equation*}
  \sum_{k=1}^n \sum_{\substack{ i_1<\dotsb<i_k ;
    j_1<\dotsb<j_{n-k} \\
    \{i_1,\dotsc,i_k\}\cup\{j_1,\dotsc,j_{n-k}\}=\{1,\dotsc,n\} }}
    (-1)^\eps \, [[x_{i_1},\dotsc,x_{i_k}],x_{j_1},\dotsc,x_{j_{n-k}}] = 0 .
  \end{equation*}
  Here, the sign $(-1)^\eps$ equals the product of the sign $(-1)^\pi$
  associated to the unshuffle as a permutation
  with the sign associated by the Koszul sign convention to the action
  of  the permutation.
\end{definition}

The operation $x\mapsto[x]$ makes the graded vector space $L$ into a
cochain complex, by the 1-Jacobi rule $[[x]]=0$. Because of the
special role played by the operation $[x]$, we denote it by $d$.
An \Linf-algebra with $[x_1,\dotsc,x_k]=0$ for $k>2$ is the same thing
as a dg Lie algebra.
Just as an ordinary Lie algebra can be captured by its Chevalley-Eilenberg chain complex, so too
for \Linf-algebras by shifting the degrees.
In terms of the graded symmetric operations
\begin{equation*}
  \ell_k(y_1,\dotsc,y_k) = (-1)^{\sum_{i=1}^k(k-i+1)|y_i|} \,
  s^{-1}[sy_1,\dotsc,sy_k]
\end{equation*}
of degree $1$ on the graded vector space $s^{-1}L$, the generalized Jacobi relations
simplify to become
\begin{equation*}
  \sum_{k=1}^n \sum_{\substack{ i_1<\dotsb<i_k , j_1<\dotsb<j_{n-k} \\
      \{i_1,\dotsc,i_k\}\cup\{j_1,\dotsc,j_{n-k}\}=\{1,\dotsc,n\} }}
  (-1)^{\tilde{\eps}} \,
  \{\{y_{i_1},\dotsc,y_{i_k}\},y_{j_1},\dotsc,y_{j_{n-k}}\} = 0 ,
\end{equation*}
where $(-1)^{\tilde{\eps}}$ is the sign associated by the Koszul sign
convention to the action of $\pi$ on the elements $(y_1,\dotsc,y_n)$ of
$s^{-1}L$.  
Moreover, the operations $\ell_k$ can be summed (formally or in the nilpotent case) to a single differential and coderivation on $C(L)$.

 A quasi-isomorphism of \Linf-algebras is an \Linf-morphism
which is a
quasi-isomorphism of the underlying cochain complexes.

There are at least two sign conventions for the definition, which are usable cf. \cite{lada-markl, ls, getzler:nilp},  just as there are for the definition of an $A_\infty$-algebra, for which there is a `geometric' explanation:
a classifying space $BG$ can be built from pieces $\Delta^n \times G^n$ or from $(I\times G)^n.$

\subsection{$L_\infty$-structure on $H(L)$}
Since the \dgL algebras that are manifestly controlling the deformations of interest are huge, it is helpful to have at hand the induced \Linf -structure on $H(L)$.  Even though  $H(L)$ inherits a strict \dgL structure (with $d$ = 0), there is in general a highly non-trivial \Linf-structure such that $L$ and  $H(L)$ are equivalent \emph{as \Linf algebras}.
This \emph{transfer of structure} began in the case of dg associative algebras with the work of Kadeishvili  \cite{kad:Hfs}.
The definitive treatment in the Lie case (which is more subtle) is due to Huebschmann \cite{johannes:Liepert}.

In terms of the  $C(L)$ definition, almost everything we have done in earlier sections  carries over. 
What changes most visibly are the form of the Maurer-Cartan equation and the differential equation governing the equivalence.

Once we had that $L_\infty$-structure on $H(Der {\cal A}),$ 
our `perturbations' (also known as `integrable elements') could be regarded as solutions of what is now called the $\infty$-MC equation:
$$ \sum 1/n! [y,...,y, y ] = 0.$$  Notice even though the differential on $H(L)$ is zero, we can still have a non-trivial $L_\infty$ structure, i.e. suitably compatible multibrackets
$$[\ ,\dots, \ ]: H^{\otimes n}\to H.$$

 Implicit (only)  was the fact that the gauge equivalence of
the original \dgL  transfers to an $\infty$-action of $H^0$ on solutions of an $\infty$-MC equation; even though  $H^0$ is a strict Lie algebra, it is the $\infty$-action that gives the gauge equivalence.  Some researchers have missed this point and tried to use just the strict Lie algebra
acting on $H^1.$ By the way,  it's not so obvious that the $\infty$-action gives an equivalence relation, as we shall see later. 
\vfill\eject


\section{The Miniversal Deformation}\label{alggeom}
\subsection{Introduction} 
In this chapter, our point of view moves from homotopy theory to algebraic geometry.
The moduli space of rational homotopy types with given  cohomology $H$  (assumed here  finite dimensional)
has the form $W/G$, where $W$ is a conical affine algebraic variety and $G$ is an algebraic ``gauge" group, or rather, groupoid,
acting on $W$. Later we will treat $W$ more precisely as a \emph{scheme}, that is, a functor from algebras to sets or topological spaces.
Unfortunately, $ W/G$ is rarely what is called a fine module space. First, $W/G$ is usually not a variety - e.g. the ubiquitous 
presence of non-closed orbits in $W$ prevents  the points of $W/G$ from being closed . Second, even when $W/G$ is a variety, 
there  may well be 
no total space $X \to W/G$  with rational homotopy space fibers. ($X$  would be a differential graded scheme, represented 
algebraically by either  an almost free dgca, free just as gca, over $R$, where $Spec\ {R} = W/G $ or by an almost free dgla over $ {R}$.) 
We can deal with these difficulties by first considering  the ``moduli functor" $M$ of augmented rational homotopy types. 

For an 
augmented algebra A , with pointed affine scheme $Spec\ A =  S$, this functor assigns to $A$ or $S$  the set of  families
$Y \to S$, together with an 
isomorphism of the  special fiber with the formal space $X_\HH$, all modulo gauge equivalence. We view $Y \to S$ as a 
"deformation" of $X_\HH$, parameterized by $S$. This functor is not in general representable, as that would require a ``universal" 
deformation  $X \to W$ or fine moduli space. However,  there is nevertheless a ``miniversal" deformation $X  \to W$ inducing every 
deformation $Y\to Spec\ S$  by a map $S \to W$. This map is not unique , but its tangent at the base point is. The (Zariski) 
tangent space to $W$ at its base point is $H^1(L)$, where $L$ is the dgla of nilpotent derivations of the minimal (algebra or Lie 
algebra) model of $\HH$, as introduced previously. In the ``unobstructed" case, $W$  is the affine space $H^1(L)$, 
but in general, $W$  will be a  subvariety of $H^1(L)$ 
defined by $m$ homogeneous polynomial functions of degree at least 2, where $m= dim\  H^2(L)$. The nilpotent Lie algebra 
$H^0(L)$ is the
tangent space to a unipotent groupoid  $U$ , such that $W/U$ is the (not fine ) moduli space of augmented rational homotopy types 
with cohomology $H$. The groupoid $G$ for unaugmented homotopy types is an extension of $Aut\ H$ by $U$.

The construction of  $X,  W$  and $U$ proceeds from the minimal model $L'$ of the dgla of negative weight derivations  of a minimal model of $H$.
 This $L'$ is an $L_\infty$ algebra 
which consists of $H( L)$, together with higher order brackets $[a,b,c] \in H^{(|a|+|b|+|c| - 1)}$ etc. of all orders $>$ 2 , and all the brackets together satisfy 
the usual generalized  Jacobi identities \cite{ls, getzler:nilp}. 


 Then $W = MC(L)$ is the set of $a \in H^1(L)$ which satisfy the Maurer-Cartan integrability condition 
$$1/2[a,a] + 1/3![a,a,a] + ... =0$$
 (the sum is finite ), $U$ is the groupoid attached to the  $L_\infty$ algebra $L'$, and $X$ is constructed from the derivation representation of $L'$.

We give precise definitions and basic details in the following sections.

\subsection{Varieties and schemes}

An affine  $k$ variety, embedded in the affine space $k^n$, is the set of $n$-tuples in $k^n$ which satisfy a collection of 
polynomial 
equations from $P:= k[ x_1 , ... ,x_n]$. By contrast, an affine $k$-scheme $X = Spec\ P/I$
is given by the choice  of an ideal $I\subset P$. 
Then , for each overfield $j$ of $k$, we have the affine $j$- variety $X(j)\subset  j^n$, of points in $j^n$ which satisfy the polynomials 
of $I$. Thus, the equations $f = 0$ and $f^2 = 0$ define the same varieties, but different schemes.
The ideal $I$ is not determined by the 
variety $X(k)$, or even by the collection $\{X(j)\}$. However, if,  for a variable $k$-algebra $A$, 
we let $X(A)\subset A^n$ denote the the set of 
points in $A^n$ which satisfy the polynomials in $I$, then the functor $A\mapsto X(A)$
determines the ideal $I$ or the scheme X. Indeed:
$$X(A) = Hom (P/I , A ).$$

An  affine scheme is called \emph{reduced} if its ideal $I$ equals its radical- that is, $P/I$ contains no nilpotents except $0.$
Over an 
algebraically closed field, there is a one-to-one correspondence between reduced schemes and varieties - the two notions 
collapse to one .

A point $p$ in the variety $X(k)$ is the same as an augmentation of the algebra $R = P/I$. For a point $p=a=(a_1,\cdots,a_n)$,
the augmentation $X_a$ is defined by evaluating at $a$ all polynomials in $P:= k[ x_1 , ... ,x_n]$. The vector space 
$$X_a(k[\epsilon]), \epsilon^2 = 0$$ of augmentation preserving points in $X(k[\epsilon])$ is called 
the \emph{Zariski  tangent space} to $X$ at $p=a=(a_1,\cdots,a_n)$ and $$X_a(k[\epsilon_n]), \epsilon_n^{n+1} = 0$$\noindent
is the set of $n$-th order jets to $X$ at $p$.

In addition to being a functor from algebras to sets, $X = Spec\  R$ is also a local ringed space. The (scheme) points  $p$ of $X$ are the 
prime ideals in $R$; the closed subsets of $X$ are the primes  containing $J$ for $J$ an ideal in R , so that  the basic open sets are the 
subsets of the form $X_f = \{ p:f  \notin p \}$ for $f\in R.$ The closed and basic open sets in this Zariski topology support subschemes 
$Spec\ (R/J)$ and $Spec\ (R[1/f])$ respectively  and the direct limit of the $R[1/f]$, for $ f \notin p$, is the local ring  $R_p$ of  $X$ at $p.$ 
If we let 
$f(p)$ denote the image of $f \in R/p \subset R_p/pR_p$, then the  condition $ f(p) \neq 0$ defines the open set $X_f$.

A general  scheme is defined to be a local ringed  space which is locally affine; morphisms of schemes are morphisms of 
local ringed spaces. We have $Hom (Spec\ A , Spec\  B )= Hom ( B,A).$  As  a scheme, $X$ determines a variety 
$X(k) = Hom (Spec\  k , X)$.

The basic example of a non-affine scheme is projective space $P^n$, for which $P^n(k) = k^{n+1} -\{ 0\}$ modulo the action of 
$G_m = k-\{0\}.$ 
Thus
$P^n = \cup\  Spec\  k[x_0 /x_i, ... , x_n/x_i]$
 where $i$ runs from $0$ to $n$. A quasi-projective (resp. quasi-affine) scheme is a locally 
closed subscheme of projective (resp. affine) space.

Finally, a differential graded scheme is the common generalization of scheme and rational homotopy space which results from 
replacing algebras with flat differential graded algebras over algebras .

\subsection{Versal deformations}
In any representation $M = V/G$ of a moduli space, or rather functor, $M,$ as a quotient of a scheme $V$ by a group scheme $G$ acting on $V$, we call $V$ a {\bf versal}  scheme for $M$. In practice,  $V$ should
be the base scheme  of a family 
$X\to V$ which is a  {\bf versal deformation} of a special fiber $X_0$.
Here are the precise definitions . Let $X_0$ be a rational homotopy space, $(R,d)$ a filtered model of $X_0$ and $B$ an augmented algebra. Denote the augmentation ideal of $B$ by $m_B$.  
\begin{dfn}
A {\bf deformation of} $X_0$, parameterized by $Spec\ B$, is a family $Spec\ C \to Spec\ B$ where $C= (R\otimes B , d + e)$ and 
 \begin{itemize}
\item a) $e  \in  Der^1( R , R \otimes  m_B)$ lowers filtration,
\item b) $(d + e)^2 = 0$, viewed as a derivation of $R \otimes  B$.
\item c) $Spec\  C$ corresponds to  the   rational homotopy space of Sullivan  and of Getzler \cite{sullivan:inf, getzler:nilp}, 
which is the geometric realization of the dgca $C$.
\end{itemize}
\end{dfn}
Let $M(B)$ be the set of such deformations, modulo isomorphism. A deformation $X\to V$ will be \emph{versal} if the induced map $V \to M$ of 
functors is surjective, with a little more called ``smoothness ". We consider functors $X$ of augmented algebras. For these we assume that $X(k)=\{0\}$ is one point.

\begin{dfn}A map $X\to Y$ of functors  (from augmented algebras to sets) is \emph{smooth} if for each surjection$B\to  B' ((m_B')^2 = 0)$
 of augmented algebras , with $(m_B')^2 = 0$, the induced map 
$$
          X ( B ) \to X ( B' ) \times _{Y (B')} Y(B) 
$$
is surjective.
\end{dfn}
 If we take $B'$ to be the ground field $k$, we find that $X ( B ) \to Y ( B )$ is surjective. 
 
 We pause to investigate the smoothness condition for maps  in the context of formal geometry, where we replace schemes by formal schemes, augmented algebras by complete local algebras, polynomials by formal power series.
 \begin{itemize}
\item a) A scheme is smooth  (over $Spec\ k$ ) exactly when it is non-singular.
\item b) A group scheme ( in characteristic 0 ) is smooth.
\item c) A principal bundle is smooth over its base.
\item d) If a group scheme $G$ acts on a scheme $V$, then the map $V \to V/G$ of functors is smooth.
\end{itemize} 
In the presence of a $G_m$ action (``weighting"), we do not need to pass to the formal setting. For example , a smooth conical scheme is a weighted affine space, i.e. the Spec
of a weighted polynomial algebra. Recall that a formal 
rational homotopy space $X_\HH$  has a $G_m$ action, so that deformation data attached to it will also .
\begin{dfn} A deformation $X\to V$ is \emph{versal} if the induced map of functors $V\to M$ is smooth.
\end{dfn} 
If $V\to M$ is smooth, we will see that $M = V/G$, for suitable $G$. 
The definition of deformation  of  $X_0$ suggests that we can apply a construction to the dgla $L = Der_{-}(R)$ (for  $R$ a filtered model of $X_0$) 
to get a formally versal deformation of $X_0$. Indeed, we can attatch to $L$ the \emph{Maurer-Cartan scheme}  $V = V_L$ whose points with values in $B$ are given by 
$$
                                        V ( B )  = \{ e\in L^1 \otimes m_B | [d,e] + 1/2[e ,e ] = 0\}
$$

We claim first that $V = Spec\ A$, where $A = H^0 ( A ( L^+ ) )$,
provided $L^1$ has finite dimension. In fact, if we take $A = H^0 (A(L^+))$  and  $B$ is any augmented algebra, $ Hom (A ( L ) , B ) = V (B )$  so that $$Hom (A , B ) = Hom ( A(L) , B) = V(B).$$ Thus $V = Spec\  A= H^0 (A(L^+)).$ (Here $Hom$ denotes augmentation preserving morphisms.)

Next , the equality $Hom(A(L),A) = V(A)$ gives us a tautological $e \in V(A)$ and thus a differential $d + e \in Der_{-}(R \otimes A).$ This in turn yields a deformation $X  = Spec\ C \to  V$ with $C =  (R \otimes A , d + e)$. From chapter \ref{ }, we see that the induced map of functors $V\to  M$ is surjective and that $V/G = M$, where $G = exp\  L^0$. By example d) above, $ V\to M$ is  formally smooth or $X\to  V$ is a formally versal deformation of $X_0$.
 If $X_0 = X_\HH$, then $X\to V$ is a (conical) versal deformation of $X_\HH$.
\def\x{\times}

The above (formally) versal deformation is certainly not unique. We may replace $X\to V$ by $X \x (X_0  \x S ) \to V \x S$, where $S$ is smooth, 
to  get another. We can also change the model $R$ of $X_0$, or replace $L = Der_{-}(R)$ by a model 
$L' \to Der_{-}(R)$. These last replacements 
provide  an $L'$  which is quasi-isomorphic to $L$, but as we will see in section \ref{miniversal}, the corresponding deformations are related by a smooth factor as above. 

Finally, we discuss infinitesimal criteria for versality. A functor $X$ from augmented algebras to sets has a tangent vector space 
$TX = X (k[\epsilon]/\epsilon^2)$. For $X   = Spec\ A$, we have $TX = (m_B/m_B^2)^*$, while for $X = V_L$, we have $TX = Z^1(L)$ and for $X = M_L , TX = H^1(L)$. There are also normal spaces: $NSpec\ A = (I/m_PI)^*$, where $A = P/I$ for  $P$ a polynomial algebra with $I$ included in $(m_P)^2$.
For $M=M_L$, we have  $NM = H^2(L).$

 For any deformation $X\to  V$, we have Kodaira-Spencer maps

                                                                  $$ t : TV \to TM$$ 
                                                                   $$n: NV\to  NM.$$
$X\to V$ is formally versal if and only if $t$ is surjective and $n$ is injective. A deformation is  \emph{miniversal}   if it is versal and $t$ is an isomorphism .

\subsection{The miniversal deformation}\label{miniversal}
We give here the construction of the conical miniversal deformation of the formal rational homotopy space $X_\HH.$
The construction of the formally miniversal deformation of a general $X_0$ is similar. Let $R$ be a weighted model of $\HH$ and $L\to  Der_{-}R$
 a model of the negative weight derivations of $R$. (We will see that the construction is independent of these choices.) If $L^1$ has finite dimension,
  we get a conical versal deformation $X\to V$ of $X_\HH$. There are two ways in which $V$  is deficient as an approximation to a fine moduli space. First, it is not a homotopy invariant of $L$ - different $L$'s do not yield isomorphic deformations. (The versal deformation is a homotopy invariant of $L^+$, 
  not $L$.) Second, the approximation can be improved - the tangent space to $V$ is $Z^1(L)$, whose dimension is greater than or equal  to that of $H^1(L)$ , with equality when the deformation is miniversal or $d|L^0 = 0$. To achieve the latter, we can replace $L$ with an $L_\infty$ minimal model $L'\to L$ with $d' = 0$. That is, $L'$ is $H(L)$ with higher order brackets $[ h_1,\cdots,h_n] \in H(L)$ adjoined as in \ref{Linfty}.  
  We denote the resulting free dgca by $A(L')$. The  Maurer-Cartan condition becomes 
                 $$mc(x) : = 1/2[x,x] + 1/3![x,x,x] + \cdots= 0.$$
The sum is finite when $H(L)$ is of finite type, hence nilpotent. This $mc$ condition then specifies the miniversal scheme $W$
  and the map $ L'\to Der_{-}(R)$ gives us the  miniversal deformation $X\to W$.
  
There are two ways to  construct the miniversal deformation, or minimal model of $L$, assuming the tangent soace $H^1(L)$ has fintie dimension. 
First we may ignore the Lie algebras and  appeal to the general construction of Schlessinger \cite{schlessinger}  and then invoke weight and nilpotence conditions to get a conical family. Or one can solve the mc equation in $L$ by succesive approximations to yield the minimal model (in degree 1) of $L$, or the minimal model (in degree 0) of $A(L)$. We give an outline of the latter approach, assuming  $H^1(L)$ is finite
dimensional.

Let $L = H(L ) \+ R \+ dR$  be a decomposition of the complex $L$. Take homogeneous elements $h_1 , \dots, h_p \in Z^1(L)$ which induce a basis of $H^1(L)$ and let $t^1 , \dots, t^p$ be dual coordinates. Set $x_1 = \sum h_i  t^i.$ Then we can find $<h_i,h_j> \in R^1$ such that , in \L, $$mc( x_1 ) = - 1/2 \sum d <h_i,h_j> t^i t^j + [ h_i ,h_j ] t^i t^j,$$
\noindent where  $[\  ,\  ]$ is the naive bracket  $H^1(L) \otimes H^1(L) \to H^2(L)$.

 Set $ x_2 =\sum <h_i , h_j > t^i t^j$. Then $$mc( x_1 + x_2) = 1/2 \sum [h_i , h_j] t^i t^j - 1/3! \sum d <h_i , h_j , h_k > t^i t^j t^k + [ h _i , h_ j , h_k ] t ^i t^j t^k.$$ 

CHECK CAREFULLY is that right now?

Here  $< h,k,l> \in R^1$ and $[h , k, l] \in H^2 $ 
if $h, k,  l\ in H^1$.

Set $x_3 = 1/3! \Sigma <h_i h_j , h_k> t^i t^j t^k$, so

$$mc (x_1 +  x_2 + x _3 ) =  1/2 \sum [h_i , h_j]t^i t^j  + \sum 1/3! [h _i, h_j, h_k ] t^i t^j  t^k  + O(4)$$.

 Continuing we get 
$$mc_L(x_1 + x_2 +\cdots) = mc_{H(L)}(x_1).$$

Notice that the right hand side above is zero in $H(L) \o B$, where 
$B = H^0(A(L'^+))$ and $W  = Spec\ B$ is  the base of the miniversal deformation. The $x_i$ give a map 
 transforming $B$ to $A = H^0(A(L^+))$.
A systematic and generalized treatment of the above construction is given in Huebschmann and Stasheff
\cite{jh-jds}.
We now compare miniversal and versal deformations. If $X\to Spec\ A $ is versal and $Y\to Spec\ B$ is miniversal, 
then, by definition of these terms, we get weighted algebra  maps $u :A\to B$ and  $v: B\to A$ respecting the total spaces $X$ and $Y$. 
 The composition $uv$ is the identity  on the cotangent space $m/m^2  ( m= m_B)$ 
 and is therefore surjective. 
 But a surjective endomorphism of a Noetherian ring is an automorphism. Hence we can change $v$ to get $uv = identity$. 
 We claim now that we have $A = B \tensor C$, where $C$ is a weighted polynomial algebra .
If we denote the cokernel of the tangential  injection $u^* : T(B) \to T(A)$  by $U$, we may let $C$ be the polynomial algebra on the
dual  positively weighted vector space $U^*$.


Thus $TC = U$,
and the map $A\to C/(C_+)^2$ induces the trivial deforamtion over $C/(C_+)^2$. If we extend this family to a trivial family over $C$, the map extends to a map $A \to C$ inducing the trivial family over $C$, by the smoothness property of versality for $A$.  This, together with $u : A \to B$, gives us a map 
$A\to B \tensor C$  which is an isomorphism on tangent spaces and respects total spaces. For the backwards map, we have $v : B\to A$ and
$C\to C/(C+)^2$  lifts to $C\to A$, since C is a polynomial algebra. Thus we have:

\begin{theorem} A versal weighted family is isomorphic to the product of a miniversal weighted family and a trivial weighted smooth family $X_\HH \times Spec\ C \to Spec\ C$ with $C$ a polynomial algebra.
\end{theorem}
In particular, two miniversal deformations are isomorphic, but not canonically.

The miniversal family $X \to W$ is formally versal for each of its fibers. At each point $p$ of the base, the deformation splits, formally, into the product of the miniversal deformation of the fiber and a trivial deformation over the orbit of $p \in W$.

Now we list sufficient conditions for  the finiteness 
we need. Given $H$ of finite type, let $\gg$ be the free  dgla on the suspension of $H^{+*}.$
 Then $R =A(\gg) $ models $\HH$ and  $L = Der_{-} (\gg)$ is a model of $Der_{-}(A(\gg))$
\begin{itemize}
\item                                    a) $L$ will have finite type if $\HH$ is simply connected of finite dimension. 
 \item b) $H(L)$ will have finite type if if $\HH$ is finitely generated in even degrees and the associated projective scheme $Proj\ H$ is smooth.
  (Under the assumptions on $\HH$, we will have $Spec\  \HH$ as the cone over $Proj\ \HH$.) 
 \item c) We conjecture that $H(L)$ will have finite type if $\HH$ is Koszul.
 \end{itemize}
Finally , we point out $W = Spec\ A$ is the cone over $U= Proj\ A$ and the family $X\to W$, i.e. $Spec\ C \to Spec\ A$, is the cone over the family 
$X' \to U$ where $X' = Proj\ C$. The projective family has the same fibers as $X \to W$ does, i.e. every space with cohomology $\HH$, except that 
$X_\HH$ is missing. But the base is now compact.
                  
\subsection{Gauge equivalence for nilpotent $L_\infty$-algebras}

ÊFor any nilpotent \Loo -algebra $K$, the variety $MC(K)$ of Maurer-Cartan elements in $K$ is the set of elements $x \in K^1$ 
satisfying the Maurer-Cartan condition 

ÊÊÊÊÊÊÊÊÊÊÊÊÊÊÊÊÊÊÊÊÊÊÊÊÊÊÊÊÊÊÊÊÊÊÊÊ$$ÊÊÊ[ x ] + 1/2 [ x,x ] + 1/3! [ x,x,x ] + \cdots = 0.$$

\noindent
Here $[ x ] = dx$, and the sum is finite by nilpotence, cf. \cite{getzler:nilp}.
The scheme $V_K$ underlying the varierty $MC(K)$ is defined by the equation $V_K(B) = MC(K \o m_B)$ or $V_K =  Spec\  H^0(A(K^+)).$ 
 If $L$ is the minimal model of Ê$K$, i.e.
 $ d = 0 \in L$, Êthen  the Maurer-Cartan scheme of $L$ isgiven by the equations
$$1/2[x,x] + 1/3![x,x,x] + \cdots = 0$$  for $x \in L^1.$
Then
$W=V_{L}$ is what is called the \emph{ miniversal Êscheme}  of $K$ - it is a homotopy invariant of $K$. 

In particular, let $K$ denote the tangent Lie algebra consisting of the negative weight derivations of either  an algebra or Lie algebra model of $\HH$. and let $L$ be a minimal \Loo  model ($d = 0$) of $K$. We assume that \L is concentrated in degrees 0, 1, 2 and has finite type. We then have the miniversal deformation $X \to W$ of $X_\HH$, where
$W = Spec \ A$ for $A=H^0(A(L^+)) $.
In case $L$ is a dgla (third and higher order brackets as 0), the equivalence relation is given by the action of the unipotent group $G=exp(L^0)$ on the pure quadratic variety $MC(L)$ associated to $W$.


Here we have the setting for the Deligne groupoid attached to $L$. (A groupoid is a small category such that all morphims are invertible. The objects here form $MC(L)$ and the morphisms are given by the action of $G$.)  Replacing $L$ by $L\o B$, we obtain the Deligne stack (functor from algebras to groupoids) and the quotient of $W$
by the stack is the \emph{moduli functor} $\mathcal M$.
For general $L$, this equivalence relation on $W$ will have the form $U \rightrightarrows W$ for a scheme $U\rightrightarrows W$, where both relative tangent spaces identify to $H^0(L),$ together with an associative ``Campbell-Hausdorff " law of composition $U\times_W  U \to U$ which guarantees the transitivity of $U$, c.f. Schlessinger   \cite{schlessinger}.


Thus we again have a stack in which the objects over $B$ are given by $W(B)$ and the morphisms by $U(B)$.
 If we choose $U$ minimal, the two smooth maps $U \rightrightarrows W$ each have relative tangent space $L^0 = H^0(L)$, which is a nilpotent Lie algebra, and $U$ will be unipotent. The replacement of $L$  by another model of the tangent algebra will result in an equivalent, but larger, stack  with unknown  tangent spaces.

A description of $U$ in terms of $L$  is given by Getzler \cite{getzler:nilp}. He starts with the simplicial set (actually, variety associated to a simplicial scheme) $MC_{\bullet}$.
whose $n$-simplices are given by  $MC_n = MC(L \o  \Omega _n)$, where $\Omega_n  = k [t_0,..., t_n, dt_0,..., dt_n ] /(t_0+ ... +t_n)$.  According to Fukaya, Oh, Ohta , and Ono , [FOOO],   we get an  equivalence relation $R = MC_1$ on $MC(L) = MC_0.$
which unfolds as follows : $x_0 \in MC(L$) is  equivalent to $x_1 \in MC(L)$  if and only if there is an $x = a(t) + b(t) dt \in L[ t,dt ]$ with $a(t) \in MC(L[t ] ), b(t) \in L^0[ t ]$ with $a(0) = x_0 , a(1) = x_1$ and  $da/dt = 1/2 [a, b] + 1/3! [a ,b ,b] + \cdots .$ But this variety $R$  is too large: it is not even a groupoid. It has the relative tangent space $L^0[ t ]$ instead of $L^0.$ Getzler obtains a groupoid by replacing  $MC_\bullet$ by a much smaller sub-simplicial set (scheme)  $\gamma _\bullet(L)$ given by the vanishing of the ``Dupont gauge "  $s: MC(L) _\bullet \to  
 MC _{\bullet -1}(L)$.  (This will restrict $b(t)$ above to be constant.) This $\gamma_\bullet$ is the nerve of a groupoid.  
 The stack $U$ is obtained by replacing $L$ by $L\o B$; it determines the moduli functor $M.$

\begin{theorem} The following stacks are equivalent:

a) The stack whose objects over $B$ are families $X\to Spec \ B$ having fiber cohomology $\HH$ and whose morphisms are isomorphisms of families.

b) The stack  $U$ associated to the the minimal model of the tangent algebra of $\HH$.
\end{theorem}

\begin{remark} i) The stack U determines  the moduli functor as $\mathcal {M}=W/U = \pi_0(U)$.
For $W=Spec\  A$,  the equivalence relation $R$ on $W(B)= Hom (A,B)$ is also given by homotopy of maps $A(L)\to B$.

                ii) If \L\  is replaced by a quasi-isomorphic dgla $K$, then $U$ is replaced by the stack determined by the action of $G=exp(K^0)$ on
                the Maurer-Cartan scheme  
                $V = Spec\  H^0(A(K^+))$. Here we have a simpler action, but the actors $G$ and $V$ have no direct description in terms of the tangent cohomology $H(K)= L.$ 
                \end{remark}

\subsection{Summary}
Let $\HH$ be a simply connected graded commutative algebra of finite type, $X_\HH$ 
 the formal rational homotopy space with cohomology $ \HH, \gg$ the free dgla generated by $s(\HH^+)^*$ with differential dual to the multiplication in $\HH$ ,
 i.e.  the Quillen free dgla model of $\HH$ or $X_\HH$, and $R = A(\gg)$ the corresponding free dgca model of $\HH$ or $X_\HH$. 
 Let $Der _{-} (\gg)$ be the dgla of  weight  decreasing derivations of $\gg$ ( which models $Der_{-}(R)$ and $L\to  Der_{-}(\gg)$ 
 the minimal ($d_L = 0$) $L_\infty$-model of both $Der_{-}(\gg)$ and $Der_{-}(R)$.  
 Let $W = Spec\ A$, for $A = H^0 (A(L^+))$, be the Maurer-Cartan scheme of $L$ and $e \in MC(Der_{-}(R) \o A)$ (i.e.$ (d+e )^2  = 0$)
 be the corresponding classifying derivation. Assume $dim\ L^i = dim\ H^i(Der_{-})$ is finite for $i = 0,1,2$. 

\begin{theorem}   \label{summarythm} There is a ``miniversal deformation" $X\to W$ of $X_\HH$ with the following properties:

a) $W  = Spec\ A$ is a conical affine scheme and $X = Spec\ (R\o A,d+e)$ is  a dg scheme,
 
b) Each fiber of $X\to W$ is a rational homotopy space with cohomology $\HH$  and every such space occurs as a fiber, up to isomorphism ,

c) The fiber over the  vertex 0 in $W$  is $X_\HH$,

d) No fiber $X_p$ for   $p \neq  0$ is isomorphic to $X_\HH$,

e) $W$ is defined by $m = dim\  L^2$ weighted homogeneous polynomial equations, without linear terms, in $n = dim\  L^1$ variables,
 
f) The equivalence relation on $W$ governing the duplication of the $X_p$, or the passage from $W$ to moduli, is given by a conical affine scheme 
$U$ with two smooth projections $U \rightrightarrows W$; the fiber dimensions are $p = dim L^0$,
In fact $W$ determines  the objects, $U$ the morphisms  in a unipotent groupoid, or rather stack, $U$. The  composition of morphisms is determined by a Campbell-Hausdorff map $c: U \times_W U \to U $,

g) At each point $p \in W$, the family $X\to W$ splits, formally, into the product of the miniversal deformation of the fiber $X_p$ and a trivial family 
$X_\HH \x S\to  S$ where the smooth base $S$ is the orbit of $p \in W$ under $U$.
\end{theorem}
The miniversal family $m: X\to  W$ satisfies certain mapping properties. Any family $n:Y \to T$ of rational  homotopy spaces with cohomology $\HH$, deforming $X_\HH$, is isomorphic to the pullback of $m$ by a map of $T\to W$.  If $n$ is versal, i.e. contains all homotopy types as fibers, then $n$ splits, formally, into the product of $m$ and a trivial family over a smooth base, as in g). (This applies in particular to the MC family attached to an \Loo-algebra quasi-isomorphic to L).

Item e) says that the Zariski tangent space to $W$ at $0$ is $L^1$. This and versality uniquely determine $m$ up to isomorphism. The relative tangent space to $U$ over W is $L^0$, as $W$ is the MC scheme attached to $L$ and  $U$ is the MC scheme attached to $L' = L[ t, dt ]$, defined by the conditions 
$u=u_1(t) + u_0(t) dt$ with $u_i\in L^i[t]$ with $u_0$ constant.

If $L$ is a dgla ( third and higher order brackets vanish), then $W$ is a pure quadratic cone  and $U$ is the action of the unipotent group $exp(L^0)$ on $W$.

If $L^2 = 0$, then $W$ is smooth, i.e. $W$ is the weighted affine space $L^1$. More generally, $W$ is smooth when all brackets of $r$ elements in $L^1$ vanish for $ r > 0$.

If $L^1 = 0$, then $W = {0} = M$; $\HH$ is intrinsically formal.
 
If $L^0 = 0$, then $U =0$ and $M= W.$

The Campbell-Hausdorff map $c$ is determined by the brackets $[x_1,\cdots, x_r, u,v] \in L^0$ with $ x_i \in L^1$ and $u,v \in L^0$. If these vanish for
 $r >0$, then $U$ degenerates into the action of the group  $exp(L^0)$.

If $X \to S$ is formally versal and the tangent map is an isomorphism, the family is\emph{formally miniversal}. 
Such a family is then unique up to (formal) isomorphism.
 Under suitable nilpotence conditions, below, the adjective "formal" may be dropped.

\vfill\eject

\def\L{L_\HH}
\section {Examples and computations}
\label{examples}

Although some of our results are of independent theoretical interest, we are
concerned primarily with reducing the problem of classification to manageable
computational proportions.  One advantage of the miniversal variety is that it
allows us to read off easy consequences for the classification from conditions
on $H(L)$.

For the remainder of this section, let $(S Z,d)$ be the minimal model for
a gca $\HH$ of finite type and let $L_{\HH} \subset {Der}\, L^c (\HH)$ be the
corresponding \dgL of weight decreasing derivations, which is appropriate for
classifying homotopy type.  The following theorems follow from \ref{summarythm} and following remarks.  

\begin{thm}    If $H^1(L_{\HH}) = 0$, then $\HH$ is intrinsically formal,
i.e., no perturbation of $(S Z,d)$ has a different homotopy type; 
$M_{\HH}$ is a point.
\end{thm}
\begin{thm}    If $H^0 (L_{\HH}) = 0$, then $M_{\HH}$ is the quotient of the
miniversal variety by ${Aut}\, \HH$.
\end{thm}
\begin{thm}    If $H^2 (L_{\HH}) = 0$, then the miniversal variety is $H^1(L_{\HH})$.
\end{thm}
\begin{thm}    If $L_{\HH}$ is formal in degree 1 (in the sense of \ref{in degree 1}),
then $M_{\HH}$ is the quotient of a pure quadratic variety by the group of outer
automorphisms of $(S Z,d)$ (cf. \S \ref{miniversal}).
\end{thm}
The following examples give very simple ways in which these conditions arise.
Let ${Der}^n_k$ denotes derivations which raise top(ological) degree
by $n$ and decreases the weight = top degree plus resolution degree by $k$. 
 For ${Der}^n_k L^c (\HH)$, this specializes as follows:
${Der}\, L^c (\HH)$ can be identified with ${Hom}\, (L^c(\HH) ,\HH)$ and hence
with a subspace of ${Hom}\, (T^c (\HH),\HH)$ where $T^c(\HH)$ is the tensor
coalgebra.  Then each $\theta_k \in
{Der}^n_k$ corresponds to an element of ${Hom}\, (\HH^{\otimes
k+p+1},\HH)$ which lowers weight by $k$.  In particular, $\theta_k$ of top 
degree 1 and weight $-k$ can
be identified with an element of ${Hom}\, (\bar \HH^{\otimes k+2},\HH)$ 
which lowers the internal $\HH$-degree
by $k$ (e.g., $d=m : \HH \otimes \HH \to \HH$ preserves degree).  
Thus examples of the theorems above arise
because of gaps in ${Der}^n_k L^c (\HH)$ for $n=0, 1, $ or $2.$

\subsection{Shallow spaces}
\label{shallow}

By a {\bf shallow space}, we mean one whose cohomological dimension is a small
multiple of its connectivity.  


\begin{thm}    \cite{felix:Bull.Soc.Math.France80, NeiMil, halperin-stasheff}
  If $\HH^i = 0$ for $i<n>1$ and $i \ge 3n-1$, 
then $\HH$ is intrinsically formal, i.e., $M_\HH$ consists of one point.
\end{thm}
From our point of view or many others, this is trivial.  We have $\bar
\HH^{\otimes k+2}= 0$ up to degree $(k+2)n$, so that ${Image}\ \theta_k$ lies
in degree at least $(k+2)n-k$ where $\HH$ is zero for $k \ge 1$.
A simple example is $\HH = H(S^n \lor S^n \lor S^{2n+1} )$ for $n>2$.
\begin{thm}    If $\HH^i = 0$ for $i<n$ and $i \ge 4n-2$, then the space
of homotopy types $M_\HH$ is $H^1(L)/{Aut}\, \HH$ \cite{felix:Bull.Soc.Math.France80,felix:diss}.
\end{thm}



\begin{proof} Now $L_\HH^1 = L^1_1$, i.e., $\theta_1$ may be non--zero but 
$\theta_k= 0$ for $k \ge 2$.  Similarly $L_\HH^2 = 0$.  Thus, $W = V_{L_\HH} = Z^1 (L)$.  
Consider $\L$ and its action.  The brackets 
 have image in
dimension at least $3n-2$, thus in computing $( exp \, \phi)(d + \theta)$ the
terms quadratic in $\phi$ lie in $\HH^i$ for $i \ge 4n-2$ and hence are zero.
Thus, $(exp\, \phi)(d + \theta)$ reduces to $(1 + ad\, \phi)(d + \theta)$.  The
mixed terms $[\phi,\theta]$ again lie in dimension at least $4n-2$ and are also
zero, so that $(exp\, \phi)(d + \theta)$ is just $d + \theta + [\phi,d]$.
Therefore, $W_{L_\HH} / exp\, \  L_\HH$ is just $H^1 (L_\HH)$.
\end{proof}

Here a simple example is $\HH = H(S^2 \lor S^2 \lor S^5)$ \cite{halperin-stasheff}
\S 6.6.  Let the generators be $x_2 ,y_2 , z_5$.  We have $L^2 = 0$ for the same
dimensional reasons, so $W_L = V_L = H^1 (L)$ which is $\Q^2$.  Finally,
${Aut}\, \HH = GL(2) \times GL(1)$ acts on $H^1 (L)$ so as to give two
orbits:  $(0,0)$ and the rest.  The space $M_\HH$ is 
$$\bigodot\  \cdot$$
meaning the non-Hausdorff two-point space with one open point and
one closed.  For later use, we will also want to represent this
as $\ \ \ \cdot \to \cdot\ \ \  $, meaning one orbit is a limit point of the
other.

If $\HH^i=0$ for $i<n$ and $i \ge 5n-2$, then $W = V_L = Z^1(L)$ still, but now
the action of $L$ may be quadratic and much more subtle.  We will return to
this shortly, but first let us consider the problem of invariants for
distinguishing homotopy types.

\subsection{ Cell structures and Massey products} 
\label{cells}
 We have mentioned that
${Der}\ L(\HH)$ can be identified with ${Hom}\,(\HH^*,L(\HH))$.  This
permits an interpretation in terms of attaching maps which is particularly
simple in case the formal space is a wedge of spheres $X = \bigvee  
S^{n_i}$.  The rational homotopy groups $\pi_*(\Omega X) \otimes
\Q$ are then isomorphic to $L(H(X))$ \cite{hilton:spheres}.  In terms of the
obvious basis for $\HH$, the restriction of a perturbation $\theta$
to $\HH_{n_i}$ can be described as iterated Whitehead products
which are the attaching maps for the cells $e^{n_i}$ in the
perturbed space.

In more detail, here is what's going on:  attaching a cell by an ordinary
Whitehead product $[S^p,S^q]$ means the cell carries the product cohomology
class.  Massey (and Uehara) \cite{massey-uehara,massey:mex} introduced Massey 
products in order to
detect cells attached by iterated Whitehead products such as $[S^p,[S^q,S^r]]$.
If we identify a perturbation $\theta_k$ with a homomorphism $\theta_k :
H^{\otimes k+2} \to H$, this suggestion of a $(k+2)$--fold Massey product can be
made more precise as follows:  Consider the term $\theta$ of least weight $k$ in
the perturbation.  By induction, we assume all $j$-fold Massey products are
identically zero for $3 \leq j < k+2$.  Now a $(k+2)$--fold Massey product would
be defined on a certain subset $M_{k+2} \subset H^{\otimes k+2}$, namely, the
kernel of $\Sigma (-1)^j(1\otimes \dots \otimes m \otimes \dots 1)$ which is to
say
$$\quad x_0\otimes \dots \otimes x_{k+1} \in M_{k+2}\qquad \text{ iff }
\overset {k}{\underset {j-0}{\Sigma}}(-1)^jx_0 \otimes \dots \otimes
x_jx_{j+1}\otimes \dots \otimes x_{k+1} = 0.
$$
We can then define $\langle x_0,\dots ,x_{k+1}\rangle$ as the coset of $\theta
(x_0\otimes \dots \otimes x_{k+1})$ in $H \ \ \text{modulo}\ \hfill
x_0H+Hx_{n+1}$. Moreover, if $\theta = [d,\phi]$ for
some $\phi \in L$, then $\langle x_0,\dots ,x_{k+1}\rangle$
will be the zero coset because 
\begin{align*} 
\quad \theta (x_0 \otimes\dots \otimes x_{k+1}) =& x_0\phi(x_1\otimes \dots
\otimes x_{k+1})\\ 
&\pm \Sigma (-1)^j\phi (x_0\otimes \dots \otimes x_jx_{j+1} \otimes \dots \otimes_{k+1})\\ 
&\pm \phi (x_0\otimes \dots \otimes x_k)x_{k+1}, 
\end{align*} 
the latter sum being zero on $M_{k+2}$.  Notice $\phi$ makes precise \lq\lq
uniform vanishing\rq\rq.

The first example of \lq\lq continuous moduli\rq\rq, i.e., of a one--parameter
family of homotopy types, was mentioned to us by John Morgan 
(cf. \cite{neis,felix:Bull.Soc.Math.France80}). Let $\HH
= H(S^3 \lor S^3 \lor S^{12})$, so that the attaching map $\alpha$ is in
$\pi_{11}(S^3 \lor S^3) \otimes \Q$ which is of dimension 6,
 while ${Aut}\,\HH = GL(2) \times GL(1)$ is of dimension 5.  Alternatively, the space of 5--fold
Massey products  $\HH^{\otimes 5} \to \HH$ is of dimension 6 and so distinguishes at
least a 1--parameter family.

The Massey product interpretation is particularly helpful when only one term
$\theta_k$ is involved.  All of the examples of Halperin and Stasheff can be
rephrased significantly in this form.  To do so, we use the following: 

 {\bf Notation}: Fix a basis for $\HH$.  For $x$ in that basis 
and $y \in L(\HH)$, 
denote by $y \partial x$ the derivation which takes $x$ to $y$ (think of
$\partial x$ as $\partial / \partial x$) and sends the complement
of $x$ in the basis to zero.

\subsection{ Moderately shallow spaces}
\label{moderately}

Returning to the range of $\HH^i = 0, i < n \text{ and } i \ge 4n-2$, consider
Example 6.5 of \cite{halperin-stasheff}, i.e., $\HH = H((S^2 \lor S^2) \times S^3) $ 
with generators
$x_1,x_2,x_3$.  Again $L^1$ is all of weight $-1$; any $\theta _1$ is a linear
combination:
\begin{align*}
\mu_1[x_1,[x_1,x_2]]\partial x_1x_3 \, &+ \, \mu_2[x_1,[x_1,x_2]]\partial x_2x_3\\
&+\sigma_1[x_2,[x_1,x_2]]\partial x_1x_3+\sigma_2[x_2,[x_1,x_2]]\partial x_2x_3.
\end{align*}
As for $L$, it has basis $[x_1,x_j]\partial x_3$ for $1 \leq i \leq j \leq 2$.
Computing $d_L$, it is easy to see that $H^1(L)$ has basis:
$[x_1,[x_1,x_2]]\partial x_1x_3 = -[x_2,[x_1,x_2]]\partial x_2x_3$.  The
action of ${Aut}\, \HH$ again gives two orbits:  $\mu_1 = \sigma_2 \text{ and
} \mu_1 \not= \sigma_2 $.  In terms of the spaces, we have respectively 
$(S^2\lor S^2) \times S^3$ and $S^2 \lor S^3 \lor S^2\lor S^3  \cup e^5 \cup e^5$ where one $e^5$ is
attached by the usual Whitehead product  and the other $e^5$  is attached by the usual Whitehead product 
plus a non--zero iterated Whitehead product.

Notice the individual Massey products $\langle x_1,x_1,x_2\rangle$ and $\langle
x_1,x_2,x_2\rangle$ are all zero modulo indeterminacy (i.e., $x_1\HH^3 + \HH^3x_2$),
but the classification of homotopy types reflects the \lq\lq uniform\rq\rq\,
behavior of all Massey products.  For example, changing the choice of bounding
cochain for $x_1x_2$ changes $\langle x_1,x_1,x_2\rangle$ by $x_1x_3$ and
simultaneously changes $\langle x_1,x_2,x_2\rangle$ by $x_2x_3$, accounting for
the dichotomy between $\mu_1 = \sigma_2$ and $\mu_1 \not= \sigma_2$.  The
language of Massey products is thus suggestive but rather imprecise for the
classification we seek.

Our machinery reveals that the superficially similar 
$\HH = H((S^3 \lor S^3)\times S^5)$ behaves quite differently.  There is only one basic element in
$L^1$, namely $\phi = [x_1,x_2]]\partial x_5$, with again
$[d,\phi] = [x_1,[x_1,x_2]]\partial x_1x_5 + [x_2,[x_1,x_2]]\partial x_2x_5
, \ \text{ so } V_L/ exp \, L \spprox \Q^3$.  If we choose as basic
\begin{align*}
\qquad p &= [x_1,[x_1,x_2]]\partial x_2x_5 \\
\qquad q &= [x_2,[x_1,x_2]]\partial x_1x_5 \\
\qquad r &= 1/2[x_1,[x_1,x_2]]\partial x_1x_5 - 1/2[x_2,[x_1,x_2]]\partial
x_2x_5 \\
\end{align*}
then $GL(2,\Q) = {Aut}\ \HH^3$ acts by the representation $sym\ 2$, the
second symmetric power, that is, as on the
space of quadratic forms in two variables.  Since  ${Aut}\,\HH^5 = \Q^*$
further identifies any form with its non--zero multiples, the rank and
discriminant (with values in $\Q^*/(\Q^*)2)$ are a complete set of invariants.
Thus there are countably many objects parameterized by $\{0\} \cup \Q/(\Q^*)^2$;
in more detail, we have $\Q^*/(\Q^*)^2 \to 0 \to 0$, meaning one zero is a limit
point of the other which is a limit point of each of the other points (orbits)
in $\Q^*/(\Q^*)2$.  Schematically we have

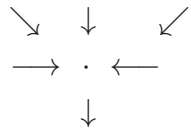
\begin{align*}
\searrow\quad&\downarrow \quad \quad\swarrow \\
\quad\quad\longrightarrow \ &\  \cdot\ \longleftarrow  \\
&\downarrow \\
\end{align*}

This can be seen most clearly by using ${Aut} H$ to choose a
representative of an orbit to have form 
\begin{align*}
 x_2 + dt_2, d \not= 0 \ &\text{(rank 2)} \quad \text{or}\\
 x_2 \ &\text{(rank 1)} \quad \text{or}\\ 
0 \ &\text{ (rank 0).} 
\end{align*}  
\subsection{ More moderately shallow spaces}
\label{more moderately}

Now consider $H^i = 0$ for $i < n$ and $i \ge 5n-2$.  We find $V_L = Z^1(L)$,
but there may be a non--trivial action of $L$.  Of course for this to happen,
we must have $H^i \not= 0$ for at least three values of $i$, e.g., $H(X)$ for
$X= S^3 \lor S^3 \lor S^5 \lor S^{10}$.  Spaces with this cohomology are of the
 form $S^3 \lor S^3 \lor S^5 \cup e^{10}$.  We have $V_L = Z^1(L) = L_1$ with
basis
\begin{align*}
[x_i,[x_j,x_5]]\partial x_{10}\ \qquad &\text{for}\ L_1^1,\\
[x_i,[x_j,[x_1,x_2]]]\partial x_{10}\ \qquad &\text{with}\ i \ge j \ \text{for}\
L_2^1.
\end{align*}
Thus $L_1^1$ corresponds to the space of bilinear forms (Massey products) on
$H^3$:
$$
\langle \ ,\ ,x_5 \rangle : H^3 \otimes H^3 \to H^{10} = \Q
$$
and thus decomposes into symmetric and antisymmetric parts.  On the other hand,
$L$ has basis $[x_1,x_2]\partial x_5$ and acts nontrivially on $L_1^1$ except
for the antisymmetric part spanned by
$$
[x_1,[x_2,x_5]]\partial x_{10}  - [x_2,[x_1,x_5]]\partial x_{10}  =
[[x_1,x_2],x_5]\partial x_{10}.
$$
(The $ exp\, \ L$ action corresponds to a one--parameter family of maps of the
bouquet to itself which are the identity in cohomology but map $S^5$
nontrivially into $S^3 \lor S^3$.)

Now $L_1^1$ is isomorphic over $SL(2,\Q) = {Aut}\, H^3$ to the space of
symmetric bilinear forms on $H^3$.  If we represent $L^1 = L_2^1 \otimes L_1^1$
as triples $(u,v,w)$ with $u,w$ symmetric and $v$ anti--symmetric, then
 $ exp \, L$ maps $(u,v,w)$ to $(u,v,tu+w)$.  We have ${Aut}\ H^5 = \Q^*$ and
${Aut}\ H^{10} = \Q^*$ acting independently on $L^1$.  If we look at the
open set in $L^1$ where $v \not= 0$, we find the discriminant of $u$ is a
modulus.  (In fact, even over the complex numbers, it is a nontrivial invariant
on the quotient which can be represented as the $SL(2,{\Bbb C})$--quotient of
$\{(u,\tilde w)\vert u \in \text{sym}^2, \tilde w \in P^2({\Bbb C}),
\text{discriminant}\, (\tilde w) = 0\}$.)  The rational decomposition of the degenerate
orbits proceeds as before.

%
%



\subsection{ Obstructions}
\label{obstructions}

We now  turn to examples in which  we do not
have $\theta^2_1 = 0$ automatically and $W_L$ may very well not be an
affine space.
One of the simplest examples of this phenomenon is
 $X = S^3 \lor S^3 \lor S^8\lor S^{13}$ (we must have $\mathcal{H}^i \neq 0$ for at least three values of $i$); 
 cf. \cite{felix:Bull.Soc.Math.France80,felix:diss}.  In particular, if 
$$
\theta_1 = [x_1,[x_1,x_2]]\partial y + [x_2,[x_1,y]]\partial z,
$$
then $[\theta_1,\theta_1] $ does not represent the zero class in $H^2(L) = L^2$.
Thus, $d + \theta_1$ can $not$ be extended to $ d + \theta_1 +
\theta_2 $ so that $(d + \theta_1 +\theta_2)^2 = 0$.  We refer to
$\theta_1$ as an {\bf obstructed} pre--perturbation or {\bf
obstructed infinitesimal} perturbation.

In terms of cells, this means we cannot attach both $e^8$ to realize $\langle
x_1,x_1,x_2 \rangle $ {\bf and} attach $e^{13}$ to realize $\langle x_2,x_1,x_8
\rangle $.

 The computations are essentially the same as those we now present for
  the miniversal scheme for homotopy types with cohomology 
  $$\mathcal{H} =H ( S^k \vee \cdots \vee S^k \vee S^{3k-1} \vee S^{5k -2})$$
 where  $k>1$ is an integer and there are $ r > 1\   k-$spheres. This is the simplest type of example with obstructed deformations of the formal homotopy type with cohomology \H; that is, the miniversal scheme $W$ is singular. In fact, $ W_{red}$ is a ``fan", the union of two linear varieties $A$ and $B$ meeting in the origin, corresponding to the two `quanta' making up $L^1.$  However, the scheme $W$ is  not reduced: If  $a_i$  and $b_j$  are coordinates in $A$ and  $B$, then each product $a_ib_j$  is nilpotent  modulo the Maurer-Cartan  ideal $I$, but few are in $I$.
 
                Let $x_1,\dots, x_r$  of degree $1- k$  be a basis of  the suspended dual  of  $\mathcal{H}^ k$ and let $y$ and $z$ play similar roles  for the other cohomology in \H.  The free graded Lie algebra  generated  by $x_i , y , z$ with $d=0$ is thus  the Quillen  model  $Q $ for \H\ and the sub-Lie algebra  of $Der\  Q$  consisting  of  negatively weighted derivations is the controlling Lie algebra $L$ for   augmented homotopy types with cohomology \H.  The subspace $L^1$ is the direct sum of  subspaces  $A$ and  $B$ where  $A$ has a Hall basis consisting   of  the derivations 
                $[x_i  ,[ x_j ,  y ]]\del z$
                (of weight -1) and the derivations  
                $[x_k , [ x_l ,  x_m ]] \del y , k \geq  l < m$  (of weight -1) form a basis for $B$.



                Thus $dim\ A = r^2$  and $ dim\ B = 2 \binom{r+1}{3}$ if  $k$ is odd . 
 From now on, for convenience, we assume $k$ is \emph{odd}.
              
                              The bracket of the indicated  basis  elements is $[x_i,[x_j,[x_k, [x_l,x_m]]]]\del z \in L^2.$
                 If these are expressed as linear combinations of a basis of $L^2$,   the transposed linear combinations exhibit generators for the Maurer-Cartan ideal $I$. 
                 

Let us analyze the situation in more detail.  We have:
$$
L^1 = L^1_1\ \text{with Hall basis}\ 
\begin{cases} \alpha_{ij} = [x_i,[x_j,y]]\partial z.\\
 \beta_{klm}=[x_k,[x_l,x_m]]\partial y
 \end{cases}
$$
with $k\geq l<m$.
All brackets in $[L^1,L^1]$ 
are zero except
$$
[\alpha_{ij},\beta_{klm}] = [x_i,[x_j,[x_k,[x_l,x_m]]]]\del z.
$$


\def\Gijklm{\gamma_{ijklm}}
\def\Gij|klm{\gamma_{ijklm}}
\noindent which we will denote $\gamma_{ijklm}.$
It is not hard to see that the bracket map from the tensor product of $A$ and $B$ to $L^2$ is surjective, expressing things in terms of a Hall basis.
The dimension of the space of quadratic generators $I_2$ of $I$ is the same  as $ dim\  L^2= dim\ F_5(r) \backsim r^5/5$. 
($F_5(r)$ denotes the space of fifth order brackets in the free Lie algebra on $r$ variables.)

 The word $\gamma_{ijklm}$ in Hall form is a \emph{simple} word in $L^2$ , as opposed to the \emph{compound} word  $[[x_i,x_j] , [x_k,[x_l,x_m]]] \del z $, which
 we denote
 $\gamma_{ij|klm}.$ 
 A basis of $L^2$ is given by $\gamma_{ijklm}$ together with $\gamma_{ij|klm}$ for $i\geq j, k\geq l < m$.
 

 
The duals of these Hall basis elements we will  denote $a_{ij}, b_{klm},  c_{ijklm}$ and $c_{ij|klm}$ and will also use the products  $a_{ij}b_{klm},$

 Finally, each of these symbols   has 
 a \emph{content} $s = (s_1,s_2, \dots ,s_r )$, where $s_i$ is the number of times $i$ occurs in the symbol  and has a \emph{partition}
  $ p  = (p_1 , p_2 \dots, p_r)$ obtained by rearranging the $s_i$ in descending order. Both these sequences 
add under bracketing or multiplying. We use a descending induction on the lexicographic order of the partition $p$ attached to $c_{ijklm}$ to prove that the latter  is nilpotent.
  A key point is that the square of a product $a_{ij}b_{klm}$ is divisible by a product of higher partition mod $I$.

     To illustrate these notions, we examine the situation when $r = 2$. We begin by expressing the brackets $[ \alpha_{ij}, \beta_{klm}] = \gamma_{ijklm}$ in terms of the Hall basis in
  $L^2$, using the Jacobi relation repeatedly. 
  
  
   In the following table for content (3,2), the columns are headed by  Hall basis elements of  $L^2$; reading down  yields a partial  basis for $I_2$ , consisting of quadrics 
which generate the Maurer-Cartan ideal $I$. Further tables for other content  (2,3), (4,1) and (1,4) yield a full set of generators for  the Maurer-Cartan ideal $I$ by the same procedure.

   We organize the tables according to content .
   
	 \begin{center}  
   		\begin{tabular}{| c | c | c |}
   \hline
(3,2)        &   	  21112            &           $12|112$\\ \hline                   
12112   &   	   1        &        1    \\ \hline   
21112         &          1         &   0 \\ \hline                                             
11212       &       1         &           1  \\ \hline 
		\end{tabular} 		
	\end{center}


 
 In the above table, the rows give the expression of $[\alpha_{ij},\beta_{klm}]$             in terms of the Hall basis of $L^2$.
 The columns then give the expansion of the dual basis elements as linear combinations of the products $a_{ij}b_{klm}.$
 Thus the first row says that $[\alpha_{12},\beta_{112}] = \gamma_{2112} + \gamma_{12|112}$, etc.
 By reading down the columns, the first column says that the dual $\gamma_{2112}^*$ is expressed as the quadric

 \begin{itemize}
 \item $(1) \quad a_{12}b_{112}+a_{21}b_{112} + a_{11}b_{212} = s_{12}b_{112}+ a_{11}b_{212}$
 \noindent where $s_{12}:= a_{12}+a_{21}$.
 \end{itemize}
           The second column gives a quadric
           \begin{itemize}
           \item $(2) \quad a_{12}b_{112}+ a_{11}b_{212}.$
           \end{itemize}
\noindent Similarly, for the other contents $(2,3), (4,1), (1,4),$ we get quadrics
\begin{itemize}
\item $(3)\quad s_{12}b_{212}+a_{22}b_{112}$
\item $(4) \quad a_{12}b_{212}$
\item $(5) \quad a_{11}b_{112}$
\item $(6) \quad a_{22}b_{212}.$
\end{itemize}

These six quadrics generate the Maurer Cartan Ideal $I$. Multiplying (2) by $a_{11}$, we get
$$a_{11}a_{11}b_{212} \equiv -a_{12} a_{11}b_{112} \quad \text{mod}\  I$$
$$\hskip8ex \equiv -a_{12} \cdot 0 \quad\text{mod}\  I.$$  
\noindent Then $(a_{11}b_{212})^2 \equiv 0,$ so $a_{11}b_{212}$ is nilpotent mod $I$.

Similarly $a_{22}b_{112}$ and $a_{21}b_{212}$ have square 0 mod $I$.

The other 3 quadrics, (4), (5), (6), are in $I$. (In addition to the above equations of the form $a^2b\equiv 0$ mod $I$, one also has 
a system of the form $ab^2\equiv 0$ mod $I$.)

We note that the above procedure shows each product $a^2b$ to be divisible by a product of higher partition  (4) or  (5),  This is essentially the induction step.

      To set up the induction step in the general setting, we outline what can be determined about the form of the quadrics in $I_2$. 
      To each Hall word $w$ of order 5 in the free Lie algebra on $x_1,\dots, x_r$, there is associated a  quadric $q_w$ , a linear combination  of the basic quadrics $a_{ij}b_{klm }$;
      the $q_w$ form a basis of $I_2$. What are the restrictions on the word $u = ijklm,\  k \geq l <m$ such  that $a_{ij}b_{klm }$  appear in $q_w$ with 
      non-zero coefficient ? 
      
        First, the content of $w$ and $u$ must be the same. It is not hard 
 to see that it suffices to work under the assumption that this content is $(1,1,1,1,1,0,0,\dots,0)$, so that  $w$ and $u$ are permutations  of 12345.
 (For, if not, we take a suitable order preserving function $f$ and transform  $q_w$  into $q_{f(w)}$.) 
 There are 4 simple Hall words, 20 compound Hall words  and 40 words $u = ijklm$ as above, 
 so one gets a $40 \times 24$ matrix for which it is straightforward to fill in the rows with entries 0,1, or -1 
 for the Hall decomposition of $u.$ 
 Only the top half  ($i<j$) matters; the bottom is the same except for the subtraction  of a 
$ 20\times 20$ identity matrix. Here are the essential features, which generalize readily to other situations.

      (1)  Suppose that $w$ is a simple Hall word, $w= pqrst $ with $p\geq  q\geq r\geq s<t$. Then the product $u=ijklm\   (k\geq l<m)$ appears in $q_w$ exactly when the contents $c(w) = c(u)$ and  $l =t$ or $m=t$. The same applies to $a_{ij}b_{klm}$. Thus $q_w$ is a sum of terms $s_{ij}b_{klm}$
      (where $s_{ij} = a_{ij} +a_{ji}$).
      
      (2)  Suppose instead that $w$ is a compound Hall word, $w= pq|rst $ with $p<  q, r\geq s<t$. Then the summands of $q_w$ are $a_{pq}b_{rst},
      s_{pr}v_{qst},   s_{rs}v_{pqt}$. Here $v_{xyz}$ stands for some linear combination of the two 3-letter Hall words spelled with $x,y,z$; e.g. $v_{123}$
      is a linear combination of $b_{213}$ and $b_{312}$.
      
      \begin{lemma} {\bf Induction} Let $a=a_{ij}$ and $a'=a_{i'j'}$ be coordinates in $A$ with different content so that $\{i,j\} \neq \{i',j'\} $, and
      $b$ and $b'$  are coordinates in $B$. Suppose $ab$ and $a'b'$ have the same content. Then one of $a'b$ or $ab'$ has higher partition than $ab$ does.
      \end{lemma}

      \begin{proof} By reordering the indices if necessary, we may assume that the partition of $ab$ equals its content. Suppose $c(a') > c(a).$ Then
      $c(a'b).=c(ab) = p(ab)$ so $p(a'b) \geq c(a'b')$ and hence $p(a'b) \geq c(a'b) > c(ab)=p(ab)$ as desired. The case $c(a')< c(a).$ is similar.
      \end{proof}

\begin{theorem}  Let $I$ be the Maurer Cartan Ideal for deformations of the homotopy type of  $X = (S^k)^{\lor r} \lor S^{3k-1} \lor S^{5k-2},\  k\ \text{odd}, r > 1.$ 
Then $L^1 = H^1(L) = A\oplus B,$ where $A$ and $B$ consist of
 unobstructed deformations and have coordinates $a_{ij}$  and $b_{klm}  ,  1\leq i,j,k,l,m \leq r, k\geq l<m$.  
 Further, each product $a_{ij}b_{klm}$ is nilpotent modulo  $I$.
\end{theorem} 

\begin{corollary} The miniversal variety $W_{red}$ for deformations of the homotopy type of $X$ decomposes  as   
 $W_{red}=A\lor B$ with an isolated singularity at the origin.
\end{corollary}


\begin{proof} To prove the theorem, we use descending induction on the partition $p(ab)$. 
When $p(ab)= (4,1,0,\dots,0)$ is as high as possible,
then $ab=a_{ii}b_{iij}$  or $a_{ii}b_{iji}=q_w$ (where  $w= iiiij$ or $w=iiiji$) is in $I$. 

Suppose the theorem true for all products $ab$ of partition $> p$ and consider $ab$ of partition $p$ and fixed content $c$.  Give $q=sb$ where, for example, $s=s_{ij} = a_{ij} + a_{ji}$, is symmetric,
then  there is a simple Hall word $w$
such that $q_w$ contains $q$ while for every other quadric $q=s'b'$ appearing in $w$ we have $c(s)\neq c(s')$.  By the Induction Lemma, $qq'$ is divisible by a quadric of higher partition; thus $qq'$ is nilpotent mod $I$. Since the square $q^2$ is minus the sum of such products mod $I$, we see 
$q^2$, and hence $q=sb,$ is nilpotent mod $I$.

To prove that a product $q=ab$ is nilpotent, we note that $sb=2ab$ when $a=a_{ii}$ is nilpotent,
so $ab$ is nilpontent mod $I$ in this case. Thus we may assume $a=a_{ij}$ with $i<j$ and $b=b_{klm}$.  Take $w$ to be the compound word
$w=ij|klm.$  By (2) above, $q_w=ab +$ terms of the form $s'b'$.  The latter are nilpotent mod $I$ so $ab$ is also.
\end{proof}

Now consider the moduli space associated to $X$. Since $L^1$ has weight $-1$ and $L^0$ has negative weight,
the action of the latter on the former is trivial. There remains to consider the action of $Aut\ \HH$ on $W$ 
where here $Aut\ \HH =GL(r)\times G_m \times G_m$, where $G_m=GL(1)$ . The number of continuous moduli for the action on $B$
 is $\binom{s+1}{2}$ when $r=2s$ is even.  We now analyze the entire moduli space $A/Aut\ \HH$ when $r=2.$

Let $V$ be the span of $x_1, x_2$ so that $B\cong V\otimes V$
 is a 4-plane on which $GL(2)$ acts as on bilinear forms on $V$ and the action of $G_m$ is by scalar multiplication. Taking normal forms for each point in the moduli space and letting $a\rightarrow b$ indicate that $b$ is in the closure of $a$, we have:

\begin{align*}
(x_1 &+ dx_2,1)\\r
\swarrow\quad&\downarrow \quad \quad\searrow \\
\quad (x_1 + dx_2,0)\to &(x_2,0) \leftarrow  (x_2,1)\\
&\downarrow \\
(0&,0)
\end{align*}

In $(x_1+d x_2,1)$, the continuous modulus is denoted by $d$, but in $(x_1+d x_2,0),$ it is determined only mod $G_m^2$.
i.e. the forms $(x_2 + dy_2,0)$ must be identified modulo squares: $d \sim \lambda^2d$.

The bilinear form may be recovered from $$<\quad ,\quad,\ H^{3k-1}>: H^k\otimes H^k\to  H^{5k-2}.$$

As for $B$, the order 3 part of the free Lie algebra on the $x_i$, the full analysis in terms of Geometric Invariant Theory
is unknown past $r=2$, where it is almost trivial: $\cdot\rightarrow\cdot$.

The obstructions can be interpreted in a very straightforward way, but with a
perhaps surprising result.  Any perturbation of $\HH$ corresponds to a rational
space of the form $S^k \lor S^k \cup e^{3k-1} \cup e^{13}$.  The obstructions tell us
that the deformations are either $S^k \lor S^k \lor S^{3k-1}  \cup e^{5k-2}$ or
$S^k \lor S^k \lor S^{5k-2}  \cup e^{3K-1}.$

For the extreme case $k=2$ of the previous example (i.e. $X= \lor^r S^2\lor S^5\lor S^8$),
we have a non-trivial action of $L^0$ on $L^1.$ Besides the previous subspaces 
of $L^1$, namely $A=\{x^3\partial y\}, B=\{x^2y\partial z\}$ of weight $-1$, 
we have an additional subspace $C=\{x^6\partial z\}$ of weight $-4$. 
For $L^0$, there are two subspaces $D=\{x^4\partial y\}$ and $E=\{x^3y\partial z\}$ of weight $-3$ and
$F=\{x^7\partial z\}$ of weight $-6$. We find that $W_{red}=C\times (A\lor B)$ with 
$[B,D] = C = [A,E]$  with all other brackets between $L^0$ and $L^1$ being 0.

\subsection{More complicated obstructions}

Here we present the simplest obstruction for a bouquet such that $W$ has an \emph{non-linear} irreducible component $V.$
In fact, $V$ will have equations of the form
$$u_pv_q - u_qv_p,\quad 1\leq p,q\leq c$$

\noindent
(and some matrix generalization of this). 
Thus $V$ is the cone over a ``Segre" manifold ${\mathbold P}^{c-1} \times {\mathbold P}^1$
(or generalization thereof).

The obstructions arise in the setting $L^1=A\oplus B$ again, where $A$ consists of
derivations of the form $\alpha = x^3y\partial z$ and $B$   consists of derivations of the form 
$\beta=x^3\partial y$. These are realized by a bouquet of spheres with $\HH = \HH^0+\HH^k+\HH^{3k-1}+\HH^{6k-2}$
where $x,y,z$ run though a basis  of the shifted dual of $\HH^+$. We outline the construction.

We decompose $A$ according to Hall type,  $i,j,k$ being the indices of the $x_i$s :
\begin{itemize}
\item $A_1: ijky\partial z, \quad i\geq j\geq k,$

\item $A_2: (ij)(ky)\partial z, \quad i < j,$

\item $A_3: yijk\partial z, \quad i\geq j < k.$
\end{itemize}

We also have 
$B: ijk\partial y,  \quad i\geq j < k.$  We set $c= dim\ A_3 = dim\ B$.

Let $p$ denote the simple word $ijk$ and $q$ denote $lmn$ and take
$$\alpha:= \sum_{p=ijk} a_p[y,p]\partial z \in A_3$$
\noindent
with $k\geq l<m$
and
$$\beta:=\sum b_qq\partial y\in B,$$
then $[\alpha,\beta] = \sum_{p<q} (a_pb_q-a_qb_p)[p,q]\partial z$.

As the $[p,q]\partial z$ are linearly independent in $L^2$, we conclude that the scheme ${\mathbold P}^{c-1} \times {\mathbold P}^1$
defined by the equations $a _pb_q-a_qb_p=0$ lies in $W_{red}.$
Proceeding as before, we find that  $W_{red}=V\cup L$ where $L$ is a linear variety meeting $V$ in another linear variety.
This result has some interesting generalizations if $dim\ \HH^{3k-1} = s>1,$ while $\HH^k$ and $\HH^{6k-2}$ have dimensions $r$ and $1$ as before.
We let $c$ denote the dimension of the the space of simple words $ijk,\ i\geq j <k$ in $r$ variables and proceed as above. We obtain a variety 
$V\subset W$ consisting of pairs of matrices $M,N$ such that $M$ is a $c\times s $ matrix, $N$ is an $ s\times c$ matrix and
$MN$ is symmetric.
$GL(s)$ acts on $V$ and the quotient is the variety of $c\times c$ matrices with $rank \leq s$.
We conjecture that $V$ is a component of $W$.



On the other hand, the obstructions above can be avoided by adding $S^{10}$ to
$S^3 \lor S^3 \lor S^8$ and then attaching $e^{13}$ so as to realize
$x_2x_{10}$.  Then the class of $[\theta_1,\theta_1]$ is zero in $H^2(L)$,
namely, $[\theta_1,\theta_1] = [d,\theta_2]$ for
$$
\theta_2 = [x_1,[x_1,[x_1,x_2]]]\partial x_{10}.
$$

\subsection{Other computations}

Clearly, further results demand computational perseverance and/or
machine implementation by symbolic manipulation and/or attention to spaces
of intrinsic interest. 

Tanr\' e \cite{tanre:stunted} has studied 
stunted infinite dimensional complex projective spaces 
${\Bbb C}P^{\infty}_n =
{\Bbb C}P^{\infty}/ {\Bbb C}P^n$.  Initial work on machine
implementation has been carried out by Umble  and has led, 
for example, to the classification of rational homotopy types 
$X$ having $H^*(X) = H^*({\Bbb C}P^n\vee {\Bbb C}P^{n+k})$ for $k$ 
in a range \cite{umble:CPwedge}.  At the next level of
complexity, he and Lupton \cite{lupton-umble} have classified rational homotopy
types $X$ having $H^*(X) = H^*({\Bbb C}P^n/ {\Bbb C}P^k)$ for all $n$ and $k$:

For further results, both computational and theoretical,
consult the extensive bibliography created by F\' elix 
building on an earlier one by Bartik.

\newpage
\section {Classification of rational fibre spaces}
\label{fibrations}

The construction of a rational homopy model for a classifying
space for fibrations with given fibre was sketched briefly by
Sullivan \cite {sullivan:inf}.  Our treatment, in which we pay particular
attention to the notion of equivalence of fibrations, is parallel
to our classification of homotopy types.  Indeed, the natural
generalization of the classification by path components of $C(L)$
provides a classification in terms of homotopy classes of maps
$[C,C(L)]$ of a dgc coalgebra into $C(L)$ of an
appropriate \dgL $L$.  However, the comparison with the topology
is more subtle; the appropriate $C(L)$ has terms in positive and negative degrees, because $L$ does,
unlike the chains on a space.

Because of the convenience of Sullivan's algebra models of a space and because
of the applications to classical algebra, we present this section largely in
terms of dgca's and in particular use $A(L)$ rather than $C(L)$ to classify. The
price of course is the need to keep track of finiteness conditions. 

Tanr\'e carried out the classification in terms of Quillen models \cite{tanre:modeles} (Corollaire VII.4. (4)).
with slightly more restrictive hypotheses in terms of connectivity.


\subsection{Algebraic model of a fibration}
The algebraic model of a fibration is a twisted tensor product.  For motivation,
consider topological fibrations, i.e., maps of spaces
$$
F \to E \overset{p}{ \rightarrow} B
$$
such that $p^{-1}(*) = F$ and $p$ satisfies the homotopy lifting property.  We
have not only the corresponding maps of dgca's
$$
A^*(B) \to A^*(E) \to A^*(F)
$$
but $A^*(E)$ is an $A^*(B)$--algebra and, assuming
 $A^*(B) \text{and} A^*(F)$ of finite type, there is an
 $A^*(B)$--derivation $D$ on $A^*(B) \otimes A^*(F)$ and an
 equivalence 
\begin{align*}
& A^*(E)\\
\nearrow \qquad &\qquad \qquad\searrow\\
A^*(B) \qquad &\qquad \qquad A^*(F)\\ 
\searrow \qquad &\qquad \qquad \nearrow\\
(A^*(B) &\otimes A^*(F),D)\,.
\end{align*}

To put  this it our algebraic setting, let $F$ and $B$ be dgca's
 (concentrated in non--negative degrees) with $B$ {\bf
 augmented}.

\begin{defn} 
A sequence
$$
B \to E \to F
$$
of dgca's such that $F$ is isomorphic to the quotient $E/\bar BE$ (where $\bar B$
is the kernel of the augmentation $B \to \Q$) is an  {\bf F fibration over B} if it is equivalent to 
one which as graded vector spaces is of the form
$$
B \overset{i}{\longrightarrow} B \otimes F \overset{p}{\longrightarrow} F
$$
with $i$ being the inclusion $b \to b \otimes 1$ and $p$ the projection induced
by the augmentation.  

.  Two such fibrations
$B \to E_i \to F$
are {\bf strongly equivalent} if there is a commutative diagram
$$
\begin{matrix}
&B\ \rightarrow &E_1\rightarrow &F\\
&id\downarrow &\downarrow &\downarrow id\\
&B\ \rightarrow &E_2\rightarrow &F
\end{matrix}
$$
(It follows by a Serre spectral sequence argument that $H(E_1) \cong H(E_2)$.)
\end{defn}

Both the algebra structure and the differential
may be twisted, but if we \emph{ assume
that $F$ is free as a cga}, then it follows that $E$ is
strongly equivalent to
$$
B \overset{i}{\longrightarrow} B \otimes F \overset{p}{\longrightarrow} F
$$
with the $\otimes$--algebra structure.  The differential in $E = B \otimes F$
then has the form
$$
d_\otimes + \tau,
$$
where $$
d_\otimes = d_B\otimes + 1\otimes d_F.
$$

The \lq\lq twisting term\rq\rq\, $\tau$ lies in ${Der}^1(F,\bar B\otimes F)$,
 the set of derivations of $F$ into
the $F$-module $\bar B \otimes F$.  This is the sub-\dgL of
${Der}(B\otimes F)$
 consisting of those derivations of $B\otimes F$
which vanish on $B$ and reduce to $0$ on $F$ via the
augmentation. 

Assuming $B$ is connected,
$\tau$ does not increase the $F$--degree so we regard $\tau$ as a
perturbation of $d_{\otimes}$ on $B \otimes F$ with respect to the
filtration by $F$ degree. The twisting term must satisfy the
integrability conditions:
	\begin{equation}\label{integrability}
 (d + \tau)^2 = 0\ \text{ or }\ [d,\tau] + \frac{1}{2}[\tau,\tau] = 0.  
\end{equation}
To obtain strong equivalence classes of fibrations, we must now factor out the
action of automorphisms $\theta\ \text{of}\ B \otimes F$ which are the identity
on $B$ and reduce to the identity on $F$ via augmentation.  Assuming $B$ is
connected, then $\theta - 1$ must take $F\ \text{to}\ \bar B \otimes F$ and
therefore lowers $F$ degree, so that $\phi = \text{log}\,\theta = \text{log}\,
(1+\theta -1)$ exists; thus $\theta =  exp \,(\phi)$ for $\phi$ in
${Der}^0(F,B \otimes F)$.  If we set $L = L(B,F) = {Der}\, (F,\bar B
\otimes F)$, then for $B$ connected, we may apply the considerations of \S \ref{invariance} to
the \dgL $L$, because the action of $L$ on $L^1$ is complete in the filtration
induced by $F$--degree.  The variety $V_L = \{\tau \in L^1 \vert (d+\tau)^2 =
0\}$ is defined with an action of $ exp \, L$ as before.  
\begin{thm}    For $B$ connected, $F$ free of finite type and $L =
{Der}\, (F,\bar B \otimes F)$, there is a one--to--one correspondence
between the points of the quotient  $M_L = V_L/ exp\, L$ 
and the strong equivalence classes of $F$ fibrations over $B$.
\end{thm}
Notice that if $F$ is of finite type, we may identify ${Der}\, (F,F \otimes
\bar B)$ with ${(Der}\ F) \hat \otimes \bar B$, i.e.,
$$
{Der}^k(F,F \otimes \bar B) \spprox\underset {n}{\Pi} {Der}^{k-n}(F)
\otimes \bar B^n.
$$
\begin{crl} If $H^1(L) = \Pi H^{1-n}({Der}\, F) \otimes
H^n(B)(n>0)$ is $0$, then every fibration is trivial.  If $H^2(L) = 0$, then
every \lq\lq infinitesimal fibration\rq\rq\, $[\tau] \in H^1(L)$ comes from an
actual fibration (i.e., there is $\tau^\prime \in [\tau]$ satisfying integrability, i.e. the Maurer-Cartan equation).
\end{crl}
We now proceed to simplify $L$ without changing $H^1(L)$, along the lines
suggested by our classification of  homtopy types.

First consider $F = (S Y,\delta)$, not necessarily minimal.  Combining
Theorems \ref{T:Cquism}  and \ref{T:semi-iso}, we can replace ${Der}\, F$ by $D = sLH\, \sharp \,
{Der}\ LH$ where $LH$ is the free Lie algebra on the positive homology of
$F$ provided with a suitable differential.  If $dim\, H(F)$ is finite, $D$
will have finite type, so we apply the $A$--construction to obtain $A(D)$.  Let
$[A(D),B]$ denote the set of augmented homotopy classes of dgca maps (cf.
Definition \ref{D:dgahmtpy}). 


\subsection{Classification theorem}\label{fiberclass}
 The proof of Theorem \ref{mainht} carries over to
\begin{thm}\label{A(D)classifies}    There is a canonical bijection between $M_{D \hat
\otimes \bar B}$ and $[A(D),B]$, that is, $A(D)$ classifies fibrations in the
homotopy category.
\end{thm}
However, $A(D)$ now has negative terms and can hardly serve as the model of a
space.  To reflect the topology more accurately, we first truncate $D$.

\begin{definition}  For a $Z$ graded complex $D$ (with differential of degree
+1), we define the $n^{th}$ truncation of $D$ to be the complex whose component
in degree $k$ is
\begin{align*}
D^k \cap\, ker\, d\ \quad &\text{if }\ k =n,\\
D^k\ \quad &\text{if }\ k<n,\\
0\ \quad &\text{if }\ k>n.
\end{align*}
We designate the truncations for $n=0$ and $n=-1$ by $D_c$ and $D_s$
respectively (connected and simply connected truncations).
\end{definition}
\begin{thm}  \label{ADc-classify}  Let $F$ be a free dgca and $D = sLH \, \sharp \,
{Der}\ LH$ as above.  If $B$ is a connected (respectively simply connected)
dgca, there is a one--to--one correspondence between classes of $F$ fibrations
over $B$ and augmented homotopy classes of dga maps $[A(D_c),B]$ (resp.
$[A(D_s),B]$).
\end{thm}
Thus $A(D_c)$ corresponds to a classifying space $B \ {Aut}\,
\mathcal F$ where ${Aut}\, \mathcal F$ is the topological monoid of 
self--homotopy equivalence of the space $\mathcal F$.  
Similarly, $A(D_s)$ corresponds to the
simply connected covering of this space, otherwise known as the classifying
space for the sub--monoid $S\, {Aut}\, F$ of homotopy equivalences
homotopic to the identity. 


\begin{proof} (Connected Case)  $D \to {Der}\, F$ is a cohomology
isomorphism. For $K = D_c \hat \otimes \bar B \text{ and } L =
{Der}\, F \hat \otimes \bar B$, it is easy to check that $K
\to L$ is a \emph{homotopy equivalence in degree one} in the sense of
\ref{in degree 1}, so that $M_K$ is homeomorphic to $M_L$.  By \ref{ADc-classify}, $[A(D_c),B]$ is
isomorphic to $M_K$, which corresponds to fibration classes by
\ref{A(D)classifies}.\end{proof} 
If we set $\mathfrak g =D_c/D_s$, so that $H(\mathfrak  g) = H^0(D)$,
then the exact sequence $$ 0 \to D_s \to D_c \to \mathfrak g \to 0$$
corresponds to a fibration $$ B\, S\, {Aut}\, F \to B\,
{Aut}\, F \to K(G,1) $$ where $G$ is the group of homotopy
classes of homotopy equivalences of $F$, otherwise known as
\lq\lq outer automorphisms\rq\rq \, \cite {sullivan:inf}.  
\begin{proof}(Simply
connected case) Since $A(\ )$ is free, the comparison of maps
$A(D_c) \to B$ and $A(D_s) \to B$ can be studied in terms of the
twisting cochains of the duals $D_c^* \to B$ and $D_s^* \to B$.
Since $D_c^*$ and $D_s^*$ differ only in degrees $0$ and $1$, if
$B$ is simply connected, the twisting cochains above are the
same, as are the homotopies.
\end{proof}
The $k^{th}$ rational
homotopy groups $(k> 1)$ of $A(D_s)$ and $A(D_c)$ are the same,
(namely, $H^{-k+1}({Der}\ F)$), but the cohomology groups
are not.

In the examples below, we will need the following:

**********************

REVISION OF 10/9/12

If $D_c$ denotes the connected (non positive) truncation of the Tangent Algebra of a formal space,  then we have 

\begin{theorem}  $H(A(D_c))$ is concentrated  in weight 0 .
\end{theorem}
\begin{proof} We take $D$ to be the outer derivations of the free Lie  algebra $LH = \mathcal{L} ( x_1 , ... x_n)$. Here $x_i$ has topological degree $t_i$,  i.e.  $t_i = 1-s_i$, where the $s_i$ are the degrees in a basis of $H$). In addition to  the topological degree 
$t$ in $LH$, we have a bracket or resolution degree $r$, and a weight $w = t - r$. In $D_c$, we have $ t \leq 0, r \geq 0, w \leq 0$.
The differential in  $D_c$, arising from the multiplication in  $H$, has degrees $1, 1, 0$ with respect to $t,r,w$. Thus  the differential in $D_c$ or its cochain algebra $A(D_c)$ preserves weight .

WELL SAID - THANKS

The key point is that the weight grading in $D_c$ (or $D$) is induced by bracketing with the derivation 
                            $$ \theta = \Sigma (t_i  - 1 ) x_i \partial x_i$$
which lies in the weight 0 part of $D_c$. That is, if $\phi$ has weight $w$, then $[\theta ,\phi] = w\phi$. Fuks in \cite{Cohomology of Infinite dimensionsl Lie algebras} refers to $w$   as an ``internal grading "  of $D_c$. It is easy to deduce the dgla version of Theorem 1.5.2 in that text and conclude that $H(A (D_c))$ is concentrated in weight 0, as desired.
\end{proof} 

Note that  in $A(D_c)$ we have $t \geq 1, r \leq 1, w = t-r \geq 0$, so that the weight 0 sub algebra of $A(D_c)$ is given by $t = 1 = r$. So the cohomology of $A(D_c)$ is the Eilenberg Maclane 

Chevalley-Eilenberg ??

cohomology of the weight 0  generating space of $D_c$ is given by the conditions $t = 0 = r$ in $D_c$ and is linearly spanned by the derivations $x_i \del x_j$ with $t_i = t_j$.

END OF REVISION

**********************

\subsection{Examples}
\begin{ex} Consider $\mathcal F = {\Bbb C}P^n$ 
and $F = S(x,y)$ with
$\vert x \vert = 2, \vert y \vert = 2n+1 \text{ and }\\
 dy = x^{n+1}$.  Since $F$
is free and finitely generated, we take $D = {Der}\ F$ and obtain
$$
D =  \{\theta^0,\theta^{-2},\phi^{-1},\phi^{-3},\dots ,\phi^{-2n-1}
\}$$
with indexing denoting degree and
\begin{align*}
\theta^0 &= 2x\partial x + (2n+2)y\partial y\\
\theta^{-2} &= \partial x\\
\phi^{-(2k+1)} &=x^{n-k} \partial y.
\end{align*}
The only nonzero differential is $d\theta^{-2} = \phi^{-1}$ and the nonzero
brackets are 
\begin{align*}
[\theta^0,\theta^{-2}] &= 2\theta^{-2}\\
 [\theta^0,\phi^{\nu}] \, \, &= (\nu -1)\phi^{\nu}\ \ \text{ and}\\
[\theta^{-2},\phi^{\nu}] = (n-k) \phi^{\nu -2}.  
\end{align*}\noindent
We then have the sub-\dgL $D_c =\{\theta^0,\phi^{-3},\dots\}$ which
yields
$$
A(D_c) \simeq S (v^1,w^4,w^6,\dots ,w^{2n+2}),
$$
the free algebra with $dv^1 = 0, dw^4 = \frac{1}{2}v^1w^4$, etc., and $v^1, w^4,
\dots $ dual to $\theta^0,\phi^{-3},\dots .$   The cohomology of this
 dgca is that of the subalgebra $S(v^1)$ (by the theorem above), which here is a
model of $BG = K(G,1)$ for $G = GL(1)$, the (discrete) group of homotopy classes
of homotopy equivalences of ${\Bbb C}P^n$.  (These automorphisms are represented
geometrically  by the endomorphisms $(z,\dots ,z_n) \mapsto (z^\lambda
,\dots ,z_n^\lambda )$ of the rationalization of ${\Bbb C}P^n$. this formula is not well defined
on ${\Bbb C}P^n$ unless $q$ is an integer, but does extend to the rationalization for all $q$ in $\Q^*$.
This follows from general principles, but may also be seen explicitly as follows. The rationalization is the inverse limit of ${\Bbb C}P^n_s$, indexed by the positive integers $s$, which are ordered by divisibility. The transition maps ${\Bbb C}P^n_s  \to {\Bbb C}P^n_t$  are given by $z_i \mapsto z_i ^{s/t}$. 
The $m$-th root of the sequence  $(x_s$) is then $( y_s)$, where $y_s = x_{ms}$.

 Algebraically  these grading automorphisms of
the formal dga $F$ are given by
$a \mapsto t^wa$ ($w$ = weight of $a$) for $a$ in $F$.
\end{ex}

Thus, the characteristic classes in  $HA(D_c)$ have detected only the fibrations
over $S^1$, not the remaining fibrations given by
$$
[A(D_c),S^{2k}] = H^{-2k-1}(D) = \{[\phi^{-2k-1}]
\}$$
\lq\lq dual\rq\rq\, to $w^{2k}$.

These other fibrations are, however, detected by
$$
H(A(D_s)) = S(w^4,w^6,\dots ,w^{2n+2}),
$$
since $D_s = \{\theta^{-2},\phi^{-1},\phi^{-3},\dots\}$ has the homotopy type
of
$\{\phi^{-3},\dots\}$.
These last fibrations come from standard ${\Bbb C}^{n+1}$ vector bundles over
$S^{2k}$, and the $w^{2i}$ correspond to Chern classes $c_i$ via the map
$BGL(n+1,{\Bbb C)}\to BS\, {Aut}\, {\Bbb C}P^n$.  (The fibration for $c_1$
is missing because, for $n=0$, the map $BGL(1,{\Bbb C)}\to BS\, {Aut}^*\,
\simeq * $ is trivial; a projectivized line bundle is trivial.


To look at some other examples, we use computational machinery and the notation of
Section \ref{examples}.

If the positive homology of $F$ is spanned by $x_1,\dots,x_r$ of degrees
$\nu_1,\dots,\nu_r, r > 1$, then ${Der}\, L(\HH)$ is spanned by symbols of the
form
$$
[x_1,[x_2,[\dots,x_{m+1}]\dots]\partial x_{m+2}
$$
of degree $\nu_{m+2} - (\nu_1 +\dots+\nu_{m+1}) + m$.

\begin{ex} If we take the fibre to be the bouquet $S^\nu
\vee   \dots \vee   S^\nu$ ($r$ times), then $\nu_i = \nu, d = 0$
in $LH$ and in $D$, and the weight $0$  part of $D_c$ has weight is
 $$ \mathfrak g = \{x_i \del x_j\} \simeq \mathfrak{gl}(r) $$ and 
  $$ H(A(D_c)) = H(A(g)) = S(v^1,v^3,\dots,v^{2n-1}) $$ (superscripts again indicate degrees) and
detects over $S^1$ the fibrations  (with fibre the bouquet) which
are obtained by twisting with an element of $GL(r)$. The model
$A(D_c)$ has homotopy groups in degree $p = m(\nu -1) + 1$,
spanned by  symbols as above (mod $ ad\, L(\HH))$.  Such a homotopy
group is generated by a map corresponding to a fibration over
$S^p$ with twisting term $$ \tau \in H^p(S^p) \otimes
H^{1-p}({Der}\ F) $$ which has weight $1-m$ in $H(S^p
\otimes F)$.  For $m>1$, this is negative and gives a
perturbation of the homotopy type of $S^p \times F$ (fixing the
cohomology); we thus have for $m>1$, a surjection from fibration
classes $F \to E \to S^p$ to homotopy types with cohomology
$H(S^p \otimes F)$, the kernel being given by the orbits of
$GL(r)$ acting on the set of fibration classes.  For $m=1, p =
\nu$, the twisting term gives a new graded algebra structure to
$H(S^p \otimes F)$ via the structure constants $a^k_{ij}$ which
give $x_ix_j = \Sigma a^k_{ij}yx_k$ where $y$ generates
$H^p(S^p)$.
\end{ex}
If we replace the base $S^\nu$ by $K(\Q,\nu)$ (for $\nu =2, {\Bbb C}P^2$ will
suffice), then the integrability condition $[\tau,\tau] = 0$ is no longer
automatic; it corresponds to the associativity condition on the $r$ dimensional
vector space $H^\nu(F)$ with multiplication given by structure constants
$a^k_{ij}$.  The cohomology of $B\, S\, {Aut}\, (S^\nu \vee  \dots \vee  
S^\nu)$ generated by degree $\nu$ is the
coordinate ring of the (miniversal) variety of associative
commutative unitary algebras of dimension $r+1$; that is, it is
isomorphic to the polynomial ring on the symbols $a^k_{ij}$
modulo the quadratic polynomials expressing the associativity
condition, and the $r$ linear polynomials arising from the action
of $ad\, x_i$ (\lq\lq translation of coordinates\rq\rq).

Apart from these low degree generators and relations, the cohomology of $A(D_s)$
remains to be determined.  For example, is it finitely generated as an algebra?
Already for the case of $S^2 \vee S^2$, there is, beside the above classes in
$H^2( A(D_s))$, an additional generator in $H^3$ dual to
$$
\theta = [x_1,[x_1,x_2]]\partial x_1 - [x_2[x_1,x_2]]\partial x_2 \in D^{-2}
$$
(which gives the nontrivial fibration $S^2 \vee S^2 \to E \to S^3$ considered
before).  Since $\theta \in [D_s,D_s]$, it yields a nonzero cohomology class.

In this last example, we saw that $H^*(E)$ need not be $H^*(B) \otimes H^*(F)$ as
an algebra.  Included in our classification are fibrations in which $H(E)$ is
not even additively isomorphic to $H(B) \otimes H(F)$.

\begin{example} Consider the case in which the fibre is $S^\nu$.  For $\nu$
odd, we have $F = S(x)$ and ${Der}\ F = S(x)\partial x$ with
$D_s = \{\partial x\}$.  The universal simply connected fibration is
$$
S^\nu \to E \to K(\Q,\nu +1)
$$
  or  
$$
S (x) \gets (S (x,u),dx = u) \gets S (u)
$$
with $E$ contractible.  Here $\tau = \partial x \otimes u$ and
the transgression is not zero.

By contrast, when $\nu$ is even, $F = S (x,y)$ with $dy = x^2$ and we get
$D_s = \{y\partial x,\partial x,\partial y\}$ which is homotopy equivalent to
$\{\partial y\}$.  The universal simply connected $S^2$--fibration is then
$$
S^\nu \to E \to K(\Q,2\nu)
$$
where $E = (S (x,y,u),dy=x^2 - u) \simeq S (x)$ is the model for
$K(\Q,\nu)$.  Here  $\tau = \partial y \otimes u$ gives a
\lq\lq deformation\rq\rq\, of the algebra $H(B) \otimes H(F) = S (x,u)/x^2$
to the algebra $H(E) = S (x,u)/(x^2-u)  =  S (x)$.
\end{example}

(Fibrations 
$$
\bigvee^n_1 S^{2(n+i)}  \to E \to K(\Q,2(n+2))
$$
\noindent occur in Tanr\' e's analysis \cite{tanre:stunted} of homotopy types 
related to the
stunted infinite dimensional complex projective spaces 
${\Bbb C}P^{\infty}_n$.) 

A neat way to keep track of these distinctions is to consider the
Eilenberg--Moore filtration of $B \otimes F$ where $F$ is a filtered model,
i.e., $\text{weight}\ (b \otimes f) = \text{degree}\ b \, + \, \text{weight}\ f
= \text{deg} \ b \, + \, \text{deg}\ f + \text{resolution degree}\ f$.
(Cf. \cite{thomas:fibrations}.)

In general, $\tau \in ({Der}\ F \hat\otimes B)1$ will have weight $\leq 1$
since weight $f \leq\ \text{degree}\ f$ and $\tau$ does not increase $F$ degree.
If, in fact, weight $\tau \leq 0$, then $H(E)$ is isomorphic to $H(B) \otimes
H(F)$ as $H(B)$--module but not necessarily as $H(B)$--algebra.

Finally, if we can accept dgca's with negative degrees (without truncating so as
to model a space), we can obtain a uniform description of fibrations and
perturbations of the homotopy type $F$.  Consider in ${Der}\, F$, the
sub-\dgL $D_-$ of negatively weighted derivations, then $[A(D_-),B]$ is for 
$B =S^0$, the space of homotopy types underlying $H(F)$ while for connected $B$, it
is the space of strong equivalence classes of $F$--fibrations over $B$.

\subsection{ Open questions}
\label{questions}

We turn now to the question of realizing a given quotient variety $M = V/G$ as
the set of fibrations with given fibre and base, or as the set of homotopy types
with given cohomology.  The structure of $M$ appears to be arbitrary, except
that $V$ must be conical (and for fibrations, $G$ must be pro-unipotent).  The
fibrations of an odd dimensional sphere $S^\nu$ over $B$ form an affine space $M
= V = H^{\nu +1}(B)$, though it is not clear how to make $V$ have general
singularities or pro-unipotent group action.

For homotopy types, we consider the following example, provoked by a letter from
Clarence  Wilkerson.  Take $F$ to be the model of $S^\nu \vee S^\nu$, 
for $\nu$ even.
As we have seen, the model $D$ of ${Der}\, F$ contains a derivation $\theta
= [x_1,[x_1,x_2]]\partial x_1 - [x_2,[x_1,x_2]]\partial x_2$ of degree $-2\nu
+2$ and weight $-2\nu$, which generates $H^{-2\nu +2}(D)$.  If $B = S^3 \times
(CP^\infty)^n$, then $H^{-2\nu +2}(D) \otimes H^{2\nu -1}(B) = V$ has weight
$-1$ and may be identified with the homogeneous polynomials of degree $\nu -2$
in $r$ variables.  If we truncate $B$ suitably, so that we have $H^1((D \otimes
B)_-) = V$, then the set of rational homotopy types with cohomology equal to
$H(F) \otimes H(B)$ is the quotient $V/GL(r)$, i.e., equivalence classes of
polynomials of (even) degree $\nu =2$ in $r$ variables.

We may ask, similarly, which dgca's occur, up to homotopy types, as classifying
algebras $A(D_c)$ or $A(D_s)$.  The general form of the representation problem
is the following: given a finite type of dgL $D$, does there exist a free \dgL
$\pi$ such that $D \sim {Der}\ \pi/ad \, \pi$?

\section{Postscript}
\label{postscript}
Some $n$ years after a preliminary version of this preprint first circulated,
there have been major developments of the general theory and significant applications, many inspried by the interaction with physics.
We have not tried to describe  them; a book would be more appropriate to address properly this active and rapidly evolving field.


\end{document}